\documentclass[11pt,reqno]{amsart}
\title{Independent sets in random subgraphs of the hypercube}
\date{\today}

\author{Gal Kronenberg}
\address{Gal Kronenberg\hfill\break\indent Mathematical Institute, University of Oxford, Andrew Wiles Building, Radcliffe Observatory\break\indent Quarter, Woodstock Road, Oxford, United Kingdom.}
\email{kronenberg@maths.ox.ac.uk}
\urladdr{https://www.maths.ox.ac.uk/people/gal.kronenberg}

\author{Yinon Spinka}
\address{Yinon Spinka\hfill\break\indent
	University of British Columbia,
	Department of Mathematics,
	Vancouver, BC V6T 1Z2, Canada\break\indent Tel Aviv University, School of Mathematical Sciences, Tel Aviv 6997801, Israel.}
\email{yinon@math.ubc.ca}
\urladdr{http://www.math.ubc.ca/~yinon}

\usepackage{amsmath, amsthm, amsopn, amssymb, enumerate, xcolor}
\usepackage[english]{babel}
\usepackage[colorlinks=true,urlcolor=blue,pdfborder={0 0 0}]{hyperref}
\hypersetup{linkcolor=[rgb]{0,0,0.6}}
\hypersetup{citecolor=[rgb]{0,0.6,0}}

\usepackage[nameinlink]{cleveref}
  \crefname{theorem}{Theorem}{Theorems}
  \crefname{thm}{Theorem}{Theorems}
  \crefname{mainthm}{Theorem}{Theorems}
  \crefname{lemma}{Lemma}{Lemmas}
  \crefname{lem}{Lemma}{Lemmas}
  \crefname{remark}{Remark}{Remarks}
  \crefname{prop}{Proposition}{Propositions}
  \crefname{defn}{Definition}{Definitions}
  \crefname{corollary}{Corollary}{Corollaries}
  \crefname{cor}{Corollary}{Corollaries}
  \crefname{section}{Section}{Sections}
  \crefname{figure}{Figure}{Figures}
  \crefname{cond}{Condition}{Conditions}
  \crefname{claim}{Claim}{Claims}

\newtheorem{thm}{Theorem}[section]
\newtheorem{lemma}[thm]{Lemma}
\newtheorem{prop}[thm]{Proposition}
\newtheorem{cor}[thm]{Corollary}
\newtheorem{claim}[thm]{Claim}
\theoremstyle{definition}
\newtheorem{remark}[thm]{Remark}

\newcommand{\N}{\mathbb{N}}

\newcommand{\Z}{\mathbb{Z}}
\newcommand{\1}{\mathbf{1}}
\newcommand{\E}{\mathbb{E}}
\newcommand{\Bin}{\textnormal{\ensuremath{\textrm{Bin}}}}
\newcommand{\Var}{\textnormal{\ensuremath{\textrm{Var}}}} 
\newcommand{\Cov}{\textnormal{\ensuremath{\textrm{Cov}}}}
\renewcommand{\Pr}{\mathbb{P}}
\newcommand{\cC}{\mathcal{C}}
\newcommand{\cD}{\mathcal{D}}
\newcommand{\cE}{\mathcal{E}}
\newcommand{\cF}{\mathcal{F}}
\newcommand{\cG}{\mathcal{G}}
\newcommand{\cI}{\mathcal{I}}
\newcommand{\cO}{\mathcal{O}}
\newcommand{\cM}{\mathcal{M}}
\newcommand{\cN}{\mathcal{N}}
\newcommand{\cS}{\mathcal{S}}
\newcommand{\cZ}{\mathcal{Z}}

\newcommand{\spn}{\textnormal{\ensuremath{\textrm{span}}}}

\DeclareMathOperator\dist{dist}

\newcommand{\bbb}[1]{{\left\vert\kern-0.25ex\left\vert\kern-0.25ex\left\vert #1 \right\vert\kern-0.25ex\right\vert\kern-0.25ex\right\vert}}

\begin{document}

\begin{abstract}
Let $Q_{d,p}$ be the random subgraph of the $d$-dimensional hypercube $\{0,1\}^d$, where each edge is retained independently with probability $p$. We study the asymptotic number of independent sets in $Q_{d,p}$ as $d \to \infty$ for a wide range of parameters $p$, including values of $p$ tending to zero as fast as $\frac{C\log d}{d^{1/3}}$, constant values of $p$, and values of $p$ tending to one. The results extend to the hardcore model on $Q_{d,p}$, and are obtained by studying the closely related antiferromagnetic Ising model on the hypercube, which can be viewed as a positive-temperature hardcore model on the hypercube. These results generalize previous results by Galvin, Jenssen and Perkins on the hard-core model on the hypercube, corresponding to the case $p=1$, which extended Korshunov and Sapozhenko's classical result on the asymptotic number of independent sets in the hypercube.
\end{abstract}

\maketitle

\section{Introduction and main results}\label{sec:intro}
%

%

The problem of computing the total number of independent sets (a set of vertices containing no edges) in a graph is known to be hard. This has been studied for various graphs, including the $d$-dimensional hypercube $Q_d$, where the problem of counting independent sets is particularly interesting due to its relation to the hardcore model from statistical mechanics. Let $i(G)$ denote the number of independent sets in a graph $G$. For a survey on counting independent sets in graphs, see~\cite{samotij2015counting}.

In the early 1980s, Korshunov and Sapozhenko~\cite{Korshunov1983Th} computed the asymptotic number of independent sets in the hypercube, showing that, as $d\to\infty$,
\begin{equation}\label{eq:KS}
i(Q_d) = (1+o(1)) \cdot 2\sqrt{e} \cdot 2^{2^{d-1}} .
\end{equation}
Sapozhenko~\cite{sapozhenko1987onthen} gave an additional proof of this shortly after (see~\cite{galvin2019independent} for an exposition).
This classical result was recently refined by Jenssen and Perkins~\cite{jenssen2020independent} who gave a formula and an algorithm for computing the asymptotics of $i(Q_d)$ to arbitrary order in $2^{-d}$, yielding for example that
\begin{equation}\label{eq:JP}
 i(Q_d) = 2\sqrt{e}\cdot 2^{2^{d-1}} \left(1 + \frac{3d^2-3d-2}{8\cdot 2^d} + O \left( d^42^{-2d}\right) \right) .
\end{equation}
See \cite[Theorem 1.1]{jenssen2020independent} for a more refined form, giving the asymptotics up to $O(d^62^{-3d})$. 

\medbreak

In this paper, we extend this to the number of independent sets in a random subgraph of the hypercube. Let $Q_{d,p}$ be the random subgraph of the hypercube $Q_d$ obtained by keeping each edge independently with probability $p$. The random graph $Q_{d,p}$ has been a subject of great interest; see for example~\cite{burtin1977probability,erdos1979evolution,ajtai1982largest,bollobas1983evolution,bollobas1990complete,bollobas1992evolution,kostochka1993radius,bollobas1994diameter,soshnikov2003largest,borgs2006random,kulich2012diameter,van2017hypercube,janson2018critical,hulshof2020slightly,mcdiarmid2021component,condon2021hamiltonicity,erde2021expansion}.

Our first main result is an extension of~\eqref{eq:JP} to the random graph $Q_{d,p}$. 

\begin{thm}\label{thm:expectation}
	For $p \ge \frac{C\log d}{d^{1/3}}$,
	\[ \E i(Q_{d,p}) =
	2 \cdot 2^{2^{d-1}} \exp\left[ \tfrac12 (2-p)^d + \left(a(p)\tbinom d2 -\tfrac14 \right) 2^d (1-\tfrac p2)^{2d} + O\left(d^4 2^d (1- \tfrac p2)^{3d}\right) \right], \]
	where
	\[ a(p) := \frac{(1+(1-p)^2)^2 }{(2-p)^4} - \frac14 .\]
\end{thm}

A simple form of \cref{thm:expectation} is obtained by keeping only the leading term in the exponent, yielding that for $p \ge \frac{C\log d}{d^{1/3}}$,
\begin{equation}\label{eq:expectation-simple}
\E i(Q_{d,p}) = 2 \cdot 2^{2^{d-1}} \exp\left[ \tfrac12 (2-p)^d (1+o(1)) \right] .
\end{equation}
We mention that this formula fails for $p=o(\frac1d)$ (see \cref{sec:open} for a discussion).
The formula~\eqref{eq:expectation-simple} gives an asymptotic expression for the logarithm of the expected number of independent sets in $Q_{d,p}$ for $p \ge \frac{C\log d}{d^{1/3}}$. For some values of $p$, we further obtain the asymptotics of $\E i(Q_{d,p})$ itself. For example, if $p=1-\frac cd + o(\frac1d)$ for a fixed $c \ge 0$, then~\eqref{eq:expectation-simple} already yields that
\begin{equation}\label{eq:expecatation-c/d}
\E i(Q_{d,p}) = (1+o(1)) \cdot 2 e^{\frac12 e^c} \cdot 2^{2^{d-1}} .
\end{equation}
For constant $p>2-\sqrt2 \approx 0.586$, \cref{thm:expectation} implies that $\E i(Q_{d,p})$ is asymptotic to $2 \cdot 2^{2^{d-1}} e^{\frac12 (2-p)^d}$.
Similarly, for constant $p>2-2^{2/3} \approx 0.413$, the theorem implies that this expectation is asymptotic to $2 \cdot 2^{2^{d-1}} e^{\frac12 (2-p)^d + (a(p)\binom d2-\frac 14)2^d(1-\frac p2)^{2d}}$.
As with the proof of Jenssen and Perkins in \cite{jenssen2020independent}, our proof gives additional correction terms in the exponent to arbitrary accuracy, allowing to compute the expansion up to $O(d^{2n-2}2^d(1-\tfrac p2)^{nd})$ for any fixed $n>0$ (see \Cref{re:betterAccuracy}). In particular, this gives a way to obtain the asymptotics of the expectation for any constant $p$.

One natural point to consider is $p=\tfrac 12$. On the hypercube, this is in fact the critical point for various graph properties of $Q_{d,p}$ including the existence of isolated vertices, minimal degree at least two, connectivity~\cite{burtin1977probability,erdos1979evolution,bollobas1983evolution}, the existence of a perfect matching~\cite{bollobas1990complete}, and as was very recently shown, Hamiltonicity~\cite{condon2021hamiltonicity}.
 For the number of independent sets, \Cref{thm:expectation} gives 
\begin{equation}\label{eq:Expectation1/2}
 \E i(Q_{d,1/2}) = 2 \exp\left[\tfrac12 (\tfrac32)^d + \tfrac14(\tfrac98)^d \left(\tfrac{91}{9}\tbinom d2 -1\right)\right] \cdot 2^{2^{d-1}} \left(1 + O(d^4) (\tfrac{27}{32})^d \right).
\end{equation}
For the reader's convenience, we mention already now a formula for the variance, which will follow from one of our later results:
\begin{equation}
\label{eq:Variance1/2}
 \Var\, i(Q_{d,1/2}) = 4\exp\left[ (\tfrac32)^d + \tfrac12 (\tfrac54)^d + (\tfrac98)^d \left(\tfrac{91}{18}\tbinom d2 -1\right) \right] \cdot 2^{2^d} \left(1 + O(d^2) (\tfrac{15}{16})^d \right) .
\end{equation}
We point out that $\Var\, i(Q_{d,1/2}) \gg (\E i(Q_{d,1/2}))^2$.

For $p \ge \frac23$, we also have the following result on the typical number of independent sets. Recall that a sequence of random variables $(X_d)$ is \textbf{tight} if for every $\varepsilon>0$ there exists $M>0$ such that $\Pr(|X_d|>M)\leq \varepsilon$ for all $d$. Equivalently, if for any function $w=\omega(1)$, we have that $|X_d|=O(w)$ with high probability.
\begin{thm}\label{cor:fluctuation-bound}
	For  $p \ge \frac23$, 
\[ i(Q_{d,p}) = 2e^{\frac12 (2-p)^d} \cdot 2^{2^{d-1}} \left(1 + (2-\tfrac{3p}2)^{d/2} \cdot X_d \right) ,\]
where $X_d$ is tight as $d \to \infty$.
	Furthermore, if also $p \le 1-\Omega(\frac1d)$, then $|X_d|=\Omega(1)$ with probability $\Omega(1)$, and in particular $X_d$ does not converge to zero in probability.
\end{thm}

The first part of the theorem gives an upper bound on the fluctuations of $i(Q_{d,p})$ when $p \ge \frac23$. In particular, for $p>\frac 23+\omega(\frac 1d)$, we get that with high probability \[i(Q_{d,p}) = 2e^{\tfrac12 (2-p)^d} \cdot 2^{2^{d-1}} \left(1 + o(1)\right).\]
 Recalling that the expectation also behaves like the right-hand side, we see that $i(Q_{d,p})$ is concentrated around its mean.  
Note that for $p=\frac 23$ the first part of the theorem does not give a concentration result. In fact, the second part of the theorem implies that  $i(Q_{d,p})$ is not concentrated for $p=\frac 23$. Thus the theorem implies that a change of behavior occurs around $p=\frac23$.
For larger values of $p$, the second part of the theorem gives a lower bound on the fluctuations, thereby pinpointing the order of magnitude of the fluctuations.

The first part of \cref{cor:fluctuation-bound} will be a consequence of an understanding of the variance of $X_d$, and the second part will be a consequence of an understanding of its fourth moment. For the latter part regarding the lower bound on the fluctuations, the assumption that $p \le 1-\Omega(\frac1d)$ can be relaxed to allow $p$ to approach 1 much faster, though the actual fluctuations are smaller than $(2-\frac{3p}2)^{d/2}$ for such $p$. These are determined by the variance which will be discussed in \Cref{thm:variance} below.
We point out that for all constant $p \in [\frac23,1)$, the fluctuations of $i(Q_{d,p})$ are of larger order of magnitude than the second correction term in the expectation obtained from \cref{thm:expectation}, namely $\Theta(d^22^d(1-\frac p2)^{2d})$, so that the the fluctuations overwhelm the latter correction term.
Thus, for constant $p \in [\frac23,1)$, while the leading order asymptotics is deterministic, the second order correction is already random. It is natural to ask whether $X_d$ in \cref{cor:fluctuation-bound} converges in distribution as $d \to \infty$, and if so, to what?

The next result answers this question for $p$ tending to 1.

\begin{thm}\label{thm:normal}
	For $p=1-o(1)$ such that $p \le 1 - 2^{-d/3 + \omega(\log d)}$,
	\begin{equation*}
	\frac {i(Q_{d,p})-\E i(Q_{d,p})}{\sqrt{\Var \, i(Q_{d,p})}} \implies \cN(0,1),
	\end{equation*}
	where $\implies$ denotes convergence in distribution.
\end{thm}

The results above follow from an understanding of the moments of $i(Q_{d,p})$. The next result gives an asymptotic formula for the second moment.
For increased neatness, we describe the result in terms of the ratio between the second moment and the square of the expectation. Together with \cref{thm:expectation} this translates to an asymptotic formula for the variance.

\begin{thm}\label{thm:variance}
	For $p \ge \frac{C\log d}{d^{1/3}}$,
	\begin{align*}
		\frac{\E i(Q_{d,p})^2}{(\E i(Q_{d,p}))^2} &=
		\frac12 \exp\left[ \tfrac12 (2-\tfrac{3p}2)^d - \tfrac12 2^d(1-\tfrac p2)^{2d} + O(d^4) (1-\tfrac p2)^d(2-\tfrac{3p}2)^d \right] \\&+ \frac12\exp\left[ \tfrac{p(1-p)}{2(2-p)^2} d 2^d \left( 1-\tfrac p2 \right)^{2d} + O(d^4) (1-\tfrac p2)^d(2-\tfrac{3p}2)^d \right].
	\end{align*}
\end{thm}

As we have seen in \cref{cor:fluctuation-bound}, there is a change in behavior around $p=\frac23$, where for this value of $p$ there is no concentration, while for larger values there is. \cref{thm:variance} shows that the variance also undergoes a change of behavior around this point.
Indeed, the variance of $i(Q_{d,p})$ is of the same order as the square of its expectation for $p = \frac23 \pm \Theta(\frac1d)$, but is of larger order for $p \le \frac23 - \omega(\frac1d)$, while for $p \ge \frac23 + \omega(\frac1d)$, it is of smaller order, implying that $i(Q_{d,p})$ is concentrated around its expectation. 
This leads to the first part of \cref{cor:fluctuation-bound}, while the second part requires also the fourth moment, which is addressed below.

Our methods yield precise asymptotic formulas for any moment of $i(Q_{d,p})$, as well as for its central moments. To keep the exposition simple, we formulate the next result with lower precision than was given in \cref{thm:variance} for the second moment. As before, for increased neatness, we describe the result in terms of the ratio between the $k$-th moment and the $k$-th power of the expectation.

\begin{thm}\label{thm:moments}
	For any $k \ge 2$ and $p \ge \frac{Ck^2\log d}{d^{1/3}}$,
	\[ \frac{\E i(Q_{d,p})^k}{(\E i(Q_{d,p}))^k} =
	2^{-k} \sum_{m=0}^k \tbinom km \exp\left[ \tfrac12\big(\tbinom m2 + \tbinom{k-m}2\big) (2-\tfrac{3p}2)^d + O(d)2^d(1-\tfrac p2)^{2d} \right] .\]
	Furthermore, for any even $k \ge 4$ and $\frac23 + \omega(\frac1d) \le p \le 1 - \omega(\frac1d)$,
	\[ \frac{\E(i(Q_{d,p})-\E i(Q_{d,p}))^k}{(\E i(Q_{d,p}))^k} = \left(2^{-k}(k-1)!! \cdot (2-\tfrac{3p}2)^{kd/2} + 2^{-k}2^d(1-p+p2^{-k})^d\right)(1+o(1)) .\]
	This remains true for $p = \frac23 \pm O(\frac1d)$ if one replaces the term $1+o(1)$ by $\Theta(1)$, and it remains true for all $p \ge \frac23$ if one replaces it by $O(1)$. This also holds for odd $k \ge 3$ and $p \ge \frac23$ with the $O(1)$.
\end{thm}
In particular, this implies that, similarly to~\eqref{eq:expectation-simple}, for any $k \ge 2$ and $p \ge \frac{Ck^2\log d}{d^{1/3}}$,
\[ \E i(Q_{d,p})^k =
2^k \cdot 2^{k2^{d-1}} \exp\left[ \tfrac k2 (2-p)^d (1+o(1)) \right].
\]
\cref{thm:moments} shows that all higher moments undergo a change in behavior around $p=\frac23$ similar to that of the second moment. Specifically, for $k \ge 2$ and $p \ge \frac23 + \omega(\frac1d)$, the $k$-th moment of $i(Q_{d,p})$ is asymptotic to the $k$-th power of its expectation, whereas for $\frac{C\log d}{d^{1/3}} \le p \le \frac23 - \omega(\frac1d)$, it is of larger order.
When $p=\frac23 \pm O(\frac1d)$, all moments of $i(Q_{d,p})/\E i(Q_{d,p})$ are $1+\Theta(1)$.

Let us now discuss the central moments, which are of particular interest.
When $\frac{C\log d}{d^{1/3}} \le p \le \frac23 - \omega(\frac1d)$, it is not hard to see using the first part of the theorem that the $k$-th central moment is asymptotically the same as the $k$-th moment itself.
On the other hand, when $p \ge \frac23 - O(\frac1d)$, the second part of the theorem gives the order of magnitude of the normalized $k$-th central moment, showing that it is the larger of one of two terms. As it turns out, for $2 \le k \le 7$, the first term is always the larger of the two, whereas for any $k \ge 8$, there are two numbers $\frac23 < p_k^* < p_k^{**} < 1$ such that the second term is larger for $p \in (p_k^*,p_k^{**})$, and otherwise the first term is larger. As $k \to \infty$, these satisfy $p_k^* \to \frac23$ and $p_k^{**} \to 1$. Thus, for any constant $p \in (\frac23,1)$, the normalized $k$-th central moment is of order $2^d(1-p+p2^{-k})^d$ for all $k$ large enough.
On the other hand, when $p$ tends to 1, the normalized $k$-th central moment is asymptotic to $2^{-k}(k-1)!! (2-\frac{3p}2)^{kd/2}$ for all even $k \ge 2$. This suggests normal behavior as in \cref{thm:normal} and this will indeed follow from a slightly stronger version of \cref{thm:moments} (see \cref{thm:central-moments}).

\subsection{The hard-core model}\label{sec:hardcore}

A natural and well-studied generalization of independent sets is the hard-core model.
In this model, one is given a parameter $\lambda>0$ called the \textbf{fugacity}, and one samples a random independent set $I$ of a given finite graph $G$ with probability proportional to $\lambda^{|I|}$. The \textbf{partition function} of the model is the normalization constant given by
\[ Z(G,\lambda) := \sum_{I\text{ indep.\ set in }G} \lambda^{|I|} .\]
Note that $\lambda=1$ corresponds to counting independent sets, i.e., $Z(G,1)=i(G)$.

The hard-core model, which originates from statistical mechanics, serves as a simple model of gas or hard spheres.
It has been extensively studied by mathematical physicists, probabilists, combinatorialists and computer scientists (scheduling problems, communications). 

Galvin~\cite{galvin2011threshold} studied the hard-core model on the hypercube. Among his results which described the typical structure of a configuration in the hard-core model, was an extension of the basic result~\eqref{eq:KS} of Korshunov and Sapozhenko to the hard-core model, which states that for $\lambda \ge \sqrt2-1+\frac{(\sqrt2+\Omega(1))\log d}{d}$,
\[ Z(Q_d,\lambda) = (1+o(1)) \cdot 2 (1+\lambda)^{2^{d-1}} \exp\left[ \tfrac \lambda 2 \left(\tfrac2{1+\lambda}\right)^d \right] ,\]
and that for $\lambda \ge \frac{C\log d}{d^{1/3}}$,
\[ Z(Q_d,\lambda) = 2 (1+\lambda)^{2^{d-1}} \exp\left[ \tfrac \lambda 2 \left(\tfrac2{1+\lambda}\right)^d (1+o(1)) \right] .\]
Jenssen and Perkins~\cite{jenssen2020independent} gave more refined results, including a formula and algorithm which allows to find the asymptotics of $Z(Q_d,\lambda)$ for any constant $\lambda$. For example, for $\lambda \ge 2^{1/3}-1+\frac{2^{7/3}\log d+\omega(1)}{2d}$,
\[ Z(Q_d,\lambda) = (1+o(1)) \cdot 2 (1+\lambda)^{2^{d-1}} \exp\left[ \tfrac \lambda 2 \left(\tfrac2{1+\lambda}\right)^d \left(1+\tfrac{(2\lambda^2+\lambda^3)d(d-1)-2\lambda}{4(1+\lambda)^d}\right) \right] .\]

Our following result extends the above to random subgraphs of the hypercube.

\begin{thm}\label{thm:expectation-general0}
Suppose that $\lambda \le \lambda_0$ and $\lambda p \ge \frac{C\log d}{d^{1/3}}$. Then
\[ \E Z(Q_{d,p},\lambda) =
 	2(1+\lambda)^{2^{d-1}} \exp\left[ \tfrac\lambda 2 2^d \alpha_1^d + \left(a\tbinom d2 -\tfrac14 \right)\lambda^2 2^d \alpha_1^{2d} + O\left(d^4 \lambda^3 2^d \alpha_1^{3d}\right) \right], \]
where
\[ \alpha_1 := 1 - \frac{\lambda p}{1+\lambda} \qquad\text{and}\qquad a := \frac{(1+\lambda)^2(1+\lambda(1-p)^2)^2}{4(1+\lambda -\lambda p)^4} - \frac14  .\]
\end{thm}

In the theorem, and throughout the paper, $\lambda_0$ is an arbitrarily large constant, and $C,c$ denote constants which may depend on $\lambda_0$ but are otherwise universal.

The requirement that $\lambda$ is bounded can be relaxed, but can not be entirely dropped. For example, the conclusion of the theorem fails when $p \in (0,1)$ is constant and $\lambda \ge e^{C(p)d}$, since $\E Z(Q_{d,p})$ is trivially bounded above by $(1+\lambda)^{2^d}$, whereas the expression in the theorem is much larger in this case. While our proof would allow to obtain results for $\lambda$ that does not grow too fast, the more interesting regime is when $\lambda$ is constant or tends to 0, and so we have opted to keep things simpler by assuming that $\lambda$ is bounded.

\cref{thm:expectation-general0} generalizes \cref{thm:expectation}. All the other results discussed in the introduction similarly generalize to the hardcore model. We do not state these generalizations here and refer the reader to \cref{sec:moments} for details.

\subsection{A family of positive-temperature extensions of the hard-core model}\label{sec:posTemp}

As it turns out, the hard-core model on a random subgraph $G_p$ of a graph $G$ (keeping each edge of $G$ independently with probability $p$) is related to another statistical mechanics model on the base graph $G$. The latter model can be thought of as a positive-temperature hard-core model (which is nothing other than the antiferromagnetic Ising model with external magnetic field). In this model, in addition to the fugacity parameter $\lambda>0$, one is given a parameter $\beta \in [0,\infty]$ called the \textbf{inverse temperature}, and one samples a subset $I$ of vertices in $G$ with probability proportional to $\lambda^{|I|} e^{-\beta |E(I)|}$,
where $E(I) := \{ e \in E(G) : e \subset I \}$ is the set of edges of $G$ spanned by $I$.
The \textbf{partition function} of the model is
\begin{equation}\label{eq:postemp-partition-function}
Z(G,\lambda,\beta) := \sum_{I \subset V(G)} \lambda^{|I|} e^{-\beta |E(I)|} .
\end{equation}
Note that when $\beta=\infty$ this reduces to the usual hard-core model on $G$ and $Z(G,\lambda,\infty)=Z(G,\lambda)$.

We will see that when $p$ and $\beta$ satisfy the relation $p=1-e^{-\beta}$, we have the following useful relation between the partition function of the hard-core model on $G_p$ and that of the positive-temperature hard-core model on $G$:
\begin{equation}\label{eq:Z-relation}
\E Z({G_p,\lambda}) = Z({G,\lambda,\beta}) .
\end{equation}
This relation is key for our understanding and analysis of the expected number of independent sets in $Q_{d,p}$.

In fact, there is also a certain representation for the moments of $Z({G_p,\lambda})$ in terms of the positive-temperature hard-core model on $G$. Taking the same relation between $p$ and $\beta$, letting $I_1,\dots,I_k$ be independent and identically distribution random variables chosen from the positive-temperature hard-core model on $G$, we have
\begin{equation}
\frac{\E Z({G_p,\lambda})^k}{(\E Z({G_p,\lambda}))^k} = \E e^{\beta(|E(I_1)|+\cdots+|E(I_k)| - |E(I_1)\cup\cdots\cup E(I_k)|)} .
\end{equation}
Actually, we will not use this relation, but found it interesting to mention as it also gives another interpretation for our results on the moments.

We will instead use a different relation, closer in spirit to that in~\eqref{eq:Z-relation}. In this relation, the $k$-th moment of $Z({G_p,\lambda})$ is related to a model of $k$ interacting sets $I_1,\dots,I_k$ of vertices of $G$. Specifically, one samples $I_1,\dots,I_k \subset V(G)$ with probability proportional to $\lambda^{|I_1|+\cdot+|I_k|} e^{-\beta |E(I_1) \cup \cdots \cup E(I_k)|}$.
The \textbf{partition function} of the model is
\[ Z_k(G,\lambda,\beta) := \sum_{I_1,\dots,I_k \subset V(G)} \lambda^{|I_1|+\cdot+|I_k|} e^{-\beta |E(I_1) \cup \cdots \cup E(I_k)|} .\]
Note that $Z_1(G,\lambda,\beta)=Z(G,\lambda,\beta)$ and that $Z_k(G,\lambda,\infty)=Z(G,\lambda)^k$.

The following is an extension of the relation~\eqref{eq:Z-relation} to arbitrary $k \ge 1$.

\begin{prop}\label{prop:relation}
Let $G$ be a finite graph, let $\lambda>0$ and $\beta \in [0,\infty]$ and set $p=1-e^{-\beta}$. Then
\[ Z_k({G,\lambda,\beta}) = \E Z({G_p,\lambda})^k .\]
\end{prop}

\begin{proof}
	Denote $G=(V,E)$ and $q=e^\beta -1= p/(1-p)$. Then
	\begin{align*}
	Z_k({G,\lambda,\beta}) &=\sum_{I_1,\dots,I_k\subseteq V}\lambda^{|I_1|+\dots+|I_k|}\prod_{e\in E}e^{-\beta\textbf{1}_{e\in E(I_1)\cup\dots\cup E(I_k)}}\\
	&=\sum_{I_1,\dots,I_k\subseteq V}\lambda^{|I_1|+\dots+|I_k|}e^{-\beta|E|}\prod_{e\in E} (1+q\textbf{1}_{e\notin E(I_1)\cup\dots\cup E(I_k)})\\ 
	&=\sum_{I_1,\dots,I_k\subseteq V}\lambda^{|I_1|+\dots+|I_k|}e^{-\beta|E|}\sum_{\omega\subseteq E}q^{|\omega|}\prod_{e\in \omega} \textbf{1}_{e\notin E(I_1)\cup\dots\cup E(I_k)}\\
	&=\sum_{\omega\subseteq E}e^{-\beta|E|}q^{|\omega|}\sum_{I_1,\dots,I_k\subseteq V}\lambda^{|I_1|+\dots+|I_k|} \textbf{1}_{\{I_1,\dots,I_k\ are\ independent\ in\ \omega\}}\\
	&=\sum_{\omega\subseteq E}e^{-\beta|E|}q^{|\omega|}Z(\omega,\lambda)^k\\
	&=\sum_{ \omega\subseteq E}Z(\omega,\lambda)^k p^{|\omega|}(1-p)^{|E\setminus \omega|}\\
	&=\E Z(G_p,\lambda)^k. \qedhere
	\end{align*}
\end{proof}

In light of this relation between the hard-core model on $Q_{d,p}$ and the positive-temperature models on $Q_d$, the proofs of the main results boil down to analyzing the latter models. This analysis establishes an understanding of the structure of a typical configuration. Let us first describe this structure when $k=1$. A typical configuration will mostly be contained in one of the sides, $\cE$ or $\cO$, of the hypercube. Configurations which are entirely contained in one side may be thought of as ``ground states'', and then a typical configuration can be seen as a small deviation from such a ground state. For $k \ge 2$, there are $2^k$ classes of ground states, each characterized by a vector $\cD \in \{\cE,\cO\}^k$, where the corresponding ground state configurations are those having $I_1 \cap \cD_1 = \cdots = I_k \cap \cD_k = \emptyset$. Thus, we think of $\cD$ as describing the ``defect/deviation sides''. We will establish a convergent cluster expansion for this model, which makes rigorous the fact that typical configurations are small deviations from such ground state configurations.

\subsection{Proof outline}
In this section we give an outline of the proofs of our main results.
The proofs combine a number of ideas and techniques, including the cluster expansion for polymer models, approximations of contours, comparison with the model on the complete bipartite graph $K_{d,d}$, the aforementioned relation between the hard-core model on $Q_{d,p}$ and the family of positive-temperature models on $Q_d$, and the method of moments.

If we were content with weaker versions of the results, we would not require all of the above ingredients. In \cref{sec:warmup}, we warm up by proving that for $p \ge \frac23 + \omega(\frac1d)$, the number of independent sets in $Q_{d,p}$ is $(1+o(1)) \cdot 2e^{\frac12 (2-p)^d} \cdot 2^{2^{d-1}}$, both in expectation and with high probability (see \cref{thm:simple}). Let us first discuss the proof outline for this weaker result, which does not rely on the cluster expansion, and involves many of the ideas that go into the main theorems.

\smallskip\noindent
\textbf{Outline for \cref{thm:simple}:}
To prove the theorem, it suffices to lower bound $i(Q_{d,p})$ with high probability, and to upper bound its expectation. For the lower bound, it is instructive to first see how one obtains a tight lower bound on $i(Q_d)$. By considering independent sets which are entirely contained in one bipartition class of the hypercube (we call these ground states), one easily sees that
\[ i(Q_d) \ge 2 \cdot 2^{2^{d-1}} - 1 .\]
This bound already gives the correct order of magnitude. To get the correct leading constant, it suffices to additionally take into account independent sets which are ``almost'' entirely contained in one side of the hypercube (small deviations from a ground state), in the sense that only a bounded number of vertices belong to the other side. Indeed, since any set of size $k$ in the hypercube has at most $kd$ neighbors, the number of independent sets which are contained in (say) the even side, except for precisely $k$ vertices in the odd side, is at least $2^{n - kd} \binom nk$, where $n:=2^{d-1}$. For fixed $k$, this is asymptotically the same as $2^n \frac1{k!} 2^{-k}$. Summing over $0 \le k \le K$, with $K$ slowly tending to infinity, and reversing the roles of even and odd to obtain an additional factor of 2 (noting that the double counting is negligible), yields that
\[ i(Q_d) \ge (1-o(1)) \cdot 2\sqrt{e} \cdot 2^{2^{d-1}} .\]

The lower bound on $i(Q_{d,p})$ is obtained by a similar ``direct'' counting argument, using the second moment method in order to control fluctuations and produce a bound which holds with high probability. While we only required bounded $k$ for the bound on $i(Q_d)$, the bound on $i(Q_{d,p})$ will use $k$ up to roughly $(2-p)^d$ (which is bounded precisely when $p=1-O(\frac1d)$; compare with~\eqref{eq:expecatation-c/d}). The details of this lower bound are given in \cref{sec:warmup-lower-bound}.

Let us give a heuristic for why $(2-p)^d$ is a relevant order of magnitude. 
Consider independent sets which are mostly contained in the even side of the hypercube. If we consider a randomly chosen independent set, then we may approximate the state of even vertices as independent fair coin flips. In order for it to be possible for a given odd vertex to belong to the independent set, each of the $d$ adjacent vertices must either be vacant or the edge connecting to it should not appear in $Q_{d,p}$. By the independence assumption, this has probability $(1-p+p/2)^d$. Thus, there are roughly $2^{d-1}(1-p+p/2)^d = \frac12(2-p)^d$ odd vertices which have no occupied neighbors.


Let us now discuss the upper bound on the expectation of $i(Q_{d,p})$.
As we have seen in \cref{sec:posTemp}, this expectation can be interpreted as the partition function of a positive-temperature model.
To upper bound the partition function of the latter model, we employ and extend techniques of Galvin~\cite{galvin2011threshold} and Peled--Spinka~\cite{peled2020long} with the goal of showing that most configurations do not deviate much from a ground state. The basic objects we work with, called \emph{polymers}, are 2-linked sets of vertices contained in one side of the hypercube, and whose closures do not contain more than $\frac34$ of the vertices on that side (the precise constant $\frac34$ is not important; any constant strictly between $\frac12$ and 1 would work). Here, 2-linked means that it is connected in the enhanced graph where edges are added between distance-two vertices, and the closure of a set is the largest set with the same neighborhood. Polymers represent local deviations from a ground state. Each polymer $A$ has an associated weight
\[ \omega(A) :=  \sum_{B \subset N(A)} 2^{-|N(A)|} e^{-\beta|E(A,B)|} ,\]
defined so that the weight of a configuration (under some minor restrictions) can be written as the product of weights of the polymers it decomposes into. At zero temperature ($\beta=\infty$), this weight is simply $\omega(A)=2^{-|N(A)|}$, which comes from the fact that removing $A$ from the independent set, frees up the vertices in $N(A)$ to be added (or not) to the independent set. At positive temperature, however, the picture is more involved as it is possible for the vertices in $N(A)$ to be occupied even when the vertices in $A$ are occupied, and thus the weight $\omega(A)$ is given by a sum of the contributions from the possible states of vertices in $N(A)$.

Not all configurations can be seen as polymers configurations. Roughly speaking, configurations which contain both a significant number of even and odd vertices cannot be identified with a polymer configuration.
We refer to these as non-polymer configurations. More precisely, non-polymer configurations are those whose closures have both even and odd 2-linked components of size larger than $\frac34 \cdot 2^{d-1}$. In particular, the closure of a non-polymer configurations contains more than $\frac34 \cdot 2^{d-1}$ even and odd vertices. In fact, this will be the only property we use in order to bound the total weight of non-polymer configurations. We note that at zero temperature, non-polymer configurations are simply not possible (for this it is enough that the closure contains more than $\frac12 \cdot 2^{d-1}$ even and odd vertices), so that there is nothing to show. At positive, but very low, temperature, each non-polymer configuration has very small weight and a simple union bound suffices. However, at lower temperature (even already for constant $\beta$), bounding the total weight of non-polymer configurations is non-trivial. For this we use a technique from~\cite{peled2020long}, based on entropy methods~\cite{kahn2001entropy,galvin2004weighted,galvin2012bounding}, which relies on a certain comparison with the model on the complete bipartite graph $K_{d,d}$.
We state the required bound (\cref{lem:bad-configs0}) in the warm-up section and leave the proof to the later \cref{sec:bad-configs}.

In order to show that deviations are unlikely in polymer configurations, we will bound the total weight of all polymers.
Small polymers have small weight relative to their size (since $N(A)$ is much larger than $A$), and so it is not hard to rule out the existence of a small deviation at a given vertex in a typical configuration. However, while the weight of a large polymer is small in absolute terms (this by itself is already a challenge at positive temperature and requires the aforementioned technique from~\cite{peled2020long}; note that the contribution to $\omega(A)$ is small for large $B$, but there are many such $B$ to sum over), it is not so small relative to the number of such polymers (since $N(A)$ is not much larger than $A$), so that it is not obvious how to rule out large deviations.

The main technical step toward obtaining an upper bound on the partition function is then to bound the total weight of large polymers. For this we use an approximation scheme (container method) for polymers. Such approximations were initially used by Korshunov and Sapozhenko~\cite{Korshunov1983Th}, and subsequently in numerous works including~\cite{galvin2011threshold,peled2020long}. The positive-temperature nature of the model makes the use of these approximation more involved (e.g., in comparison to~\cite{Korshunov1983Th,galvin2011threshold}), and again require the use of the aforementioned technique from~\cite{peled2020long}. As the bound on the weight of large polymers is quite technical and long, we have only stated the required bound (\cref{lem:Lemma15ForTheorem1.1}) in the warm-up section and left the proof to the later \cref{sec:ProofOfGab}, where the bound is shown in the generality needed for the other results of the paper as well. The other details of the upper bound are given in \cref{sec:warmup-upper-bound}.


\smallskip\noindent
\textbf{Outline for \cref{thm:expectation,thm:expectation-general0}:}
While one could in theory use the same approach as for \cref{thm:simple} to obtain precise lower and upper bounds on $\E i(Q_{d,p})$, such a ``hands-on'' approach would likely be cumbersome in practice. Instead, we employ the well-developed machinery of the cluster expansion, which will allow for nice bookkeeping and provide formulas for various quantities of interest in terms of polymers (in a similar spirit as inclusion-exclusion).
Background on the cluster expansion can be found in~\cite{brydges1984short,scott2005repulsive}. A nice example of how the cluster expansion can be used to obtain precise asymptotic and further probabilistic information is given by the work of Jenssen and Perkins~\cite{jenssen2020independent} who extend results of Galvin~\cite{galvin2011threshold} on the hard-core model on $Q_d$.

The new object here is a \emph{cluster}, which is a sequence of polymers (repetition is allowed) with certain connectivity properties among them. The weight of a cluster is the product of the weights of its polymers, times another factor (which may be positive or negative), called the Ursell function, which depends on the connectivity structure between the polymers.
 The cluster expansion is a formal expression for the logarithm of the partition function, which expresses it as the sum of weights of all clusters. Although the system is finite, there are infinitely many clusters, making this a formal sum which could potentially be absolutely divergent. There are several conditions in the literature which guarantee the absolute convergence of the cluster expansion. A particularly useful one, which we shall use, is due to Koteck\'y and Preiss~\cite{kotecky1986cluster}.
 
Two steps remain in order to obtain the theorem.
The first step is to verify the Koteck\'y--Preiss condition for the absolute convergence of the cluster expansion. The two main inputs needed for this are ones which were already needed and discussed for \cref{thm:simple} -- bounding the total weight of non-polymer configurations and bounding the total weight of large polymers. The second step is to compute the cluster expansion series (to the desired accuracy). For the results as stated in \cref{thm:expectation} and \cref{thm:expectation-general0}, we compute precisely the contribution to the series from clusters of size one and two, and bound the absolute contribution from larger clusters. Of course, one could compute more terms precisely and thereby obtain more precise results.

\smallskip\noindent
\textbf{Outline for \cref{cor:fluctuation-bound,thm:normal}:}
Both theorems will follow from a good understanding of the (central) moments of $i(Q_{d,p})$, which is given in \cref{thm:expectation,thm:variance,thm:moments} and also in \cref{thm:moments-general,thm:central-moments}. For the first part of \cref{cor:fluctuation-bound}, we use an upper bound on the variance, while for the second part, we use a lower bound on the variance and an upper bound on the fourth central moment.
For \cref{thm:normal}, we show that the standardized moments converge to those of a standard normal random variable, and the convergence in distribution will follow.

\smallskip\noindent
\textbf{Outline for \cref{thm:variance,thm:moments}:}
The starting point for understanding the moments of $i(Q_{d,p})$ is the fact that the $k$-th moment $\E i(Q_{d,p})^k$ can be interpreted as the partition function $Z_k$ of a positive-temperature $k$-component model (see \cref{sec:posTemp}). The same techniques used for \cref{thm:expectation} can be applied to obtain a convergent cluster expansion for this $k$-component model (with a suitable definition of a polymer). Computing this cluster expansion to some desired accuracy leads to the formulas for the moments.

To study the central moments, we use a binomial expansion in order to write the $k$-th central moment $\E (i(Q_{d,p}) - \E i(Q_{d,p}))^k$ in terms of the (non-central) moments. Plugging in the cluster expansion series and suitably manipulating the series, we obtain an expression for the $k$-th central moment as a sum over sequences of clusters, where the only allowed sequences are those which ``span'' all $k$ components of the system (in a certain precise sense).
Computing this series leads to the formula for the $k$-th central moment.

%
%

\subsection{Notation}\label{sec:notation}

Given a graph $G=(V,E)$, we write $u \sim v$ when $u$ and $v$ are adjacent vertices. We write $N(v)$ for the neighbors of $v \in V$, and we write $N(U) := \bigcup_{u \in U} N(u)$ for the neighborhood of $U \subset V$.
For a subset $U \subset V$, we define the \textbf{closure} of $U$ to be largest set $[U]$ with the same neighborhood as $U$, i.e., $[U] := \{ v \in V : N(v) \subset N(U) \}$.

The $d$-dimensional hypercube $Q_d$ is the Hamming graph on $\{0,1\}^d$, i.e., the graph with vertex set $\{0,1\}^d$ and edge set $\{ \{u,v\} : \sum_{i=1}^d |u_i-v_i|=1 \}$. We call a vertex of $Q_d$ even or odd according to the sum of its coordinates. We denote by $\cE$ and $\cO$ the set of even and odd vertices of $Q_d$, respectively.

For real numbers $a$ and $b$, we write $a \wedge b := \min\{a,b\}$ and $a \vee b := \max\{a,b\}$.

We write $f \ll g$ for $f = o(g)$. All asymptotics are as $d \to \infty$, unless otherwise stated, and $d$ is assumed to be large enough when needed.

We use $\lambda_0$ to denote an arbitrarily large constant. 
We write $C,c$ for positive constants, which may depend on $\lambda_0$ but are otherwise universal, and which may change from line to line (with large constants only increasing, and small constants decreasing). We use the notation $a\approx b$ in the sense that $\frac{a}{b}$ is bounded away from zero and infinity by universal constants.

\subsection{Preliminaries}
We will make use of the following isoperimetric inequalities from~\cite{Korshunov1983Th,galvin2011threshold,jenssen2020homomorphisms}.

\begin{lemma}\label{lem:isoperimetry}
	Let $d \ge 1$ and let $\cE \cup \cO$ be the partition of $V(Q_d)$ into the even and odd vertices.
	Suppose $S\subseteq \cE$ (or $S\subseteq \cO$). Then
\[ |N(S)| \ge \begin{cases}
 d|S|-2|S|^2 &\text{if }|S|\leq d/10 \\
 \frac1{10}d|S| &\text{if }|S|\leq d^4 \\
 \big(1+\Omega\big(1/\sqrt d\big)\big)|S| &\text{if }|S|\leq (1-\Omega(1)) 2^{d-1}
\end{cases} .\]
\end{lemma}


We also use the following graph-theoretic lemma (see, e.g., \cite[Lemma 13]{jenssen2020independent}).

\begin{lemma}\label{lem:Lemma13}
	The number of 2-linked subsets $S\subseteq V(Q_d)$ of size at most $t$ which contain given vertex $v$ is at most $(ed^2)^{t-1}$.
\end{lemma}

\subsection{Organization}

In \cref{sec:warmup}, we prove a simplified version of \cref{thm:expectation,cor:fluctuation-bound}.
In \cref{sec:clusterEx}, we introduce the cluster expansion, define a polymer model, and establish results relating the partition function of the polymer model to the partition function $Z(Q_d,\lambda,\beta)$ of the positive-temperature hard-core model.
In \cref{sec:convergence}, we prove the convergence of the cluster expansion by verifying the Koteck\'y--Preiss condition (\cref{thm:KP}), and further provide bounds on the tail of the cluster expansion series. For this, we define and use approximations to bound the total weight of large polymers.
In \cref{sec:moments}, we establish stronger versions of our main results for the hard-core model on $Q_{d,p}$. In particular, all the results of \cref{sec:intro} will follow from \cref{thm:expectation-general,thm:moments-general,thm:central-moments}. We conclude with a discussion and open questions in \cref{sec:open}.

\section{Warm-up}\label{sec:warmup}

In this section, we give a ``hands-on'' proof of basic versions of some of our results for $p \ge \frac23$. While the general proof does not follow the same route (in particular, we do not use the cluster expansion here), we still hope this helps convey some basic ideas in a simple setting. While the lower bound is not too difficult and we provide below all details of the proof, the upper bound relies on a special case (\cref{lem:Lemma15ForTheorem1.1}) of a powerful technical lemma (about the total weight of larger polymers via approximations) and on a special case (\cref{lem:bad-configs0}) of an additional technical lemma (about the total weight of non-polymer configurations via entropy methods) whose proofs are only given later in the general setting (see \cref{sec:convergence}). We prove the following:

\begin{thm}\label{thm:simple}
Suppose that $p \ge \frac23 + \omega(\frac1d)$. Then
\[ i(Q_{d,p}) = (1+o(1)) \cdot 2e^{\frac12 (2-p)^d} \cdot 2^{2^{d-1}} \]
with high probability and in expectation.
\end{thm}

\cref{thm:simple} will follow from a lower bound on $i(Q_{d,p})$ that holds with high probability, and a matching upper bound on its expectation.
Namely, for the lower bound, we need to show that, with high probability,
	\begin{equation}\label{eq:lower-bound}
	i(Q_{d,p}) \ge (1+o(1)) \cdot 2 e^{\frac12 (2-p)^d} \cdot 2^{2^{d-1}} .
	\end{equation}
For the upper bound, we need to show that
\begin{equation}\label{eq:upper-bound}
\E i(Q_{d,p}) \le (1+o(1)) \cdot 2 e^{\frac12 (2-p)^d} \cdot 2^{2^{d-1}} .
\end{equation}
These two bounds together yield \cref{thm:simple}.

The proof of \eqref{eq:lower-bound} is by a rather direct computation.
The proof of \eqref{eq:upper-bound} requires more work and will use the relation with the positive-temperature hard-core model given by \cref{prop:relation}. We mention that the upper bound will in fact work for $p \ge 2-\sqrt{2} + \omega(\frac{\log d}d)$, and a matching lower bound on the expectation also follows for such $p$ from the proof of~\eqref{eq:lower-bound} (and in fact also for much smaller $p$, but the bound is far from the truth in that case).

\subsection{The lower bound}\label{sec:warmup-lower-bound}

Let $i_{m,\ell}$ denote the number of independent sets in $Q_{d,p}$ which contain exactly $m$ vertices of $\cE$ and $\ell$ vertices of $\cO$. Then
\[ i(Q_{d,p}) = \sum_{m,\ell} i_{m,\ell} .\]
Define $i_{m,*} := \sum_\ell i_{m,\ell}$ and $i_{*,\ell} := \sum_m i_{m,\ell}$.
Also define $i_{\le k} := \sum_{m,\ell \le k} i_{m,\ell}$. Then
\begin{equation}\label{eq:lower-bound-basic}
i(Q_{d,p}) \ge \sum_{k=0}^K (i_{k,*}+i_{*,k}) - i_{\le K} \qquad\text{for any }K \ge 0 .
\end{equation}

Our plan is to use~\eqref{eq:lower-bound-basic} with $K := (2-p)^d$.
For this we first aim to give a lower bound on $i_{k,*}$ which holds with high probability for any particular $k \le K$.
Using that $i_{k,*}$ and $i_{*,k}$ have the same distribution, and that $i_{\le K}$ is trivially always at most $2^{2Kd}$, this will already yield a lower bound close to~\eqref{eq:lower-bound} (with the $1+o(1)$ in the exponent). To obtain the desired bound~\eqref{eq:lower-bound}, we will give a lower bound on the sum of $i_{k,*}$ over $k \le K$ in a similar manner.

We proceed to bound $i_{k,*}$ from below.
Observe that
\[ i_{k,*} = \sum_{A \subset \cE: |A|=k} 2^{2^{d-1}- |\tilde N(A)|} \ge 2^{2^{d-1}} \sum_{A \subset \cE: |A|=k} 2^{- \sum_{v \in A} d_v} ,\]
where $\tilde N(A)$ is the neighborhood of $A$ in $Q_{d,p}$, and $d_v$ is the degree of $v$ in $Q_{d,p}$.
Define
\[ S_k := \sum_{A \subset \cE: |A|=k} X_A \qquad\text{where}\qquad X_A := 2^{-\sum_{v \in A} d_v} .\]
We proceed to lower bound $S_k$.
Note that $S_k$ is a sum of identically distributed random variables. These random variables are not all independent, but most pairs are, and we can expect that $S_k$ is concentrated around its mean. To show this, we first compute the mean of $S_k$ and then bound its variance. Using that $\E x^{\Bin(m,p)} = (1-p+px)^m$, we see that
\[ \E S_k = \binom nk \E 2^{-\Bin(dk,p)} = \binom nk (1-\tfrac p2)^{kd} ,\]
where $n := 2^{d-1}$.
For the variance of $S_k$, we have that
\[ \Var(S_k) = \sum_{A,A'} \Cov(X_A,X_{A'}) = \sum_{i=1}^k \binom nk \binom ki \binom {n-k}{k-i} \Cov(X_{A_0},X_{A_i}) ,\]
where the first sum is over sets $A,A' \subset \cE$ of size $k$, and where $A_0,\dots,A_k$ are any subsets of $\cE$ of size $k$ such that $|A_0 \cap A_i|=i$. We have that
\[ \Cov(X_{A_0},X_{A_i}) \le \E X_{A_0}X_{A_i}  = \E 4^{-\Bin(id,p)} \E 2^{-\Bin(2(k-i)d,p)} = (1-\tfrac{3p}4)^{id} (1-\tfrac p2)^{2(k-i)d} .\]
Thus, using that $\binom ki \le k^i$ and $\binom {n-k}{k-i} \le \binom n{k-i} \le \binom nk \big(\frac k{n-k}\big)^i$,
\[ \Var(S_k) \le \binom nk^2 (1-\tfrac p2)^{2kd} \sum_{i=1}^k \left(\frac{k^2}{n-k}\right)^i \left(\frac{1-\tfrac{3p}4}{(1-\tfrac p2)^2}\right)^{id} .\]
Note that since $p \ge \frac23 + \omega(\frac1d)$, we have for $k \le K$ that
\[ \frac{k^2}{n-k} \left(\frac{1-\tfrac{3p}4}{(1-\tfrac p2)^2}\right)^d =  o(1) .\]
Thus, $\Var(S_k) \ll (\E S_k)^2$ so that, by Chebychev's inequality, $S_k \ge (1-o(1)) \E S_k$ with high probability.
In particular, for any $k \le K$, with high probability,
\[ i_{k,*} \ge (1-o(1)) \cdot 2^{2^{d-1}} \cdot \binom nk (1-\tfrac p2)^{kd} .\]
From here, simply using that $i(Q_{d,p}) \ge i_{K,*} + i_{*,K} - 2^{2Kd}$, it would already be possible to deduce that, with high probability,
\[ i(Q_{d,p}) \ge 2 \cdot 2^{2^{d-1}} \cdot e^{\frac12 (2-p)^d (1-o(1))} .\]

To get the desired lower bound~\eqref{eq:lower-bound}, we aim to show
that $S := \sum_{k=0}^K S_k$ is concentrated around its mean. Observe first that
\[ \E S = \sum_{k=0}^K \binom nk (1-\tfrac p2)^{kd} = (1-o(1)) (1 +(1-\tfrac p2)^d)^n = (1-o(1)) e^{\frac12(2-p)^d} ,\]
where in the second equality we used that $[0,K]$ contains a symmetric interval of size $\omega(\sqrt{nq(1-q)})$ around $nq$ (which tends to infinity), where $q:=(1-\tfrac p2)^d /(1+(1-\tfrac p2)^d)$, and where in the last equality we used that $e^{x-x^2} \le 1+x \le e^x$ and that $n(1-\frac p2)^{2d}=o(1)$ since $p \ge 2-\sqrt2 + \omega(\frac1d)$.
Let us now bound the variance of $S$. In a similar manner as before, we obtain that
\[ \Var(S) = \sum_{A,A'} \Cov(X_A,X_{A'}) \\
  \le \sum_{k=0}^K\sum_{k'=0}^K \sum_{i=1}^{k \wedge k'} \binom nk \binom ki \binom {n-k'}{k'-i} (1-\tfrac{3p}4)^{id} (1-\tfrac p2)^{(k+k'-2i)d} ,\]
where the first sum runs over sets $A,A' \subset \cE$ of size at most $K$, so that
\[ \Var(S) \le \sum_{k=0}^K \sum_{k'=0}^K \binom nk \binom n{k'} (1-\tfrac p2)^{(k+k')d} \sum_{i=1}^{k \wedge k'} \left(\frac{kk'}{n-k'}\right)^i \left(\frac{1-\tfrac{3p}4}{(1-\tfrac p2)^2}\right)^{id} \ll (\E S)^2 .\]
Thus, \eqref{eq:lower-bound-basic} yields that, with high probability,
\[ i(Q_{d,p}) \ge (1-o(1)) 2^{2^{d-1}} \cdot 2\E S - 2^{2Kd} = (1-o(1)) \cdot 2 \cdot 2^{2^{d-1}} e^{\frac12 (2-p)^d} .\]
This establishes~\eqref{eq:lower-bound}.

\subsection{The upper bound}\label{sec:warmup-upper-bound}

Recall the positive-temperature hard-core model from \cref{sec:posTemp} and recall from~\eqref{eq:Z-relation} that $\E i(Q_{d,p})=Z$, where $Z := Z(Q_d,1,\beta)$ was defined in~\eqref{eq:postemp-partition-function} and $\beta:=-\log(1-p)$. Our goal is thus to upper bound $Z$.

A polymer is a 2-linked subset of $\cE$ whose closure (defined in \cref{sec:notation}) has size at most $\frac34 \cdot 2^{d-1}$ (later in \cref{sec:clusterEx} we define more general polymers). For $A \subset \cE$, define
\[ \omega(A) := \sum_{B \subset N(A)} 2^{-|N(A)|}e^{-\beta|E(A,B)|} .\]
We begin by showing that
\begin{equation}\label{eq:warmup-Z-bound}
Z \le 2 \cdot 2^{2^{d-1}} \cdot \exp\left( \sum_{\gamma\text{ polymer}} \omega(\gamma) \right) + \sum_{I \in \cI} \omega(I) ,
\end{equation}
where $\cI$ is the collection of all configurations $I \subset V(Q_d)$ such that $|[I \cap \cE]|,|[I \cap \cO]| >\frac34 2^{d-1}$.
To see this, recall from~\eqref{eq:postemp-partition-function} that $Z$ is a sum over all configurations $I$, and write $Z=Z'+Z''$, where $Z''$ sums over configurations $I \in \cI$ and $Z'$ sums over the remaining configurations. By even-odd symmetry,
\[ Z' \le 2 \cdot 2^{2^{d-1}} \sum_{A \subset \cE, B \subset N(A), |[A]| \le \frac34 2^{d-1}} 2^{-|N(A)|}e^{-\beta|E(A,B)|} = 2 \cdot 2^{2^{d-1}} \sum_{A \subset \cE, |[A]| \le \frac34 2^{d-1}} \omega(A) .\]
By decomposing $A$ into its 2-linked components $A_1,\dots,A_m$ (which are polymers), noting that $\omega(A)=\omega(A_1)\cdots\omega(A_m)$, and taking into account the $m!$ possible ordering of $A_1,\dots,A_m$, we get
\begin{align*}
 \sum_{A \subset \cE, |[A]| \le \frac34 2^{d-1}} \omega(A)
  \le \sum_{m=0}^\infty \frac1{m!} \left(\sum_{\gamma\text{ polymer}} \omega(\gamma)\right)^m
  = \exp\left(\sum_{\gamma\text{ polymer}} \omega(\gamma)\right).
\end{align*}
This proves~\eqref{eq:warmup-Z-bound}.

It suffices to show that
\[ \sum_{\gamma\text{ polymer}} \omega(\gamma) = \frac12 (2-p)^d + o(1) \qquad\text{and}\qquad \sum_{I \in \cI} \omega(I) = o(Z) .\]
A simple computation shows that polymers of size 1 contribute $2^{d-1} 2^{-d} (1+e^{-\beta})^d = \frac12 (2-p)^d$.
The following two lemmas show that the contribution from larger polymers is negligible and that the second sum above is negligible, thereby completing the proof of~\eqref{eq:upper-bound}.

\begin{lemma}\label{lem:Lemma15ForTheorem1.1}
	 For $p \ge 2-\sqrt2 + \omega(\frac{\log d}d)$,
	 \[  \sum_{\gamma\text{ polymer},~|\gamma| \ge 2} \omega(\gamma) = o(1). \]
\end{lemma}

\begin{lemma}\label{lem:bad-configs0}
	 For $p \ge \frac{C\log d}{d^{1/3}}$,
	 \[ \sum_{I \in \cI} \omega(I) = o(Z). \]
\end{lemma}

The lemmas are proved in \cref{sec:convergence} (see \Cref{lem:Lemma15pre} for a stronger version of the first lemma and \cref{lem:bad-configs-k=1} for a stronger version of the second lemma).

\section{Cluster expansion}\label{sec:clusterEx}

Recall the model of $k$ interacting sets $I_1,\dots,I_k$ described in \cref{sec:posTemp}. We will henceforth refer to this as the \textbf{$k$-system}.
The goal of this section is to write the partition function $Z_k=Z_k(Q_d,\lambda,\beta)$ of the $k$-system using an expansion into so-called clusters. 
We will define a new model, called \textbf{the polymer model}, based on the $k$-system, which inherits the parameters $k,d,\lambda,\beta$ from the relevant $k$-system, and is used in order to give a good estimate for the partition function $Z_k$.
We first give the required definitions, with explanations following the theorem.

A \textbf{polymer} is a tuple $\gamma=(A_1,\dots,A_k)$ of sets such that
 \begin{itemize}
  \item Each $A_i$ is contained in either $\cE$ or $\cO$.
  \item Each $[A_i]$ has size at most $\frac34 \cdot 2^{d-1}$ (recall the definition of $[\cdot]$ from \cref{sec:notation}).
  \item The graph $H_\gamma$ is connected, where $H_\gamma$ is the graph whose vertices are all pairs $(i,u)$ with $i \in [k]$ and $u \in A_i$ and with two vertices $(i,u)$ and $(j,v)$ adjacent whenever $i=j$ and $\dist(u,v)=2$, or $i \neq j$ and $\dist(u,v) \in \{0,1\}$.
 \end{itemize}
The \textbf{weight} of the polymer $\gamma$ is
\begin{equation}\label{eq:polymer-weight-def}
\omega(\gamma):= \sum_{B_1 \subset N(A_1), \dots, B_k \subset N(A_k)} \frac{\lambda^{|A_1|+\cdots+|A_k|+|B_1|+\cdots+|B_k|}}{(1+\lambda)^{|N(A_1)|+\cdots+|N(A_k)|}} e^{-\beta|E(A_1,B_1) \cup \cdots \cup E(A_k,B_k)|}.
\end{equation}

Fix $\cD \in \{\cE,\cO\}^k$.
A \textbf{$\cD$-polymer} is a polymer $\gamma=(A_1,\dots,A_k)$ such that $A_i \subset \cD_i$ for all $i$.
Two $\cD$-polymers $\gamma$ and $\gamma'$ are \textbf{incompatible} if their coordinate-wise union $\gamma \cup \gamma' = (A_1 \cup A'_1,\dots,A_k \cup A'_k)$ satisfies that $H_{\gamma \cup \gamma'}$ is connected; otherwise they are \textbf{compatible} (in which case, the connected components of $H_{\gamma \cup \gamma'}$ are precisely $H_\gamma$ and $H_{\gamma'}$). We write $\gamma \sim \gamma'$ for compatible polymers and $\gamma \not\sim \gamma'$ for incompatible polymers.
A \textbf{$\cD$-cluster} is an ordered tuple $\Gamma=(\gamma_1,\dots,\gamma_n)$ of $\cD$-polymers such that the \textbf{incompatibility graph} $H_\Gamma$ is connected. Here $H_\Gamma$ is the graph with vertex set $\{1,\dots,n\}$ and with $i$ and $j$ adjacent when $\gamma_i \not\sim \gamma_j$. The \textbf{weight} of the cluster $\Gamma$ is
 \[ \omega(\Gamma) := \phi(H_\Gamma)\omega(\gamma_1)\cdots \omega(\gamma_n) ,\]
where $\phi(H)$ is the Ursell function of a graph $H$, defined by
\[ \phi(H) := \frac1{|V(H)|!} \sum_{\substack{S \subset E(H)\\\text{spanning, connected}}} (-1)^{|S|} .\]
We denote the set of all $\cD$-clusters by $\cC_\cD$.
Note that $\cC_\cD$ is infinite since the same polymer can be repeated any number of times in a cluster.

The following condition will recur in many of our results:
\begin{equation}\label{eq:lambda-cond-k}
\lambda \le \lambda_0 \qquad\text{and}\qquad \lambda(1-e^{-\beta}) \ge \frac{Ck^2\log d}{d^{1/3}} .
\end{equation}
We remind the reader that the particular constant $\lambda_0$ is not important, that $C$ is a universal constant (except that it may depend on $\lambda_0$) and that $d$ is always assumed to be sufficiently large.

\begin{thm}\label{thm:ClusterEx}
	Fix $k \ge 1$ and suppose that~\eqref{eq:lambda-cond-k} holds. Then
	\[ Z_k = (1+\lambda)^{k 2^{d-1}} \sum_{\cD\in \{\cE,\cO\}^k}  \exp\left( \sum_{\Gamma \in \cC_{\cD}} \omega(\Gamma)\right) \cdot (1+O(\exp(-2^d/d^4))) ,\]
	where the cluster expansion series (the inner sum) is absolutely convergent.
\end{thm}

%
%

In applications of the theorem, it is useful to have explicit bounds on the absolute tail of the cluster expansion series; such bounds are provided in \cref{sec:convergence}. In fact, such a bound is already needed for the proof of \cref{thm:ClusterEx}.
We state the required bound here, but defer its proof to \cref{sec:convergence}.
The \textbf{size} of a polymer $\gamma=(A_1,\dots,A_k)$ is $\|\gamma\| := |A_1|+\cdots+|A_k|$, and the \textbf{size} of the cluster $\Gamma=(\gamma_1,\dots,\gamma_n)$ is $\|\Gamma\| := \|\gamma_1\|+\cdots+\|\gamma_n\|$.


\begin{lemma}\label{lem:cluster-convergence}
Fix $k \ge 1$ and $\cD \in \{\cE,\cO\}^k$ and suppose that~\eqref{eq:lambda-cond-k} holds. Then
\[ \sum_{\Gamma\in \cC_\cD}|\omega(\Gamma)| e^{\|\Gamma\| d^{-3/2}} = O(1) \cdot 2^d \left(\frac{1+\lambda e^{-\beta}}{1+\lambda}\right)^d .\]
\end{lemma}

We will also need the following lemma for the proof of \cref{thm:ClusterEx}.
\begin{lemma}\label{lem:bad-configs}
Fix $k \ge 1$ and suppose that~\eqref{eq:lambda-cond-k} holds. Then
\[ \sum_{\substack{I_1,\dots,I_k \subset V(Q_d):\\|[I_i \cap \cE]|,|[I_i \cap \cO]| > \frac34 \cdot 2^{d-1} \text{ for some }i \in [k]}} \lambda^{|I_1|+\cdot+|I_k|} e^{-\beta |E(I_1) \cup \cdots \cup E(I_k)|} \le Z_k \cdot O(\exp(-2^d/d)) .\]
\end{lemma}

\bigbreak

Let us now motivate the definitions given above.
Recall that in the $k$-system, configurations are tuples $I=(I_1,\dots,I_k)$ of subsets of $Q_d$ and that one samples such a configuration with probability proportional to $\lambda^{|I_1|+\cdots+|I_k|} e^{-\beta |E(I_1) \cup \cdots \cup E(I_k)|}$. The corresponding probability measure $\mu_k$ is given by
\[ \mu_k(I) := \frac{\lambda^{|I_1|+\cdots+|I_k|} e^{-\beta |E(I_1) \cup \cdots \cup E(I_k)|}}{Z_k} .\]
Let $\cD \in \{\cE,\cO\}^k$ and let $g=(A_1,\dots,A_k,B_1,\dots,B_k)$ be a tuple of sets $A_i \subset \cD_i$ and $B_i \subset N(A_i)$. Define
\begin{equation}\label{eq:dec-polymer-weight-def}
\omega(g):= \frac{\lambda^{|A_1|+\cdots+|A_k|+|B_1|+\cdots+|B_k|}}{(1+\lambda)^{|N(A_1)|+\cdots+|N(A_k)|}} e^{-\beta|E(A_1,B_1) \cup \cdots \cup E(A_k,B_k)|}.
\end{equation}
 Let $E_\cD(g)$ be the event that $I_i \cap \cD_i = A_i$ and $I_i \cap N(A_i) = B_i$ for all $i$.
A straightforward computation reveals that
\begin{equation}\label{eq:weight-ratio}
 \mu_k(E_\cD(\emptyset)) = \frac{(1+\lambda)^{k2^{d-1}}}{Z_k} \qquad\text{and}\qquad \frac{\mu_k(E_\cD(g))}{\mu_k(E_\cD(\emptyset))} = \omega(g),
\end{equation}
where $\emptyset$ is identified here with $(\emptyset,\dots,\emptyset)$.
It is precisely now that the definition of a polymer comes into play. Suppose for a moment that each $A_i$ has size at most $\frac34 \cdot 2^{d-1}$. Then there is a unique set of $\cD$-polymers $\{\gamma_j\}$ such that the connected components of $H_{(A_1,\dots,A_k)}$ are precisely $\{H_{\gamma_j}\}$. Note that the polymers $\{\gamma_j\}$ are necessarily pairwise compatible. We extend each polymer $\gamma_j$ to a ``decorated'' polymer $\hat\gamma_j$ in the following way: if $\gamma_j=(A^{j}_1,\dots,A^{j}_k)$ then $\hat\gamma_j:=(A^{j}_1,\dots,A^{j}_k,B^{j}_1,\dots,B^{j}_k)$, where $B^{j}_i := B_i \cap N(A^{j}_i)$. Then $\{A^j_i\}_j$ and  $\{B^j_i\}_j $ partition $A_i$  and $B_i$, respectively, and $\omega(g)$ factorizes over these decorated polymers:
\[ \omega(g) = \prod_j \omega(\hat\gamma_j) .\]
Note that our earlier assumption that $A_i$ has size at most $\frac34 \cdot 2^{d-1}$ was not strictly necessary; it was only used to ensure that each of the components $\gamma_j$ themselves satisfy the analogous requirement.

This leads us to the following definitions.
A \textbf{decorated polymer} is a tuple $\hat\gamma = (A_1,\dots,A_k,B_1,\dots,B_k)$ such that $\gamma=(A_1,\dots,A_k)$ is a polymer and $B_i \subset N(A_i)$ for all $i$. The \textbf{size} of such a decorated polymer is $\|\hat\gamma\| := \|\gamma\|$ and its weight $\omega(\hat\gamma)$ is defined by the same formula as in~\eqref{eq:dec-polymer-weight-def}.
Thus, the weight $\omega(\gamma)$ of a polymer $\gamma$ is the sum of the weights $\omega(\hat\gamma)$ of decorated polymers $\hat\gamma$ which extend it. We say that two decorated polymers are compatible if their underlying polymers are compatible.
Let $\Omega_{\cD}$ denote the family of all sets of pairwise compatible decorated $\cD$-polymers. We sometimes refer to the elements of $\Omega_\cD$ as \textbf{polymer configurations}.
The \textbf{size} of a polymer configuration $\Theta$ is $\|\Theta\| := \sum_{\hat\gamma \in \Theta} \|\hat\gamma\|$.
The \textbf{the polymer model} (associated with $\cD$) is the probability measure $\nu_{\cD}$ on $\Omega_{\cD}$ define by
\[ \nu_{\cD}(\Theta) := \frac{\prod_{\hat\gamma \in \Theta} \omega(\hat\gamma)}{\Xi_{\cD}} , \qquad \Theta \in \Omega_\cD , \]
where the partition function $\Xi_{\cD}$ is given by
\[ \Xi_{\cD} := \sum_{\Theta \in \Omega_{\cD}}  \prod_{\hat\gamma \in \Theta} \omega(\hat\gamma) .\]
The cluster expansion for the logarithm of the partition function of the polymer model associated to $\cD$ is the formal power series in the weights of the clusters:
\begin{equation}\label{eq:cluster-expansion}
\log \Xi_\cD = \sum_{\Gamma \in \cC_\cD} \omega(\Gamma) .
\end{equation}
The cluster expansion is a powerful and classical tool which applies to general abstract polymer models (for background see, e.g., \cite{jenssen2020independent} and references therein). For our particular polymer model, \cref{lem:cluster-convergence} will ensure that the above cluster expansion series is absolutely convergent for the corresponding parameter range.


It is instructive to note that if we were to drop the size requirement from the definition of a polymer, then $\nu_\cD$ could be precisely identified with a certain marginal of $\mu_k$, namely, the distribution of $(I_i\cap \cD_i,N(I_i\cap\cD_i))_i$ where $(I_1,\dots,I_k)$ is sampled from $\mu_k$.
%
%
This size requirement is, however, crucial and makes the two measures quite different (though $\mu_k$ is related to a mixture of the $\nu_{\cD}$).

\subsection{Remarks}\label{sec:remarks}

Let us give some remarks regarding the above definitions and results.
Regarding the requirement that $H_\gamma$ is connected in the definition of a polymer, we note that this implies that $A_1 \cup \cdots \cup A_k$ is a 2-linked set (for $k=1$, it is exactly equivalent). Many of our arguments regarding polymers (e.g., for their weighted counting) will only rely on this weaker property. In fact, both \cref{thm:ClusterEx} and \cref{lem:cluster-convergence} would remain true if we were to replace the requirement that $H_\gamma$ is connected in the definition of a polymer with the requirement that $A_1 \cup \cdots \cup A_k$ is 2-linked. On the other hand, the stronger requirement will make precise computations easier to handle as it gives rise to less polymers.
We mention that it would have also been a natural choice to define the decorated polymers to be polymers to begin with (which would change the notion of a cluster accordingly), but we have found our choice more convenient to work with.

For the requirement that $[A_i]$ has size at most $\frac34 \cdot 2^{d-1}$, the precise constant $\frac34$ is not important; any constant greater than $\frac12$ and less than 1 would suffice for our purposes. Previous works on the hard-core model used the constant $\frac12$~\cite{galvin2011threshold,jenssen2020independent}, which naturally arises from the fact that a subset of $Q_d$ which contains more than half of the even vertices and half of the odd vertices cannot be an independent set. Since configurations in the positive temperature model are arbitrary subsets of $Q_d$ and not just independent sets, it is simpler to work with a constant $c \in (\frac12,1)$, which guarantees that any subset of $Q_d$ which contains a $c$-fraction of the even vertices and of the odd vertices is far from being an independent set in the sense that it spans many edges. We note that similar considerations are also relevant in the homomorphism models studied in~\cite{jenssen2020homomorphisms}, where a suitable constant greater than $1/2$ is also used.

Let us also discuss the role of $\cD \in \{\cE,\cO\}^k$.
This vector indicates for each of the $k$ sets $I_1,\dots,I_k$, which side of the hypercube is the ``defect side'', with the other side being the dominant side where most of the configuration resides. Configurations in which $I_1 \cap \cD_1 = \cdots = I_k \cap \cD_k = \emptyset$ are ground states which  correspond to the polymer model associated with $\cD$, and the cluster expansion describes configurations as (typically small) deviations from such ground states. Each choice of $\cD$ actually gives a different polymer model (having its own cluster expansion), with two different choices $\cD$ and $\cD'$ leading to isomorphic models if $m(\cD)=m(\cD')$ or $m(\cD)=k-m(\cD')$, where $m(\cD):=|\{ i \in [k] : \cD_i=\cE\}|$. In particular, there are only $\lfloor k/2 \rfloor+1$ truly different polymer models. For example, when $k=1$, the two choices of $\cD$ lead to the ``even'' and ``odd'' polymer model, which are clearly symmetric. When $k=2$, there are two symmetric polymer models having the defects on the same side and two symmetric ones having them on different sides, but the former two are not equivalent to the latter two.

\subsection{Some computational examples}\label{sec:computational-examples}
The reader may find it helpful to see some examples and computations involving polymers and their weights. We give several such examples here. These will not be needed in this section, but will be used later in \cref{sec:moments}.

\smallskip
\noindent
\emph{Scenario I:}
Consider the case $k=1$ and let $\gamma=(A)$ be a polymer.
The smallest polymer is obtained when $A=\{v\}$ for some vertex $v$.
Let us compute the weight of this polymer. There are $2^d$ decorated polymers $\hat\gamma=(A,B)$ extending $\gamma$, one for each subset $B \subset N(v)$. Any such decorated polymer has $|E(A,B)|=|B|$.
Thus,
\[ \omega(\gamma) = \sum_{B \subset A} \frac{\lambda^{|A|+|B|}}{(1+\lambda)^{|N(A)|}} e^{-\beta|E(A,B)|} = \sum_{B \subset N(v)} \frac{\lambda^{1+|B|}}{(1+\lambda)^d} e^{-\beta|B|} = \lambda \left(\frac{1+\lambda e^{-\beta}}{1+\lambda}\right)^d .\]

\smallskip
\noindent
\emph{Scenario II:}
The next simplest polymer (still with $k=1$) is obtained when $A=\{u,v\}$, where $u$ and $v$ are vertices at distance two from each other. Note that $N(A)$ has size $2d-2$, with two vertices there being common neighbors of $u$ and $v$, and the remaining $2d-4$ vertices adjacent to only one of $u$ or $v$. Thus, if a vertex of the former type belongs to $B$, then it contributes 2 edges to $E(A,B)$, while vertices of the latter type in $B$ contribute only one edge.
Thus, $|E(A,B)| = |B_1| +2|B_2|$ where $B_1 = B \setminus (N(u) \cap N(v))$ and $B_2 = B \cap N(u) \cap N(v)$.
Thus,
\begin{align*}
\omega(\gamma)
 &= \sum_{B \subset N(\{u,v\})} \frac{\lambda^{2+|B|}}{(1+\lambda)^{2d-2}} e^{-\beta|E(A,B)|} \\&= \frac{\lambda^2}{(1+\lambda)^{2d-2}} \sum_{B_1 \subset N(\{u,v\}) \setminus N(u) \cap N(v)} \lambda^{|B_1|} e^{-\beta|B_1|} \sum_{B_2 \subset N(u) \cap N(v)} \lambda^{|B_2|} e^{-2\beta|B_2|} \\
 &= \frac{\lambda^2}{(1+\lambda)^{2d-2}} (1+\lambda e^{-\beta})^{2d-4} (1+\lambda e^{-2\beta})^2
 = \lambda^2 \left(\frac{1+\lambda e^{-\beta}}{1+\lambda}\right)^{2d-2} \left(\frac{1+\lambda e^{-2\beta}}{1+\lambda e^{-\beta}}\right)^2 .
\end{align*}
We note that for polymers of size 3, where $A=\{u,v,w\}$, there are two different types: one obtained when any two of $u,v,w$ are at distance two from each other; the other obtained when two of these pairs are at distance two and the third pair is at distance four.
We do not compute the weights of these polymers here.

\smallskip
\noindent
\emph{Scenario III:}
Let us now consider general $k$. We demonstrate a computation in the particular case when the polymer $\gamma=(A_1,\dots,A_k)$ has the smallest possible support (defined as $A_1 \cup \cdots \cup A_k$), but the largest possible size under this restriction. This occurs when $\cD \in \{\cE,\cO\}^k$ consists of all $\cE$ or all $\cO$, and $A_1=\cdots=A_k=\{v\}$ for some vertex $v$. A decorated polymer extending $\gamma$ is determined by a choice of subsets $B_1,\dots,B_k \subset N(v)$. Given such a choice, we have $|E(A_1,B_1) \cup \cdots \cup E(A_k,B_k)| = |B_1 \cup \cdots \cup B_k|$. Thus,
\begin{align*}
\sum_{B_1,\dots,B_k \subset N(v)} \lambda^{|B_1|+\cdots+|B_k|} e^{-\beta|B_1 \cup \cdots \cup B_k|} &= \left(\sum_{b_1,\dots,b_k \in \{0,1\}} \lambda^{b_1+\cdots+b_k} e^{-\beta \1_{\{b_1+\cdots+b_k \ge 1\}}}\right)^d .
\end{align*}
The sum on the right-hand side equals $1 + ((1+\lambda)^k-1)e^{-\beta}$, and hence,
\[ \omega(\gamma) = \lambda^k \left(\frac{1 + ((1+\lambda)^k-1)e^{-\beta}}{(1+\lambda)^k}\right)^d .\]

\smallskip
\noindent
\emph{Scenario IV:}
We consider one last example. Suppose that $k=2$ and that $\cD$ is either $(\cE,\cO)$ or $(\cO,\cE)$ (corresponding to a polymer model where the defects of $I_1$ and $I_2$ lie on different sides of the hypercube). Consider a polymer $\gamma=(A_1,A_2)$ of size 2 whose support also has size 2. That is, $A_1=\{u\}$ and $A_2=\{v\}$, where $u$ and $v$ are adjacent vertices (one is even and one is odd). A decorated polymer extending $\gamma$ is determined by two subsets $B_1 \subset N(u)$ and $B_2 \subset N(v)$. For such a choice, $E(A_1,B_1)$ and $E(A_2,B_2)$ are disjoint except for the edge $\{u,v\}$ in the case that $v \in B_1$ and $u \in B_2$. Thus,
\[ \omega(\gamma) = \lambda^2 \frac{(1+\lambda e^{-\beta})^{2d-2}}{(1+\lambda)^{2d}} \sum_{b_u,b_v \in \{0,1\}} \lambda^{b_u+b_v} e^{-\beta \1_{\{b_u+b_v \ge 1\}}} = \lambda^2 \left( \frac{1+\lambda e^{-\beta}}{1+\lambda} \right)^{2d-2} \cdot \frac{1+2\lambda e^{-\beta} + \lambda^2 e^{-\beta}}{(1+\lambda)^2} .\]

\subsection{Proof of \cref{thm:ClusterEx}}
The rest of this section is devoted to the proof of \cref{thm:ClusterEx}. The proof relies on \cref{lem:cluster-convergence} and a sequence of additional lemmas which we proceed to state and prove. Our approach here follows closely that of~\cite[Section~3.2]{jenssen2020independent}. We assume throughout the section that $\lambda$ is bounded and that $\lambda(1-e^{-\beta}) \gg \frac{\log d}{d^{1/3}}$.

\begin{lemma}\label{lem:Lemma16}
	Let $\Theta$ be a random configuration sampled according to $\nu_\cD$. Then with probability at least $1-O(\exp(-2^d/d^4))$, we have $\|\Theta\| \leq 2^d/d^2$.
\end{lemma}
\begin{proof}
Consider a new polymer model on $\Omega_\cD$ whose weights are
\[ \tilde\omega(\gamma) := \omega(\gamma) \cdot e^{\|\gamma\| d^{-3/2}} .\]
Let $\tilde\Xi_{k,m} = \sum_{\Theta \in \Omega_{\cD}}  \prod_{\hat\gamma \in \Theta} \tilde\omega(\hat\gamma)$ be its partition function and observe that
\[ \E\left[e^{\|\Theta\| d^{-3/2}}\right] = \frac{\tilde\Xi_{k,m}}{\Xi_{k,m}} .\]
Applying \cref{lem:cluster-convergence}, and then using that $\lambda(1-e^{-\beta}) \gg (\log d)/d$, we get that
\[ \log\tilde\Xi_{k,m} \le O(1) 2^d \left(\frac{1+\lambda e^{-\beta}}{1+\lambda}\right)^{d} \le O( 2^d d^{-10}) .\]
Hence,
\[ \log \E\left[e^{\|\Theta\| d^{-3/2}}\right] = \log \tilde\Xi_{k,m} - \log \Xi_{k,m} \le \log \tilde\Xi_{k,m} \le O(2^d d^{-10}) .\]
Thus, by Markov's inequality,
\[ \Pr(\|\Theta\| \ge s) \le e^{-sd^{-3/2}} \E\left[e^{\|\Theta\|d^{-3/2}}\right] \le \exp\left[O(2^d d^{-10})-sd^{-3/2}\right] .\]
Plugging in $s=2^d d^{-2}$, we obtain the lemma.
\end{proof}

We define a measure $\hat\mu_k$ on triplets $(I,\cD,\Theta)$ of configurations $I=(I_1,\dots,I_k)$, vectors $\cD \in \{\cE,\cO\}^k$, called the \textbf{defect side vector}, and polymer configurations $\Theta$ as follows:
\begin{enumerate}
 \item Choose the defect side vector $\cD \in \{\cE,\cO\}^k$ with probability proportional to $\Xi_{k,\cD}$.
 \item Sample a decorated polymer configuration $\Theta \in \Omega_{k,\cD}$ from $\nu_{k,\cD}$.
 \item For each $i \in [k]$:
 \begin{enumerate}
 \item Assign all vertices of $D_i := \bigcup_{(A_1,\dots,A_k,B_1,\dots,B_k) \in \Theta} (A_i \cup B_i)$ to be occupied in $I_i$.
 \item For each vertex $v \notin \cD_i \cup N(D_i)$, include $v$ in $I_i$ with probability $\frac{\lambda}{1+\lambda}$.
 \end{enumerate}
\end{enumerate}
We note that $\Theta$ can be recovered from $(I,\cD)$, so that we may regard $\hat\mu_k$ as a measure on pairs $(I,\cD)$. This measure can be explicitly written: for any feasible $(I,\cD)$, i.e., which can be constructed via the above procedure,
\[ \hat\mu(I,\cD) = \frac{\lambda^{|I_1|+\cdots+|I_k|} e^{-\beta |E(I_1) \cup \cdots \cup E(I_k)|}}{(1+\lambda)^{k2^{d-1}}\sum_{m=0}^k\binom km \Xi_{k,m}} .\]
Denote
\begin{equation}\label{eq:Z-hat}
\hat Z_k := (1+\lambda)^{k2^{d-1}}\sum_{m=0}^k\binom km \Xi_{k,m} .
\end{equation}
Following the remark after the definition of $\nu_\cD$, we note that if we were to drop the size requirement from the definition of a polymer, then any pair $(I,\cD) \in (\{0,1\}^k)^{Q^d} \times \{\cE,\cO\}^k$ would be feasible for $\hat\mu_k$, which means that $\hat\mu_k$ would simply be the product of $\mu_k$ and a uniform vector in $\{\cE,\cO\}^k$. As mentioned before, this size requirement is essential, and $I$ and $\cD$ are not independent under $\hat\mu_k$. The relation between $\hat\mu_k$ and $\mu_k$ is made precise below (see \cref{cor:total-variation}).

The \textbf{minority side vector} of a configuration $I=(I_1,\dots,I_k)$ is $\cM(I) := (\cM(I_1),\dots,\cM(I_k)) \in \{\cE,\cO\}^k$, where the minority side $\cM(A)$ of a subset $A$ of $Q_d$ is $\cE$ or $\cO$ according to the smaller of $|A \cap \cE|$ and $|A \cap \cO|$ (breaking ties arbitrarily).

\begin{lemma}\label{lem:Lemma17}
Let $(I,\cD,\Theta)$ be sampled from $\hat\mu_k$. Then the minority side vector coincides with the defect side vector with high probability. More precisely,
\[ \Pr(\cM(I) \neq \cD) \le O(\exp(-2^d/d^4)) .\]
\end{lemma}
\begin{proof}
By \cref{lem:Lemma16},
\[ \Pr(\cM(I) \neq \cD) \le \Pr\left(\cM(I) \neq \cD \mid \|\Theta\| \le 2^d/d^2\right) + O(\exp(-2^d/d^4)) .\]
By a union bound, it suffices to fix $i \in [k]$ and bound the (conditional) probability that $\cM(I_i) \neq \cD_i$. Since $|I_i \cap \cM(I_i)| \le |I_i \cap \cD_i| \le \|\Theta\|$, it suffices to show that
\[ \Pr\left(|I_i \setminus \cD_i| \le 2^d d^{-2} \mid \|\Theta\| \le 2^d d^{-2}\right) \le O(\exp(-2^d/d^4)) .\]
We henceforth condition on $\Theta$ and work on the event that $\|\Theta\| \le 2^d/d^2$.
Note that $|I_i \setminus \cD_i|$ is the sum of $|D_i \setminus \cD_i|$ (with $D_i$ as above) and an independent Binomial random variable with $L_i:=2^{d-1}-|N(D_i) \setminus \cD_i|$ trials of success probability $\frac{\lambda}{1+\lambda}$. In particular, $|I_i \setminus \cD_i|$ stochastically dominates $\Bin(L_i,\frac{\lambda}{1+\lambda})$. Since $\frac{\lambda}{1+\lambda} \ge \frac1d$, $L_i \ge 2^{d-1} - d|D_i \cap \cD_i|$ and $|D_i \cap \cD_i| \le \|\Theta\|$, we have
\[ \Pr\left(|I_i \setminus \cD_i| \le 2^d d^{-2} \mid \|\Theta\| \le 2^d d^{-2}\right) \le \Pr\left(\Bin\big(2^{d-2}, \tfrac1d\big) \le 2^d d^{-2}\right) \le \exp(-c2^d/d),\]
where the second inequality follows from a standard Chernoff bound.
\end{proof}

\begin{lemma}\label{lem:Lemma14}
	\[ \left|\log Z_k-\log \hat Z_k\right|=O(\exp(-2^d/d^4)) .\]
\end{lemma}
\begin{proof}
Let $\cI := (2^{V(Q_d)})^k$ be the set of all configurations $I=(I_1,\dots,I_k)$. For $\cI' \subset \cI$, define
\[ \cZ(\cI') := \sum_{I \in \cI'} \lambda^{|I_1|+\cdots+|I_k|} e^{-\beta |E(I_1) \cup \cdots \cup E(I_k)|} .\]
Observe that
\[ Z_k = \cZ(\cI) \qquad\text{and}\qquad \hat Z_k = \sum_{\cD \subset \{\cE,\cO\}^k} \cZ(\hat\cI_\cD) ,\]
where $\hat\cI_\cD$ is the set of $I \in \cI$ such that $(I,\cD)$ is feasible under $\hat\mu_k$. 
The set $\hat\cI_\cD$ can be described explicitly, but we only require the observation that
\[ \hat\cI_\cD \supset \cI_\cD := \big\{ I \in \cI : |[\cD_1 \cap I_1]|, \dots, |[\cD_k \cap I_k]| \le \tfrac34 \cdot 2^{d-1} \big\} .\]

Let $\cI^0$ be the set of $I \in \cI$ which belong to no $\cI_\cD$. Then each $I \in \cI$ contributes to $Z_k$ exactly once, and it contributes to $\hat Z_k$ at least once unless $I \in \cI^0$. Some $I$ contribute more than once (anywhere up to $2^k$ times) to $\hat Z_k$, but any $I$ can only contribute once with $\cD=\cM(I)$. Thus, denoting
\[ \hat\cI^*_\cD := \{ I \in \hat\cI_\cD : \cM(I) \neq \cD \} ,\]
we see that
\[ \hat Z_k - \sum_\cD \cZ(\hat\cI^*_\cD) \le Z_k \le \hat Z_k + \cZ(\cI^0) .\]
Thus,
\[ \log(1-\hat\mu_k(\cM(I) \neq \cD)) \le \log Z_k - \log \hat Z_k \le -\log(1-\mu_k(\cI^0)) .\]
It therefore suffices to show that $\hat\mu_k(\cM(I) \neq \cD)$ and $\mu_k(\cI^0)$ are each at most $O(\exp(-2^d/d^4))$.
The former case is precisely \cref{lem:Lemma17} and the latter case follows from \cref{lem:bad-configs}.
\end{proof}

We note the following simple consequence of \cref{lem:Lemma17,lem:Lemma14}.
\begin{cor}\label{cor:total-variation}
Let $\bar\mu_k$ be the distribution of $(I,\cM(I))$ under $\mu_k$. Then
\[ \|\hat\mu_k - \bar\mu_k\|_{\text{TV}} = O(\exp(-2^d/d^4)) .\]
\end{cor}
\begin{proof}
By considering pairs $(I,D)$ with $\hat\mu_k(I,D)>\bar\mu_k(I,D)$, we see that
\[ \|\hat\mu_k - \bar\mu_k\|_{\text{TV}} \le \hat\mu_k(\cM(I) \neq \cD) + \left|1-\tfrac{\hat Z_k}{Z_k}\right| .\]
The corollary now follows from \cref{lem:Lemma17,lem:Lemma14}.
\end{proof}

\begin{proof}[Proof of \cref{thm:ClusterEx}]
By \cref{lem:Lemma14} and~\eqref{eq:Z-hat}, we have that
\[ Z_k = (1+\lambda)^{k2^{d-1}} \sum_\cD \Xi_\cD \cdot e^{O(\exp(-2^d/d^4))} .\]
Thus, it suffices to show that for each $\cD$, $\log \Xi_\cD = \sum_{\Gamma \in \cC_\cD} \omega(\Gamma)$,
where the sum is absolutely convergent. Indeed, this is the cluster expansion series~\eqref{eq:cluster-expansion} of the polymer model associated to $\cD$, and its absolute convergence follows from \cref{lem:cluster-convergence}.
\end{proof}

\section{Convergence of the cluster expansion}
\label{sec:convergence}

Recall the polymer model defined in \cref{sec:clusterEx} and that it has various parameters: the dimension $d$, the fugacity $\lambda$, the inverse temperature $\beta$, the number of sets $k$, and the defect side vector $\cD \in \{\cE,\cO\}^k$. In this section, we give bounds on the absolute tail of the cluster expansion~\eqref{eq:cluster-expansion} of this polymer model. In particular, we will prove \cref{lem:cluster-convergence}.

Throughout this section, we fix $k \ge 1$ and $\cD \in \{\cE,\cO\}^k$, and polymer refers to $\cD$-polymer.

%
   

For $\ell \ge 1$, denote
\[ \alpha_\ell := \frac{1 + ((1+\lambda)^\ell-1)e^{-\beta}}{(1+\lambda)^\ell} \qquad\text{and}\qquad \tilde\alpha_\ell := (\alpha_\ell)^{1/\ell} .\]
Define
\begin{equation}\label{eq:f-g-def}
\tilde g(n) := \begin{cases}
	(dn-3n^2) \log (1/\tilde\alpha_k) - 7n\log d &\text{if }1 \le n \le \frac d{10} \\
	\frac 1{20} dn \log (1/\tilde\alpha_k) &\text{if } \frac d{10} < n \le d^4 \\
	\frac{n}{d^{3/2}} &\text{if }n>d^4 
	\end{cases} .
\end{equation}

%
%
\begin{lemma}\label{lem:Lemma15}
Assume~\eqref{eq:lambda-cond-k}. Then
\[ \sum_{\Gamma\in \cC_\cD : \|\Gamma\| \ge n}|\omega(\Gamma)| e^{\|\Gamma\|d^{-3/2}} \leq d^{-3/2}2^d e^{-\tilde g(n)}  \qquad\text{for any }n \ge 1 .\]
\end{lemma}

\cref{lem:Lemma15} establishes the absolute convergence of the cluster expansion~\eqref{eq:cluster-expansion} and provides bounds on its tail. Let us also mention that it yields a slightly weaker version of \cref{lem:cluster-convergence}, which would already suffice for the applications in \cref{sec:clusterEx}. The precise bound stated in \cref{lem:cluster-convergence}, as well as further estimates on the absolute tail of the cluster expansion, will be shown in \cref{sec:improved-bounds}.

The proof of \cref{lem:Lemma15} relies on checking the Koteck\'y--Preiss condition~\cite{kotecky1986cluster} for convergence of the cluster expansion. This condition can be used for abstract polymer models (where a set of abstract polymers are given, together with weights and a compatibility relation, and some functions $g$ and $f$ on the set of polymers), but we formulate it here specialized to our situation. We refer the reader to~\cite{jenssen2020independent} for a short exposition in the abstract setting.

Fix two functions $f,g \colon \N \to [0,\infty)$.
For a polymer $\gamma$, define $f(\gamma):=f(\|\gamma\|)$ and $g(\gamma):=g(\|\gamma\|)$, and also define $g(\Gamma):=\sum_{\gamma\in\Gamma}g(\gamma)$ for a cluster $\Gamma$.
Recall from \cref{sec:clusterEx} that $\gamma \not\sim \gamma'$ means that $\gamma$ and $\gamma'$ are incompatible polymers. For a polymer $\gamma$ and a cluster $\Gamma$, we write $\gamma \not\sim \Gamma$ whenever $\gamma \not\sim \gamma'$ for some $\gamma' \in \Gamma$.
\begin{thm}[Koteck\'y--Preiss~\cite{kotecky1986cluster}]\label{thm:KP}
Suppose that 
\begin{equation}\label{eq:KP-cond}
\sum_{\gamma'\not\sim \gamma}\omega(\gamma')e^{f(\gamma')+g(\gamma')}\leq f(\gamma) \qquad\text{for any polymer }\gamma .
\end{equation}
Then the cluster expansion~\eqref{eq:cluster-expansion} is absolutely convergent, and furthermore,
\begin{equation}\label{eq:KP-conc}
\sum_{\Gamma\not\sim \gamma}|\omega(\Gamma)|e^{g(\Gamma)}\leq f(\gamma) \qquad\text{for any polymer }\gamma.
\end{equation}
\end{thm}

When applying \cref{thm:KP}, we must specify the functions $f$ and $g$.
In order to prove convergence of the cluster expansion and obtain bounds on the absolute tail of the cluster expansion, it would suffice to apply the theorem with $f(n)=nd^{-3/2}$ and $g=\tilde g$. However, in order to obtain the additional $e^{\|\Gamma\|d^{-3/2}}$ factor in \cref{lem:Lemma15} (which was needed for the proof of \cref{thm:ClusterEx} via \cref{lem:cluster-tail}), we will actually apply the theorem with
\[ f(n) := nd^{-3/2} \qquad\text{and}\qquad g(n) := f(n)+\tilde g(n) .\]
The main input needed to verify that the Koteck\'y--Preiss condition~\eqref{eq:KP-cond} holds with this choice (and for suitable choices of the model parameters) is given in the following lemma whose proof is given in \cref{sec:KP,sec:ProofOfGab}.

For a polymer $\gamma=(A_1,\dots,A_k)$, we define its \textbf{support} to be
\[ S(\gamma) :=A_1 \cup \cdots \cup A_k .\]

%
%
\begin{lemma}\label{lem:Lemma15pre}
	Assume~\eqref{eq:lambda-cond-k}.
	Then for any vertex $v$,
		\begin{align}\label{eq:Lemma15:KPconsition}
	\sum_{\gamma:v \in S(\gamma)}\omega(\gamma)e^{f(\gamma)+g(\gamma)}\leq d^{-7/2}.
	\end{align}
\end{lemma}


\begin{proof}[Proof of \cref{lem:Lemma15}]
	Let us check that the Koteck\'y--Preiss condition~\eqref{eq:KP-cond} holds.
	Indeed, since $\gamma\not\sim \gamma'$ implies that the supports of $\gamma$ and $\gamma'$ are at distance at most 2, or equivalently, that $S(\gamma')\cap B_2(v)\neq \emptyset$ for some $v\in S(\gamma)$ (where 
	$B_2(v)$ is the ball of radius 2 around $v$), and since $|B_2(v)|=1+d+\binom d2\leq d^2$, using  \cref{lem:Lemma15pre} we have that
	\[\sum_{\gamma'\not\sim \gamma}\omega(\gamma')e^{f(\gamma')+g(\gamma')}\leq |S(\gamma)|\cdot |B_2(v)|\cdot d^{-7/2}\leq \|\gamma\|d^{-3/2}=f(\gamma).
	\]
	
	By \cref{thm:KP}, we have the inequality~\eqref{eq:KP-conc}.
	Let us show how this yields the inequality of the lemma.
	Define the support of a cluster $\Gamma$ to be $S(\Gamma) := \bigcup_{\gamma \in \Gamma} S(\gamma)$.
	For any vertex $v$, we can choose a polymer $\gamma$ such that $\|\gamma\|=1$ and $v\in S(\gamma)\cup N(S(\gamma))$, to which we apply \eqref{eq:KP-conc} to obtain that
		$$	\sum_{\Gamma\in \cC_\cD:  v\in S(\Gamma)}|\omega(\Gamma)|e^{g(\Gamma)}\leq \sum_{\Gamma\in \cC_\cD:  \Gamma\not\sim \gamma}|\omega(\Gamma)|e^{g(\Gamma)}\leq f(\gamma) = f(1) =
	 d^{-3/2}.$$
		Summing over all $v$, and writing $\tilde g(\Gamma)=\sum_{\gamma\in\Gamma}\tilde g(\|\gamma\|)$, we get that
\[ \sum_{\Gamma\in \cC_\cD}|\omega(\Gamma)|e^{\|\Gamma\|d^{-3/2}} e^{\tilde g(\Gamma)} = \sum_{\Gamma\in \cC_\cD}|\omega(\Gamma)|e^{g(\Gamma)} \leq 2^dd^{-3/2} .\]
	Since $\tilde g$ is sub-additive as a function on $\N $ (this follows from the fact that $\tilde g(n)/n$ is non-increasing in $n$, which is straightforward to verify using \Cref{cor:alphaIncreasing} and the assumption on $\lambda$ and $\beta$), we have $\tilde g(\|\Gamma\|)\leq \tilde g(\Gamma)$. Since $\tilde g$ is also non-decreasing, we obtain the lemma.
\end{proof}

\begin{proof}[Proof of \cref{lem:Lemma15ForTheorem1.1}]
Since $f$ and $g$ are non-decreasing, \cref{lem:Lemma15pre} yields that
\[ \sum_{\|\gamma\|\geq 2} \omega(\gamma) e^{f(2)+g(2)} \leq \sum_\gamma\omega(\gamma) e^{f(\gamma)+g(\gamma)} \leq  2^d d^{-7/2}. \]
Using that $f(2)+g(2) \ge \tilde g(2)$, plugging in the value of $\tilde g(2)$ (with $k=1$ and $\lambda=1$) and using the assumption that $p \ge 2-\sqrt2+\omega(\frac{\log d}d)$, we get that
\[ \sum_{\|\gamma\|\geq 2} \omega(\gamma) \le 2^d d^{-7/2} e^{-\tilde g(2)} = 2^d d^{-7/2} \left(\frac{1+e^{-\beta}}2\right)^{2d-12} d^{14} = o(1) . \qedhere \]
\end{proof}

\subsection{Verifying the Koteck\'y--Preiss condition}\label{sec:KP}

In this section, we prove \cref{lem:Lemma15pre}, which as we saw, easily yields the Koteck\'y--Preiss condition~\eqref{eq:KP-cond}.
A main technical step is to bound the contribution from large polymers.
We state this as a lemma and prove it separately in \cref{sec:ProofOfGab}.

\begin{lemma}\label{lem:large-polymers}
Assume~\eqref{eq:lambda-cond-k}. Then
\[ \sum_{\gamma : \|\gamma\| > d^4} \omega(\gamma) e^{3\|\gamma\| d^{-3/2}} \le O\big(e^{-d^2}\big) .\]
\end{lemma}

A second main step toward proving \cref{lem:Lemma15pre} is to bound the weight of small polymers. The following provides a bound on the weight of an arbitrary polymer, but is effective primarily for small polymers.

For a polymer $\gamma=(A_1,\dots,A_k)$, we define
\[ N(\gamma) :=(N(A_1),\dots,N(A_k)) \qquad\text{and}\qquad \|N(\gamma)\| := |N(A_1)|+\cdots+|N(A_k)| .\]

\begin{lemma}\label{lem:weight-of-polymer}
For any polymer $\gamma$, we have
\[ \omega(\gamma) \le \lambda^{\|\gamma\|} \tilde\alpha_k^{\|N(\gamma)\|} .\]
\end{lemma}
\begin{proof}
Let $\gamma=(A_1,\dots,A_k)$ be a polymer.
The inequality of the lemma is equivalent to
\[ \sum_{B_1 \subset N(A_1),\dots,B_k \subset N(A_k)} \lambda^{|B_1|+\cdots+|B_k|} e^{-\beta|E(A_1,B_1) \cup \cdots \cup E(A_k,B_k)|} \le \delta_k^{\frac1k \|N(\gamma)\|} ,\]
where
\[ \delta_\ell := \alpha_\ell (1+\lambda)^\ell = 1 + ((1+\lambda)^\ell-1)e^{-\beta} .\]
Since every element in $B_i$ contributes an incident edge to $E(A_i,B_i)$, the sum is at most
\[ \bar\omega(\gamma) := \sum_{B_1 \subset N(A_1),\dots,B_k \subset N(A_k)} \lambda^{|B_1|+\cdots+|B_k|} e^{-\beta|B_1 \cup \cdots \cup B_k|} .\]
At this point, we could simply use that $|B_1 \cup \cdots \cup B_k| \ge \frac1k (|B_1| + \cdots + |B_k|)$ and Newton's binomial to conclude that
\[ \bar\omega(\gamma) \le \sum_{B_1 \subset N(A_1),\dots,B_k \subset N(A_k)} (\lambda e^{-\beta/k})^{|B_1|+\cdots+|B_k|}=(1+\lambda e^{-\beta/k})^{\|N(\gamma)\|}. \]
However, as $1+\lambda e^{-\beta/k}>\delta_k^{1/k}$, this would yield a slightly worse bound than desired. Instead, we proceed to bound $\bar\omega(\gamma)$ as follows. We may rewrite $\bar\omega(\gamma)$ as
\[ \bar\omega(\gamma) = \prod_v \sum_{\substack{b_1,\dots,b_k \in \{0,1\}\\ b_i=0\text{ unless }v \in N(A_i)}} \lambda^{b_1+\cdots+b_k} e^{-\beta \1_{\{b_1+\cdots+b_k \ge 1\}}} = \prod_v \delta_{m_v(\gamma)} ,\]
where
\[ m_v(\gamma) := |\{ i : v \in N(A_i) \}| .\]
A straightforward computation shows that the sequence $(\delta_m)_{m=0}^\infty$ is super-multiplicative, meaning that $\delta_m \delta_n \le \delta_{n+m}$ for any $n,m \ge 0$. In fact, it has the stronger property that $\delta_m \delta_n \le \delta_{m-\ell}\delta_{n+\ell}$ for any $n \ge m \ge \ell \ge 0$ (see \cref{cl:super-multiplicative}). In particular, $\delta_m^{1/m} \le \delta_k^{1/k}$ for $0 \le m \le k$. Since $m_v(\gamma)$ is at most $k$ and since $\sum_v m_v(\gamma) = \|N(\gamma)\|$, applying the former inequality repeatedly and then the latter inequality once yields that
\[ 
\bar\omega(\gamma) = \prod_v \delta_{m_v(\gamma)} \le \delta_k^{\lfloor \frac1k \|N(\gamma)\| \rfloor} \delta_{\|N(\gamma)\| - k\lfloor \frac1k \|N(\gamma)\| \rfloor} \le \delta_k^{\frac1k \|N(\gamma)\|}. \qedhere
\]
\end{proof}

\begin{proof}[Proof of \cref{lem:Lemma15pre}]

	\smallskip
	
	We now turn toward the sum in~\eqref{eq:Lemma15:KPconsition}.
	We split the sum into three parts, bounding each by $\frac 1{3d^{7/2}}$.
	We start with polymers $\gamma$ such that $\|\gamma\|\leq d/10$.
	By \Cref{lem:isoperimetry}, we have $|N(A_i)|\geq d|A_i|-2|A_i|^2$ for each $i$. Thus,
	\[ \|N(\gamma)\| \geq d(|A_1|+\cdots+|A_k|)-2(|A_1|^2+\dots+|A_k|^2)\geq d\|\gamma\|-2\|\gamma\|^2 .\]
	Suppose that $\gamma$ has size $\|\gamma\|=n$.
	Since $\gamma$ contains a given vertex $v$, the number of ways to choose its support is at most $(ed^2)^{n-1}$ by \Cref{lem:Lemma13}. Thus, the number of ways to choose the polymer $\gamma$ itself is at most $(ed^2)^{n-1}2^{kn}\leq d^{3n}$, for $d>e^k$. Thus, using \cref{lem:weight-of-polymer},
	\begin{align*}
	\sum_{\gamma: v \in S(\gamma),\ \|\gamma\|\leq d/10}\omega(\gamma)e^{f(\gamma)+g(\gamma)}&\leq \sum_{n=1}^{d/10}d^{3n}\lambda^n(\tilde\alpha_k)^{dn-2n^2}e^{nd^{-3/2}+g(n)} \le \sum_{n=1}^\infty d^{-4n} \lambda^n e^{2nd^{-3/2}}
	\leq \frac 1{3d^{7/2}},
	\end{align*}
	where the second inequality follows from the definition of $g$ and the fact that $\tilde\alpha_k \le 1$, and the last inequality uses that $\lambda\leq d$.
	\smallskip
	
	Next, we consider polymers $\gamma$ having $d/10<\|\gamma\|\leq d^4$. In this case, \Cref{lem:isoperimetry} yields that $|N(A_i)|\geq d|A_i|/10$ for all $i$, so that $\|N(\gamma)\|\geq d\|\gamma\|/10$. Thus, as before,
	\begin{align*}
	\sum_{\gamma: v \in S(\gamma),\ d/10< \|\gamma\|\leq d^4}\omega(\gamma)e^{f(\gamma)+g(\gamma)}&\leq \sum_{n=d/10}^{d^4}d^{3n}\lambda^n(\tilde\alpha_k)^{dn/10}e^{nd^{-3/2}+g(n)} \\&\leq \sum_{n=d/10}^{d^4} \left(d^3 \lambda (\tilde\alpha_k)^{d/20}e^{d^{-3/2}}\right)^n \le \frac 1{3d^{7/2}},
	\end{align*}
	where the second inequality uses the definition of $g$, and the last inequality is obtained by bounding the sum by $d^4$ times the maximum term. Note that the maximum term is obtained for $n=d/10$ and is $d^{-\omega(1)}$, since $\lambda$ is bounded and $(\tilde\alpha_k)^d = d^{-\omega(1)}$. Recalling that $k$ is fixed, the latter follows from the observations that $\tilde\alpha_k = \alpha_k^{1/k}$, $\alpha_k<\alpha_1$, and $\alpha_1^d = d^{-\omega(1)}$ since $1-\alpha_1=\frac{\lambda(1-e^{-\beta})}{1+\lambda}$, $\lambda(1-e^{-\beta})\gg \log d/d$ and $\lambda$ is bounded.
	
	\smallskip

	Finally, we consider polymers $\gamma$ having $\|\gamma\|>d^4$. 
	\cref{lem:large-polymers} gives that 
	\[ \sum_{\gamma: v \in S(\gamma),\ \|\gamma\|> d^4}\omega(\gamma)e^{f(\gamma)+g(\gamma)} = \sum_{\gamma: v \in S(\gamma),\ \|\gamma\|> d^4} \omega(\gamma) e^{3\|\gamma\| d^{-3/2}}  \le \frac 1{3d^{7/2}} .\]
	
	Putting the three cases together yields the lemma.
\end{proof}

\subsection{Bounding the weight of large polymers via approximations}
\label{sec:ProofOfGab}

In this section, we bound the total weight of large polymers, and in particular prove \cref{lem:large-polymers}. We also prove \cref{lem:bad-configs}.

Recall the definition of the closure $[A]$ of a set $A \subset V(Q_d)$ from \cref{sec:notation}.
For a $\cD$-polymer $\gamma=(A_1,\dots,A_k)$, we denote $[\gamma]:=([A_1],\dots,[A_k])$.
Define
\[ \cG_\cD(a,b) := \Big\{ \gamma\text{ is a $\cD$-polymer} : \|[\gamma]\|=a,~ \|N(\gamma)\|=b \Big\} .\]
We will bound the total weight of polymers in $\cG_\cD(a,b)$ for any $a,b>0$ with $a \ge d^4$. Note that \cref{lem:isoperimetry} implies that $\cG_\cD(a,b)$ is empty unless $b \ge (1+\frac{c}{\sqrt d})a$.
Thus, throughout this section, we fix $a \ge d^4$ and $b \ge (1+\frac{c}{\sqrt d})a$. We also denote $\alpha:=\lambda(1-e^{-\beta})$.

\begin{lemma}\label{lem:Gab}
Assume~\eqref{eq:lambda-cond-k}. Then
\[ \sum_{\gamma \in \cG_\cD(a,b)} \omega(\gamma) \le 2^d \exp\left(-\frac{c(b-a) \alpha^2}{k^2\log d}\right) .\]
\end{lemma}

Recall that one requirement in the definition of a polymer $\gamma=(A_1,\dots,A_k)$ is that the graph $H_\gamma$ is connected. As we have mentioned, this implies that $A_1 \cup \cdots \cup A_k$ is 2-linked. The proof of \cref{lem:Gab} will not use the stronger assumption, but rather only its latter implication. In particular, the statement of \cref{lem:Gab} remains true if one replaces $\cG_\cD(a,b)$ with the larger collection of all $\gamma=(A_1,\dots,A_k)$ such that each $A_i$ is contained in either $\cE$ or $\cO$, $A_1 \cup \cdots \cup A_k$ is 2-linked, $\|[\gamma]\|=a$ and $\|N(\gamma)\|=b$.

Let us see how \cref{lem:Gab} yields \cref{lem:large-polymers}.
\begin{proof}[Proof of \cref{lem:large-polymers}]
	We need to bound the sum of $\omega(\gamma) e^{3\|\gamma\|d^{-3/2}}$ over polymers $\gamma$ of size $\|\gamma\|>d^4$. In fact, we will prove the stronger statement that this bound holds when summing over all polymers $\gamma$ having $\|[\gamma]\|>d^4$.
	For any polymer $\gamma=(A_1,\dots,A_k)$, we have that $|N(A_i)|\geq (1+\frac c{\sqrt d})|A_i|$ for all $i$ by \cref{lem:isoperimetry} (this is the only place where we need the size restriction appearing in the definition of a polymer). In particular, $\|N(\gamma)\|\geq (1+\frac c{\sqrt d})\|\gamma\|$.
Since $[\gamma]$ is also a polymer, we also have that $\|N(\gamma)\|=\|N([\gamma])\| \geq (1+\frac c{\sqrt d})\|[\gamma]\|$.
	 Thus,
\begin{align*}
	\sum_{\gamma : \|[\gamma]\|>d^4} \omega(\gamma)e^{3\|\gamma\|d^{-3/2}}
	 &\leq \sum_{a> d^4, b\ge (1+\frac c{\sqrt d})a}e^{3ad^{-3/2}}\sum_{\gamma\in \cG_\cD(a,b)}\omega(\gamma) \\
	 &\le 2^{2d} \sum_{a> d^4}  e^{\left(3d^{-3/2} - \frac{c\alpha^2}{k^2\sqrt d \log d}\right)a} \le e^{-cd^{17/6} \log d} ,
\end{align*}
	where the second inequality follows from \Cref{lem:Gab} and the last inequality uses~\eqref{eq:lambda-cond-k}.
\end{proof}

The proof of \cref{lem:Gab} is based on the following notion of an approximation of a polymer. We write $\bar\cE := \cO$ and $\bar\cO := \cE$.
An \textbf{approximation} is a tuple $(F_1,\dots,F_k,H_1,\dots,H_k)$ of sets  $F_i \subset \bar\cD_i$, $H_i \subset \cD_i$ such that each $H_i \cup (\bar\cD_i \setminus F_i)$ induces a subgraph of maximum degree at most $d^{2/3}$.
We write $(F_i,H_i)_i$ as shorthand for $(F_1,\dots,F_k,H_1,\dots,H_k)$.
We say that $(F_i,H_i)_i$ \textbf{approximates a polymer} $\gamma=(A_1,\dots,A_k)$, denoted $\gamma\approx (F_i,H_i)_i$, if for all $i$,
\begin{equation}\label{eq:approx-contain-def}
F_i \subset N(A_i) \qquad\text{and}\qquad H_i \supset [A_i] .
\end{equation}
We note that whether or not a given $(F_i,H_i)_i$ approximates $\gamma$ depends on $\gamma$ only through $N(\gamma)$.


%
%
\begin{lemma}\label{lem:Lemma5.1+5.2}
	There exists a family $\mathcal A $ of approximations with 
	$$|\mathcal A|\leq 2^d\exp\left(\frac {Ck(b-a)\log d}{d^{2/3}}\right),$$
such that every polymer in $\mathcal G_\cD(a,b)$ is approximated by an element in $\mathcal A$.
\end{lemma}
\begin{proof}
The proof is basically that of Lemma 5.1 and Lemma 5.2 in \cite{galvin2011threshold}. Indeed, the case $k=1$ follows directly from these two lemmas (see the paragraph following Lemma 5.3 there, and note that the maximum degree condition is written in the proof of Lemma 5.2). The case $k>1$ requires only minor modifications, which we now explain. The constructions in the proof of Lemma~5.1 are carried out separately for each coordinate $i \in [k]$, yielding the sets $(F',T_0,T'_0,T_1,T,L,\Omega)$ for each coordinate. The algorithmic procedure at the end of Lemma~5.1 and in Lemma~5.2 is also done separately for each coordinate. The only part of the argument which is not done separately for each coordinate is related to the enumeration in Lemma~5.1: (1) The argument that $F'$ is 4-linked and hence that $T$ is 8-linked works as written for the unions over all coordinates $i$ of the respective sets. (2) Given the union of the $T$s, we must choose the subsets $(T_0,T_1,\Omega)$ for each coordinate, and hence the terms in (5.9) other than $|Y|$ are raised to the power $k$.
\end{proof}

Recall the definition of $\alpha$ from \cref{lem:Gab}.

%
%
\begin{lemma}\label{lem:bdOnSum}
Assume~\eqref{eq:lambda-cond-k}.
Then for any approximation $(F_i,H_i)_i$,
\[ \sum_{\gamma\in \cG_\cD(a,b):\gamma\approx (F_i,H_i)_i} \omega(\gamma) \le \exp\left( -\frac{c(b-a)\alpha^2}{k^2\log d}\right) .\]
\end{lemma}

It is not hard to deduce \cref{lem:Gab} from \cref{lem:Lemma5.1+5.2} and \cref{lem:bdOnSum}.
\begin{proof}[Proof of \cref{lem:Gab}]
\cref{lem:Lemma5.1+5.2} and \cref{lem:bdOnSum} yield that
\[ \sum_{\gamma \in \cG_\cD(a,b)} \omega(\gamma) \le 2^d \exp\left(-\frac{c(b-a) \alpha^2}{k^2\log d} + \frac {Ck(b-a)\log d}{d^{2/3}} \right) .\]
To obtain the lemma (with different constants), it suffices to check that $\frac{(b-a)\alpha^2}{k^2\log d}$ is greater than $\frac {Ck(b-a)\log d}{d^{2/3}}$. This follows from~\eqref{eq:lambda-cond-k}.
\end{proof}

It remains to prove \cref{lem:bdOnSum}.
The proof will boil down to the case of $k=1$. Recall that in this case, a polymer is just a 2-linked subset $A$ of $\cE$ or $\cO$ whose closure has size at most $\frac34 2^{d-1}$. An approximation in this case is simply a pair $(F,H)$ of subsets of $Q_d$ satisfying the required properties.
The proof is split into two lemmas, each effective for a different size of $F$. We write $\cG_1(a,b)$ as shorthand for $\cG_{\cE}(a,b)$, which may also be identified with $\cG_{\cO}(a,b)$.

\begin{lemma}\label{lem:bdOnSumForLargeF}
Suppose that $\lambda \le \lambda_0$ and $\lambda(1-e^{-\beta})^2 \ge \frac{C\log d}d$. For any approximation $(F,H)$,
\[ \sum_{\gamma\in \cG_1(a,b):\gamma\approx (F,H)} \omega(\gamma) \le \binom{2db} {b-|F|} \exp\left(\frac b{d^4}- c\alpha (b-a)\right) .\]
\end{lemma}

\begin{lemma}\label{lem:bdOnSumForSmallF}
Suppose that $\lambda \le \lambda_0$ and $\lambda(1-e^{-\beta})^2 \ge \frac{C\log d}d$. For any approximation $(F,H)$,
\[ \sum_{\gamma\in \cG_1(a,b):\gamma\approx (F,H)} \omega(\gamma) \le \exp\left(-c\alpha (b-|F| - 3(b-a)d^{-1/3}) + b/d^4 + Cbd \alpha e^{-c\alpha d} \right) .\]
\end{lemma}

Before proving these two lemmas, let us show how they yield \cref{lem:bdOnSum}.

\begin{proof}[Proof of \cref{lem:bdOnSum}]
Consider a polymer $\gamma=(A_1,\dots,A_k)$ and a decorated polymer $\hat\gamma=(A_1,\dots,A_k,B_1,\dots,B_k)$ extending it.
Recall the definitions of their weights $\omega(\gamma)$ and $\omega(\hat\gamma)$ from~\eqref{eq:polymer-weight-def} and~\eqref{eq:dec-polymer-weight-def}. In this proof, we will need to keep track of the inverse temperature parameter $\beta$, and we write it explicitly in the notation of the weights $\omega_\beta(\gamma)$ and $\omega_\beta(\hat\gamma)$.
Using that $|E(A_1,B_1) \cup \cdots \cup E(A_k,B_k)| \ge \frac1k (|E(A_1,B_1)| + \cdots + |E(A_k,B_k)|)$, we see that
\[ \omega_\beta(\hat\gamma) \le \prod_{i=1}^k \frac{\lambda^{|A_i|+|B_i|}}{(1+\lambda)^{|N(A_i)|}} e^{-\frac \beta k |E(A_i,B_i)|} .\]
After applying this bound, the weight of a $\cD$-polymer factorizes over the $k$ components:
\[ \omega_\beta(\gamma) \le \sum_{B_1 \subset N(A_1),\dots,B_k \subset N(A_k)} \prod_{i=1}^k \frac{\lambda^{|A_i|+|B_i|}}{(1+\lambda)^{|N(A_i)|}} e^{-\frac \beta k |E(A_i,B_i)|} = \prod_{i=1}^k \omega_{\beta/k}((A_i)) .\]
Note that the term on the right-hand side refers to weights of polymers in the 1-system (that is, $(\cE)$-polymers or $(\cO)$-polymers).
Hence,
\[ \sum_{\gamma\in \cG_\cD(a,b),\gamma\approx (F_i,H_i)_i} \omega_\beta(\gamma) \le \sum_{\substack{a_1+\cdots+a_k=a\\b_1+\cdots+b_k=b}} \prod_{i=1}^k \sum_{\substack{\gamma\in \cG_{\cD_i}(a_i,b_i):\\\gamma\approx (F_i,H_i)}} \omega_{\beta/k}(\gamma) .\]
We claim that each term in the product satisfies
\[ \sum_{\gamma\in \cG_1(a_i,b_i),\gamma\approx (F_i,H_i)} \omega_{\beta/k}(\gamma) \le e^{-c(b_i-a_i) \alpha^2/k^2\log d + b_i/d^3} .\]
Indeed, if $b_i-|F_i| \le \frac{c\alpha(b_i-a_i)}{k\log d}+3(b_i-a_i)d^{-1/3}$ (for a small enough constant $c$), we apply \cref{lem:bdOnSumForLargeF} (note that $\alpha$ decreases by at most a factor $2k$ when $\beta$ decreases by a factor $k$) to deduce that the left-hand side is bounded by
\[ \binom{2db_i}{\frac{c\alpha(b_i-a_i)}{k\log d}+\frac{3(b_i-a_i)}{d^{1/3}}} \exp\left(\frac{b_i}{d^4} - \frac{c\alpha}k(b_i-a_i) \right) ,\]
which is seen to at most the claimed value by using the bound $\binom nm \le (en/m)^m$ and that $\alpha \ge \frac{Ck^2\log d}{d^{1/3}}$ and $b_i/(b_i-a_i) \le C\sqrt{d}$ by \cref{lem:isoperimetry} (otherwise $\cG_{\cD_i}(a_i,b_i)$ is empty). Otherwise, we apply \cref{lem:bdOnSumForSmallF} to deduce that the left-hand side is bounded by
\[ \exp\left(-\frac{c\alpha^2(b_i-a_i)}{k^2\log d} + \frac{b_i}{d^4} + Cb_i \tfrac{\alpha d}k e^{-c\alpha d/k}\right) ,\]
which is at most the claimed value since $1/d^4 + \frac{C\alpha d}k e^{-c\alpha d/k} \le 1/d^3$ using that $\frac{\alpha d}k \ge C\log d$.
Thus,
\[ \sum_{\gamma\in \cG_{k,m}(a,b),\gamma\approx (F_i,H_i)_i} \omega_\beta(\gamma) \le  \binom{a-k+1}{k-1} \binom{b-k+1}{k-1} e^{-c(b-a)\alpha^2 / k^2\log d + b/d^3} .\]
Each multinomial is at most $b^k$. Thus,
\[ \sum_{\gamma\in \cG_{k,m}(a,b),\gamma\approx (F_i,H_i)_i} \omega_\beta(\gamma) \le  e^{-c(b-a)\alpha^2 / k^2 \log d + b/d^3 + 2k \log b} , \]
and the lemma follows after noting that $b/d^3+2k\log b$ is negligible compared with $(b-a)\alpha^2 / k^2 \log d$ since $b/(b-a) \le C\sqrt{d}$, $b\ge a >d^4$ and the assumption on $\alpha$.
\end{proof}

\medbreak

The rest of this section is devoted to the proofs of \cref{lem:bdOnSumForLargeF} and \cref{lem:bdOnSumForSmallF}.
We first require some preparation in the form of a preliminary tool from~\cite{peled2020long} and an additional computation.
The tool, which we now present, is a method for bounding the weight of certain collections of configurations in the positive-temperature hard-core model on $Q_d$.

For a family $\cF$ of configurations $I \subset V(Q_d)$, define
\[ \tilde\omega(\cF):=\sum_{I\in \cF}\lambda^{|I|}e^{-\beta|E(I)|} .\]
For $\Psi \subset \{0,1\}^d$, define
\[ Z(\Psi) := \sum_{\psi \in \Psi} \lambda^{|\psi|} \left(1+ \lambda e^{-\beta |\psi|}\right)^d ,\]
where we identify an element $\psi \in \{0,1\}^d$ with a subset of $V(Q_d)$ (so that $|\psi|$ is the same as $|\psi^{-1}(1)|$).
Given $U \subset V \subset V(Q_d)$ and $f \in \{0,1\}^V$, we write $f_U$ for the restriction of $f$ to $U$, and $|f_U|$ for $|\{ u \in U : f(u)=1\}|$.
Note that $Z(\{0,1\}^d)$ is exactly the partition function $Z_1(K_{d,d},\lambda,\beta)$ of the positive-temperature hard-core model on the complete bipartite graph $K_{d,d}$.

Our analysis relies on a entropy tool from~\cite{peled2020long}, given in \cite[Lemma 7.3]{peled2020long}. This is general tool which applies to nearest-neighbor discrete spin systems on regular bipartite graphs (it was formulated for $\Z^d$, but the statement and proof holds more generality). Rather than stating the general lemma (which would require additional definitions), we formulate three special cases for the positive-temperature hard-core model, which we shall require in our proofs. We begin with the simplest of these:

\begin{lemma}[{\cite[Lemma 7.3]{peled2020long}}]\label{lem:weighted-shearer0}
	Let $T \subset Q_d$ be odd and let $\cF \subset \{0,1\}^T$.
	Then
	\[ \tilde\omega(\cF) \le \prod_{v \in T^{\text{odd}}} Z(\Psi_v)^{\frac1d} ,\]
	where $\Psi_v := \{ f_{N(v)} : f\in \cF\}$.
\end{lemma}

The above special case is obtained from \cite[Lemma 7.3]{peled2020long} by taking $S=T \cup N(T)$, $\mathbb S_u = \{0\}$ for all $u$, identifying $\cF$ as a subset of $\{0,1\}^T$ in the obvious way, and taking all $X_v$ to be trivial (constant) random variables.
Two additional special cases are obtained by either taking all $X_v$ to be $\1_{|f_{N(v)}|=0\}}$ or to be $(\1_{\{|f_{N(v)}|=0\}},\1_{\{|f_{N(v)}| \le s\}})$.

\begin{lemma}[{\cite[Lemma 7.3]{peled2020long}}]\label{lem:weighted-shearer}
	Let $T \subset Q_d$ be odd and let $\cF \subset \{0,1\}^T$ be a collection of sets containing no isolated odd vertices.
	Then
	\[ \tilde\omega(\cF) \le \prod_{v \in T^{\text{odd}}} Z(\Psi_v)^{\frac{p_v}d} \left(\tfrac 1{p_v}\right)^{\frac{p_v}d} \left(\tfrac1{1-p_v}\right)^{\frac{1-p_v}d} ,\]
	where $\Psi_v := \{ f_{N(v)} : f\in \cF,~ |f_{N(v)}|>0\}$ and $p_v := \Pr(|f_{N(v)}|>0)$ when $f$ is a random element of $\cF$ chosen according to weight $\tilde\omega$.
\end{lemma}

\begin{lemma}[{\cite[Lemma 7.3]{peled2020long}}]\label{lem:weighted-shearer3}
	Let $T \subset Q_d$ be odd, let $\cF \subset \{0,1\}^T$ and let $s>0$.
	Then
	\[ \tilde\omega(\cF) \le \prod_{v \in T^{\text{odd}}} Z(\Psi_v)^{\frac{p_v}d} Z(\Psi'_v)^{\frac{p'_v}d} (1+\lambda)^{1-p_v-p'_v} \left(\tfrac 1{p_v}\right)^{\frac{p_v}d} \left(\tfrac 1{p'_v}\right)^{\frac{p'_v}d} \left(\tfrac1{1-p_v-p'_v}\right)^{\frac{1-p_v-p'_v}d} ,\]
	where $\Psi_v := \{ f_{N(v)} : f\in \cF,~ 1 \le |f_{N(v)}| \le s \}$, $\Psi'_v := \{ f_{N(v)} : f\in \cF,~ |f_{N(v)}| > s \}$, $p_v := \Pr(1 \le |f_{N(v)}| \le s)$ and $p'_v := \Pr(|f_{N(v)}|> s)$, when $f$ is a random element of $\cF$ chosen according to weight $\tilde\omega$.
\end{lemma}

To make practical use of the above lemmas, we need to combine them with suitable bounds on $Z(\Psi)$. The required bound is given in the following lemma.
Define
\[ \ell_\Psi := \#\{ i \in [d] : \psi_i=0 \text{ for all }\psi \in \Psi \} .\]

\begin{lemma}\label{lem:Z-Psi-bound}
Suppose that $\lambda \le \lambda_0$ and $\lambda (1-e^{-\beta})^2 \ge \frac{C\log d}d$.
Then for $\Psi \subset \{0,1\}^d \setminus \{\bar0\}$,
\[ Z(\Psi)\leq (1+\lambda)^d e^{1/d^3- \frac12 \alpha \ell_\Psi} .\]
\end{lemma}

We will prove a stronger version of \cref{lem:Z-Psi-bound} which does not require $\lambda$ to be bounded. Define
\[ \bar\alpha := -\log\left(1-\frac{\alpha}{1+\lambda}\right) = \log\left(\frac{1+\lambda}{1+\lambda e^{-\beta}}\right) .\]
Note that when $\lambda$ is bounded, $c\bar\alpha \le \alpha \le C\bar\alpha$. Thus, the following immediately implies \cref{lem:Z-Psi-bound}.

\begin{lemma}\label{lem:Z-Psi-bound-gen}
Suppose that
\[ \frac{\lambda}{1+\lambda} \ge \frac{C\log d}d + \frac{C\log(\lambda d^4)}{\beta d} \qquad\text{and}\qquad \bar\alpha \ge \frac{C\log d}d + \frac {C\log(d(1+\lambda))\log (2+\lambda)}{\beta d} .\]
Then for $\Psi \subset \{0,1\}^d \setminus \{\bar0\}$, we have $Z(\Psi)\leq (1+\lambda)^d e^{1/d^3- \frac12 \bar\alpha \ell_\Psi}$.
\end{lemma}


\begin{proof}
Denote $\ell := \min\{\ell_\Psi,d/2\}$ and set $s := \frac{d-\ell}2 \frac{\lambda}{1+\lambda}$ and $s':= \frac {C\log (2+\lambda)}\beta$.

We begin with the case when $|\psi| \ge s$. We have
\begin{align*}
	\sum_{\psi \in \Psi : |\psi| \ge s} \lambda^{|\psi|} \left(1+\lambda e^{-\beta |\psi|}\right)^d &\le (1+\lambda e^{-\beta s})^d \sum_{\psi \in \Psi} \lambda^{|\psi|} \\
	& \le  (1+\lambda e^{-\beta s})^d (1+\lambda)^{d-\ell_\Psi}\\
		& \le  (1+\lambda)^de^{-\bar\alpha \ell}(1+\lambda e^{-\beta s})^d,
\end{align*}
where we used in the last inequality that $\ell_\Psi\ge \ell$ and $\bar\alpha\le \log (1+\lambda)$.

Next we deal with the case when $s'<|\psi|<s$. We have
\begin{equation}\label{eq:Z-Psi-bound-mid}
\begin{aligned}
\sum_{\psi \in \Psi : s'<|\psi|<s} \lambda^{|\psi|} \left(1+\lambda e^{- \beta |\psi|}\right)^d 
 &\le (1+\lambda e^{- \beta s'})^d \sum_{\psi \in \Psi : |\psi|<s} \lambda^{|\psi|} \\
 &\le (1+\lambda e^{- \beta s'})^d (1+\lambda)^{d-\ell} \cdot \Pr\left(\Bin\big(d-\ell,\tfrac \lambda{1+\lambda}\big) < s\right) \\
 &\le (1+\lambda e^{- \beta s'})^d (1+\lambda)^{d-\ell} e^{-s/2}\\
 &\leq (1+\lambda)^de^{-\bar\alpha \ell}e^{-s/2}(1+\lambda e^{- \beta s'})^d , 
\end{aligned}
\end{equation}
where we used a Chernoff bound (and the definition of $s$) in the third inequality.

Finally, we deal with the case when $|\psi| \le s'$. Using that $|\psi|\ge 1$ for all $\psi \in \Psi$, we have
\begin{equation}\label{eq:Z-Psi-bound-small}
\begin{aligned}
\sum_{\psi \in \Psi : |\psi| \le s'} \lambda^{|\psi|} \left(1+\lambda e^{- \beta |\psi|}\right)^d 
 &\le (1+\lambda e^{-\beta})^d \sum_{\psi \in \Psi : |\psi| \le s'} \lambda^{|\psi|} \\
 &\le (1+\lambda e^{-\beta})^d (d(1+\lambda))^{s'} \\
 &= (1+\lambda)^d \left(1 - \tfrac{\lambda(1-e^{-\beta})}{1+\lambda}\right)^d (d(1+\lambda))^{s'}\\
 &\leq (1+\lambda)^de^{-\bar\alpha \ell}e^{-\frac 12\bar\alpha d}(d(1+\lambda))^{s'}.
\end{aligned}
\end{equation}

Together we get that
\[ Z(\Psi) \le (1+\lambda)^d e^{-\bar\alpha \ell} \left( (1+\lambda e^{-\beta s})^d + e^{-s/2}(1+\lambda e^{- \beta s'})^d + e^{-\frac 12\bar\alpha d}(d(1+\lambda))^{s'} \right) .\]
Plugging in the definitions of $s$ and $s'$ and using the assumption of the lemma, one checks that the parenthesis term is at most $e^{1/d^3}$, and the lemma follows.
\end{proof}

We are now ready to prove \cref{lem:bdOnSumForLargeF} and \cref{lem:bdOnSumForSmallF}. We will use the rather simple fact (see, e.g., the proof of~\cite[Lemma~5.2]{galvin2011threshold}) that if $(F,H)$ approximates a polymer $A \in \cG_1(a,b)$, then
\begin{equation}\label{eq:approx-size-def}
|H| \le |F| + \frac{3(b-a)}{d^{1/3}} \qquad\text{and}\qquad |E(H,N(A) \setminus F)| \le 3(b-a)d^{2/3}.
\end{equation}

\begin{proof}[Proof of \cref{lem:bdOnSumForLargeF}]

The main step of the proof is to bound the sum of weights of polymers with a given closure. Specifically, we claim that for any $A' \subset \cE$ with $|A'|=a$ and $|N(A')|=b$, we have
\begin{equation}\label{eq:bdOnGivenA}
\sum_{A \subset \cE : [A]=A'} \sum_{B \subset N(A)} {\lambda^{|A|+|B|} e^{-\beta |E(A,B)|}} \le (1+\lambda)^be^{b/d^4-c\alpha (b-a)} .
\end{equation}
To see this, define $T := A'\cup N(A')$ and note that
\[ \cF := \{ (A,B) : [A]=A',~ B \subset N(A) \} \]
can be naturally identified with a subset of $\{0,1\}^T$.
Using this identification, $\tilde\omega(\cF)$ is precisely the sum on the left-hand side of~\eqref{eq:bdOnGivenA}, and \cref{lem:weighted-shearer0} and \cref{lem:Z-Psi-bound} yield that
\[ \tilde\omega(\mathcal F)\leq \prod_{v\in N(A')}Z(\Psi_v)^{1/d} \le ((1+\lambda)e^{1/d^4})^{|N(A')|} \prod_{v\in N(A')} e^{-c\alpha \ell_{\Psi_v}/d} , \]
where $\Psi_v$ is as in \cref{lem:weighted-shearer0} and $\ell_{\Psi_v}$ was defined before \cref{lem:Z-Psi-bound} (note that $\bar 0 \notin \Psi_v$ since $[A]=A'$ so that \cref{lem:weighted-shearer0} is applicable).
Observe that
\[ \sum_{v \in N(A')} \ell_{\Psi_v} = \sum_{v \in N(A')} |N(v)\setminus A'| = \sum_{v \in N(A')} (d- |N(v)\cap A'|) = d(b-a) .\]
This establishes~\eqref{eq:bdOnGivenA}.

The lemma will immediately follow from~\eqref{eq:bdOnGivenA} and a union bound, once we bound the number of possible closures of polymers under consideration, showing that
\[ \#\Big\{ [A] : A \in \cG_1(a,b),~ A \approx (F,H) \Big\} \le \binom{2db} {b-|F|} .\]
To see this, note that by~\eqref{eq:approx-size-def} and~\eqref{eq:approx-contain-def} (assuming there exists $\gamma \in \cG_1(a,b)$ such that $\gamma \approx (F,H)$),
\begin{equation}\label{eq:N(H)-small}
|N(H)| \le d |H| \le d(|F| + 3(b-a)d^{-1/3}) \le 2db .
\end{equation}
Since any $A$ under consideration has $F \subset N(A) \subset N(H)$ by~\eqref{eq:approx-contain-def}, and since $N(A)$ determines $[A]$, the closure of $\gamma$ is determined by $N(A) \setminus F$, which is a subset of $N(H)$ of size $|N(A)|-|F|$.
\end{proof}

\begin{proof}[Proof of \cref{lem:bdOnSumForSmallF}]

Define $T:=H\cup N(H)$ and identify
\[ \cF := \{ (A,B): A \in \cG_1(a,b),~ B \subset N(A),~ A \approx (F,H) \} \]
with a subset of $\{0,1\}^T$ in the natural way. Observe that with this identification, $\cF$ is a collection of subsets of $T$
containing no isolated odd vertices (since $B \subset N(A)$) and
\[ \sum_{\gamma\in \cG_1(a,b),\gamma\approx (F,H)} \omega(\gamma) = \frac{\tilde\omega(\cF)}{(1+\lambda)^b} ,\]
Our goal now becomes to bound $\tilde\omega(\cF)$.
By \cref{lem:weighted-shearer},
\[ \tilde\omega(\mathcal F)\leq \prod_{v\in N(H)} Z(\Psi_v)^{\frac{p_v}d} \left(\tfrac 1{p_v}\right)^{\frac{p_v}d} \left(\tfrac1{1-p_v}\right)^{\frac{1-p_v}d} ,\]
where $\Psi_v$ and $p_v$ are as in \cref{lem:weighted-shearer}.
By \cref{lem:Z-Psi-bound} (and writing $\ell_v := \ell_{\Psi_v}$),
\[ Z(\Psi_v)^{1/d} \le (1+\lambda)e^{1/d^4} e^{-c\alpha \ell_v/d} .\]
Splitting $e^{-c\alpha \ell_v/d}$ into the product of two factors $e^{-c\alpha \ell_v/d}$ (with a modified constant $c$), we get
\begin{equation}\label{eq:Z-bound}
\tilde\omega(\cF) \le \left[\prod_{v\in N(H)} \left((1+\lambda)e^{1/d^4} e^{-c\alpha \ell_v/d}\right)^{p_v} \right] \cdot \left[ \prod_{v\in N(H)} e^{-c\alpha \ell_v p_v/d} \left(\tfrac 1{p_v}\right)^{\frac{p_v}d} \left(\tfrac1{1-p_v}\right)^{\frac{1-p_v}d} \right] .
\end{equation}
To obtain the lemma, it thus suffices to show that the first term on the right-hand side of~\eqref{eq:Z-bound} is at most $(1+\lambda)^b e^{b/d^4} e^{-c\alpha (b-|F| - 3(b-a)/d^{1/3})}$ and that the second term is at most $e^{bd \alpha e^{-c\alpha d}}$.

\smallskip

Consider the first term in~\eqref{eq:Z-bound}. The desired bound will follow once we show that
\[ \sum_{v \in N(H)} p_v = b \qquad\text{and}\qquad \sum_{v \in N(H)} p_v \ell_v \ge bd-|F|d - 3(b-a) d^{2/3} .\]
Since $p_v=\Pr(v \in N(A))$, both sums can be seen as expectations, namely, 
\[ \sum_{v \in N(H)} p_v = \E |N(A)| \qquad\text{and}\qquad \sum_{v \in N(H)} p_v \ell_v = \E \sum_{v \in N(A)} \ell_v .\]
Since every $A$ under consideration (namely, $(A) \in \cG_1(a,b)$ such that $(A) \approx (F,H)$) satisfies that $|N(A)|=b$, we have that $\sum_{v \in N(H)} p_v = b$. We claim that every such $A$ also satisfies that $\sum_{v \in N(A)} \ell_v \ge bd-|F|d - (b-a)d^{2/3}$. To see this, observe first that $\ell_v = \ell_{\Psi_v} \ge |N(v)\setminus H|$, so that $\sum_{v \in N(A)} \ell_v \ge |E(N(A),\cE \setminus H)|$. We have $|E(N(A),\cE)| = db$ and, by~\eqref{eq:approx-size-def},
\[ |E(N(A),H)| = |E(F,H)| + |E(N(A) \setminus F,H)| \le |F|d + 3(b-a)d^{2/3} .\]
We conclude that $|E(N(A),\cE \setminus H)| \ge db - |F|d - 3(b-a)d^{2/3}$. This establishes the desired bound on the first term in~\eqref{eq:Z-bound}. 

\smallskip

Consider now the second term in~\eqref{eq:Z-bound}.
Since $|N(H)| \le 2db$ by~\eqref{eq:N(H)-small}, it suffices to show that each term in the product is at most $e^{C\alpha e^{-c\alpha d}}$, or after taking logarithms, that
\[ \frac{H(p_v)}d - \frac{c_1\alpha \ell_v p_v}{d} \le C\alpha e^{-c_2\alpha d} .\]
This clearly holds when $p_v=1$. We may thus assume that $p_v<1$. In particular, $v \notin F$ since $|f_{N(v)}| = 0$ ($v \notin N(A)$) for some $f \in \cF$. Thus, by the definition of an approximation, at most $d^{2/3}$ neighbors of $v$ belong to $H$. That is, $\ell_v \ge |N(v) \setminus H| \ge d - d^{2/3} \ge 2d/3$. 
Thus, it suffices to show that
\[ \frac{H(p_v)}d - c_3\alpha p_v \le C\alpha e^{-c_2\alpha d} .\]
The left-hand side is negative when $p_v > e^{-c_3\alpha d/2}$ (when $p_v \ge 1/e$ we use that $H(p_v) \le \log 2$ and $\alpha \ge C/d$, and otherwise we use that $H(p_v) \le 2p_v \log \frac1{p_v}$). We may thus assume that $p_v \le e^{-c_3\alpha d/2} \le 1/e$, in which case, using that $x\log(1/x)$ is increasing on $(0,1/e)$, we have
\[ \frac{H(p_v)}d - c_3\alpha p_v \le \frac{H(p_v)}d \le \frac{2p_v \log \frac1{p_v}}d \le c_3 \alpha e^{-c_3\alpha d/2} . \qedhere \]
\end{proof}

\subsection{Bounding the weight of non-polymer configurations}\label{sec:bad-configs}
In this section, we prove \cref{lem:bad-configs}.
The $k=1$ case is given in the following lemma.

\begin{lemma}\label{lem:bad-configs-k=1}
Suppose that $\lambda \le \lambda_0$ and that $\lambda(1-e^{-\beta}) \ge \frac{C\log d}{d^{1/3}}$. Then
\[ \sum_{I \subset V(Q_d): |[I \cap \cE]|,|[I \cap \cO]| > \frac34 \cdot 2^{d-1}} \lambda^{|I|} e^{- \beta |E(I)|} \le Z_1 \cdot O(\exp(-2^d/d)) .\]
\end{lemma}
\begin{proof}
Define $m:=2^d d^{-2/3}$, $s := \frac d2 \frac{\lambda}{1+\lambda}$ and
\[ \cI_\cO := \big\{ I \subset V(Q_d) : \text{there are at least $m$ vertices $v \in \cO$ such that $1 \le |N(v) \cap I| \le s$} \big\} .\]
Define $\cI_\cE$ similarly.
Let $\cI'$ be the set of $I \notin \cI_\cE \cup \cI_\cO$ such that $|[I \cap \cE]|,|[I \cap \cO]| \ge \frac34 \cdot 2^{d-1}$. It suffices to show that each of $\tilde\omega(\cI_\cE)$, $\tilde\omega(\cI_\cO)$ and $\tilde\omega(\cI')$ is at most $Z_1 \cdot O(\exp(-2^d/d))$.

Let us begin with $\cI_\cO$ (the argument for $\cI_\cE$ is the same).
Let $\Psi_v$, $\Psi'_v$, $p_v$ and $p'_v$ be defined as in \cref{lem:weighted-shearer3}.
By \cref{lem:Z-Psi-bound},
\[ Z(\Psi'_v) \le (1+\lambda)^d e^{1/d^3} .\]
Denoting $\alpha := \lambda(1-e^{-\beta})$ and $s':= \frac {C\log (2+\lambda)}\beta$, by~\eqref{eq:Z-Psi-bound-mid} and~\eqref{eq:Z-Psi-bound-small} (taking $\ell=0$ there and recalling that $c\bar\alpha \le \alpha \le C\bar\alpha$ when $\lambda$ is bounded),
\[ Z(\Psi_v) \le (1+\lambda)^d e^{-s/2}(1+\lambda e^{- \beta s'})^d + (1+\lambda)^d e^{-\frac 12\alpha d}(d(1+\lambda))^{s'} \le (1+\lambda)^d e^{-c\alpha d} ,\]
where the second inequality follows from plugging in the definitions of $s$ and $s'$ and using the assumption on $\alpha$.
Thus, by \cref{lem:weighted-shearer3},
\[ \tilde\omega(\cI_\cO) \le (1+\lambda)^{2^{d-1}} \prod_{v \in \cO} e^{-c\alpha p_v} e^{p'_v / d^4} \left(\tfrac 1{p_v}\right)^{\frac{p_v}d} \left(\tfrac 1{p'_v}\right)^{\frac{p'_v}d} \left(\tfrac1{1-p_v-p'_v}\right)^{\frac{1-p_v-p'_v}d} .\]
Since $Z_1 \ge (1+\lambda)^{2^{d-1}}$, $\sum_{v \in \cO} p'_v \le 2^{d-1}$ and $\sum_{v \in \cO} p_v \ge m$ (by the definition of $\cI_\cO$), to deduce that $\tilde\omega(\cI_\cO) \le Z_1 \cdot O(\exp(-2^d/d))$, it suffices to show that
\[  e^{-c\alpha m} e^{2^d/d^4} \left[\prod_{v \in \cO} \left(\tfrac 1{p_v}\right)^{p_v} \left(\tfrac 1{p'_v}\right)^{p'_v} \left(\tfrac1{1-p_v-p'_v}\right)^{1-p_v-p'_v} \right]^{\frac1d} \le O(\exp(-2^d/d))  .\]
This indeed holds since $\alpha \ge Cd^{-1/3} \log d$, $m=2^d d^{-2/3}$ and each term in the product is at most 3 (since it is the exponential of the entropy of a random variable which takes at most 3 values).

We now proceed to bound $\tilde\omega(\cI')$.
We claim that each $I \in \cI'$ satisfies that $|E(I)| \ge ms$.
This will yield the lemma since
\[ \tilde\omega(\cI') \le (1+\lambda)^{2^d} e^{-\beta ms} = e^{O(2^d)} e^{-\Omega(2^d \log d)} = O(\exp(-2^d)) .\]
Let $I \in \cI'$. Using that $[I \cap \cE]$ and $[I \cap \cO]$ are each of size at least $\frac34 \cdot 2^{d-1}$, it follows that $[I]$ contains at least half of the edges of the hypercube, i.e., $|E([I])| \ge \frac d2 2^{d-1}$. Since the graph spanned by $E([I])$ has maximum degree at most $d$, we see that
\[ |[I] \cap N([I]) \cap \cO]| \ge \tfrac1d |E([I])| \ge 2^{d-2} .\]
Using that $N(I)=N([I])$ and $I \notin \cI_\cO$,
\[ |I \cap N(I) \cap \cE| \ge (|[I] \cap N(I) \cap \cO| - m) \tfrac sd \ge c\lambda 2^d \ge 2m .\]
Using that $I \notin \cI_\cE$, we conclude that $|E(I)| \ge (|I \cap N(I) \cap \cE|-m)s \ge ms$.
\end{proof}

We are now ready to prove \cref{lem:bad-configs}.

\begin{proof}[Proof of \cref{lem:bad-configs}]
Since $|E(I_1) \cup \cdots \cup E(I_k)| \ge \frac1k (|E(I_1)| + \cdots + |E(I_k)|)$, the sum in the lemma is easily bounded by
\[ k Z^{k-1}_1(\tfrac \beta k) \sum_{I \subset V(Q^d): |[I \cap \cE]|,|[I \cap \cO]| > \frac34 \cdot 2^{d-1}} \lambda^{|I|} e^{-\tfrac \beta k |E(I)|} ,\]
where $Z_1(\beta/k)$ is shorthand for $Z_1(Q_d,\lambda,\beta/k)$.
Using \cref{lem:bad-configs-k=1}, the above is bounded by
\[ k Z^k_1(\tfrac \beta k) O(\exp(-2^d/d)) .\]
Thus, it suffices to show that $Z^k_1(\tfrac \beta k) \le Z_k(\beta) e^{O(2^d/d^2)}$. By \cref{lem:weighted-shearer0,lem:Z-Psi-bound}, we obtain that
\[ Z_1(\tfrac \beta k) \le (1+\lambda)^{2^{d-1}} e^{2^{d-1}/d^4} \le Z_1(\beta) e^{2^{d-1}/d^4} .\]
Since $Z_1^k(\beta) \le Z_k(\beta)$, the lemma follows.
\end{proof}

\subsection{Improved bounds for small clusters}\label{sec:improved-bounds}

\cref{lem:Lemma15} gives a bound on the weight of clusters of size at least $n$, for any value of $n$, which may depend on $d$. For fixed $n$, the bound obtained in this manner is not optimal. In this section, we provide some improvements on this (namely, \cref{lem:cluster-tail,lem:cluster-tail-improved}). In particular, \cref{lem:cluster-convergence} follows immediately from \cref{lem:cluster-tail-improved}.

\begin{lemma}\label{lem:cluster-tail}
Assume~\eqref{eq:lambda-cond-k}. Then for any fixed $n \ge 1$,
\[ \sum_{\Gamma\in \cC_\cD : \|\Gamma\|\geq n}|\omega(\Gamma)|e^{\|\Gamma\| d^{-3/2}} = O\left( 2^d d^{2n-2} \lambda^n \tilde\alpha_k^{nd} \right) .\]
\end{lemma}

The exponential term $\tilde\alpha_k^{nd}$ is not optimal when $n$ is not a multiple of $k$. An improved exponential term is provided in \cref{lem:cluster-tail-improved} below.

\begin{proof}
We first apply \cref{lem:Lemma15} to obtain that
\[ \sum_{\Gamma \in \cC_\cD : \|\Gamma\|\ge n+1} |\omega(\Gamma)|e^{\|\Gamma\| d^{-3/2}} \le 
d^{7(n+1)-3/2} 2^{d-1} \tilde\alpha_k^{d(n+1)-3(n+1)^2} .\]
Since the right-hand side is $O\left(2^d d^{2n-2} \lambda^n \tilde\alpha_k^{dn} \right)$, it remains only to bound the contribution from clusters of size $n$.
	Define the support of $\Gamma$ to be $S(\Gamma) := \bigcup_{\gamma \in \Gamma} S(\gamma)$ and note that $|S(\Gamma)| \le n$.
	The number of ways to choose the support of a cluster of size $n$ is at most $O(2^d d^{2n-2})$ by \cref{lem:Lemma13}. For any $S$ of size at most $n$, there are at most a constant (depending on $k$ and $n$, which are fixed) number of clusters with support $S$.
	It follows that there are at most $O(2^d d^{2n-2})$ clusters of size $n$. By \cref{lem:weight-of-polymer}, the absolute weight of any such cluster $\Gamma$ satisfies
	\[ |\omega(\Gamma)| = |\phi(H_\Gamma)| \prod_{\gamma \in \Gamma}\omega(\gamma) \le C(n) \lambda^{\|\Gamma\|} \tilde\alpha_k^{\|N(\Gamma)\|} .\]
By \cref{lem:isoperimetry},
we have that $\|N(\Gamma)\| \ge dn - 2n^2$, so that
\[ \sum_{\Gamma \in \cC_\cD : \|\Gamma\|=n} |\omega(\Gamma)| \le O\left(2^d d^{2n-2} \lambda^n \tilde\alpha_k^{dn-2n^2} \right) ,\]
which is $O\left(2^d d^{2n-2} \lambda^n \tilde\alpha_k^{dn} \right)$, since $\lambda$ bounded implies that $\tilde\alpha_k$ is bounded away from zero.
\end{proof}

For a polymer $\gamma=(A_1,\dots,A_k)$, we define its \textbf{span} to be
\begin{equation}\label{eq:span-def}
\spn(\gamma) := \{ i \in [k] : A_i \neq \emptyset \} .
\end{equation}
We define the span of a cluster $\Gamma=(\gamma_1,\dots,\gamma_n)$ to be $\spn(\Gamma) := \spn(\gamma_1) \cup \cdots \cup \spn(\gamma_n)$.

\begin{lemma}\label{lem:cluster-tail-improved}
Assume~\eqref{eq:lambda-cond-k}. Then for any fixed $n$, writing $n=ak+b$ for $a \ge 0$ and $1 \le b \le k$, we have
\[ \sum_{\Gamma\in \cC_\cD : \|\Gamma\|\geq n}|\omega(\Gamma)|e^{\|\Gamma\| d^{-3/2}} = O\left( 2^d d^{2n-2} \lambda^n \alpha_k^{ad}\alpha_b^d \right) .\]
Furthermore, for any $1 \le \ell \le k$,
\[ \sum_{\Gamma\in \cC_\cD : |\spn(\Gamma)|\geq \ell}|\omega(\Gamma)|e^{\|\Gamma\| d^{-3/2}} = O\left( 2^d \lambda^\ell \alpha_\ell^d \right) .\]
\end{lemma}
\begin{proof}

	We prove the first statement by inverse induction on $b$ (with $a$ fixed). The base case $b=k$ of the induction is precisely \Cref{lem:cluster-tail}. For the induction step, let $b<k$ and assume that the statement holds for $b+1$, so that, in particular,
		\[ \sum_{\Gamma\in \cC_\cD : \|\Gamma\|>n}|\omega(\Gamma)|e^{\|\Gamma\| d^{-3/2}} = O\left( 2^d d^{2n} \lambda^{n+1} \alpha_k^{ad}\alpha_{b+1}^d \right).\]
		Using that $\lambda(1-e^{-\beta})=\omega(\frac{\log d}d)$ and $\lambda$ is bounded, we see that the right-hand side is $o(2^d\lambda^n\alpha_k^{ad}\alpha_b^d)$.
	It remains to show that
	\[ \sum_{\Gamma\in \cC_\cD : \|\Gamma\|= n}|\omega(\Gamma)| = O\left( 2^d d^{2n-2} \lambda^n \alpha_k^{ad}\alpha_b^d \right) .\]
	Since there are $O(2^dd^{2n-2})$ clusters of size $n$, it suffices to show that 
	\[ \max_{\|\Gamma\|=n}|\omega(\Gamma)| = O(\alpha_k^{ad} \alpha_b^d) .\]
	As we have seen in the proof of the previous lemma, \cref{lem:weight-of-polymer} and \cref{lem:isoperimetry} imply that $|\omega(\Gamma)| \le C(n)\lambda^n \tilde\alpha_k^{dn-2n^2}$ for any cluster $\Gamma$ such that $\|\Gamma\|=n$. We require a stronger bound here (recall that $\tilde\alpha_\ell$ is increasing in $\ell$, so that $\tilde\alpha_k^{n} \ge \tilde\alpha_k^{ak}\tilde\alpha_b^{b}=\alpha_k^{a} \alpha_b$). To obtain the required bound, it suffices to improve \cref{lem:weight-of-polymer} to show that any polymer $\gamma$ with $\|\gamma\|=n$ satisfies
	\[ \omega(\gamma) \le \lambda^n \alpha_k^{ad} \alpha_b^d .\]
	(Actually we need to use this for polymers of size at most $n$, but we prefer not to introduce new notation and just continue using $n=ak+b$.)
	Following the proof of \cref{lem:weight-of-polymer} (and in the notation of that lemma), we have that $\omega(\gamma) \le \lambda^n (1+\lambda)^{-\|N(\gamma)\|} \bar\omega(\gamma)$ and
	\begin{align*}
	\bar\omega(\gamma) = \prod_v \delta_{m_v(\gamma)}.
	\end{align*}
	It suffices to show that $\bar\omega(\gamma) \le O(\delta_k^{ad}\delta_b^d)$.
	We only keep in the product those $v$ which have a unique neighbor in the support of $\gamma$ (these are all but $O(1)$ many vertices). For each $u$ in the support of $\gamma$, we consider the product $\prod_v \delta_{m_v}(\gamma)$ over all $v$ adjacent to $u$ which are not adjacent to any other vertex in the support. Then all $m_v(\gamma)$ in the product equal $n_u(\gamma) := |\{i: u \in A_i\}|$. Thus,
	\[ \bar\omega(\gamma) \le \left(\prod_u \delta_{n_u(\gamma)}\right)^{d-O(1)} .\]
	It thus suffices to show that $\prod_u \delta_{n_u(\gamma)} \le \delta_k^a \delta_b$. This will follow from the fact that $\delta_i \delta_j \le \delta_{i-1}\delta_{j+1}$ for any $0<i \le j<k$.
	Indeed, starting from the set of numbers $\{ n_u(\gamma) \}$, and repeatedly choosing a pair of numbers $\{i,j\}$ such that $0<i\le j<k$ and replacing it with the pair $\{i-1,j+1\}$, we eventually reach a set of numbers which are all 0 or $k$, except perhaps one number. Since their sum is preserved throughout this process, there must be exactly $a$ numbers which are $k$ and a single one which is $b$. Since this process only increased the product, we conclude that $\prod_u \delta_{n_u(\gamma)} \le \delta_k^a \delta_b$.

	We now prove the second part. By the the first part, the contribution from clusters of size larger than $\ell$ is negligible. We thus only need to show that
	\[ \sum_{\Gamma\in \cC_\cD : |\spn(\Gamma)|=\|\Gamma\|=\ell}|\omega(\Gamma)| = O\left( 2^d \lambda^\ell \alpha_\ell^d \right) ,\]
	There are $O(2^d)$ clusters in the sum (note that any cluster in the sum has a support which is a singleton), and by what we have just shown, each satisfies that $|\omega(\Gamma)| \le C(\ell)\lambda^\ell \alpha_\ell^d$.
\end{proof}

\section{The moments}\label{sec:moments}

In this section, we state and prove extensions of the main theorems stated in \cref{sec:intro}.
The proofs rely on \cref{thm:ClusterEx,lem:cluster-tail,lem:cluster-tail-improved}.

Denote $Z$ as shorthand for the partition function $Z(Q_{d,p},\lambda)$ of the hard-core model at fugacity $\lambda$ on the random subgraph $Q_{d,p}$. 
Recall that $p=1-e^{-\beta}$ and recall from \cref{sec:convergence} that we denote
\[ \alpha_\ell := \frac{1 + ((1+\lambda)^\ell-1)e^{-\beta}}{(1+\lambda)^\ell} .\]

The next theorem is an extension of \Cref{thm:expectation} to the hard-core model.

\begin{thm}\label{thm:expectation-general}
Suppose that $\lambda \le \lambda_0$ and $\lambda p \ge \frac{C\log d}{d^{1/3}}$. Then
\[ \E Z =
 	2(1+\lambda)^{2^{d-1}} \exp\left[ \tfrac\lambda 2 2^d \alpha_1^d + \left(a\tbinom d2 -\tfrac14 \right)\lambda^2 2^d \alpha_1^{2d} + O\left(d^4 \lambda^3 2^d \alpha_1^{3d}\right) \right], \]
where
\[ a := \frac{(1+\lambda)^2(1+\lambda (1-p)^2)^2}{4(1+\lambda (1-p))^4} - \frac14  .\]
\end{thm}

The next theorem gives a formula for the moments of $Z$, and in particular yields \Cref{thm:variance} and the first part of \Cref{thm:moments}.
\begin{thm}\label{thm:moments-general}
Let $k \ge 2$. Suppose that $\lambda \le \lambda_0$ and $\lambda p \ge \frac{Ck^2\log d}{d^{1/3}}$. Then
 \begin{align*}
 \frac{\E Z^k}{(\E Z)^k} =
 	2^{-k} \sum_{m=0}^k &\tbinom km \exp\left[ \tfrac{\lambda^2}2\big(\tbinom m2 + \tbinom{k-m}2\big) 2^d (\alpha_2^d - \alpha_1^{2d})\right] \\ &\quad\cdot \exp\left[\tfrac{dp(1-p)\lambda^4}{2(1+\lambda-p\lambda)^2} m(k-m)2^d \alpha_1^{2d} + O(\lambda^3 d^4 2^d \epsilon_k^d) \right] ,
 \end{align*}
 	where $\epsilon_2 := \alpha_1 \alpha_2$ and $\epsilon_k := \alpha_3$ for $k \ge 3$.
\end{thm}

Finally, the next theorem gives a formula for the  central moments of $Z$, and in particular yields the second part of \Cref{thm:moments}.
\begin{thm}\label{thm:central-moments}
Let $k \ge 3$. Suppose that $\lambda \le \lambda_0$. Then for $\frac{(1+\lambda)^2}{2\lambda(2+\lambda)} + \omega(\frac1d) \le p \le 1 - 2^{-d/3 + \omega(\log d)}$, we have
\begin{equation}\label{eq:moment-asymptotics}
\frac{\E(Z-\E Z)^k}{(\E Z)^k} = \begin{cases}
(\frac \lambda 2)^k 2^d\alpha_k^d + \sigma^k(k-1)!! + o(\sigma^k + 2^d \alpha_k^d) &\text{if $k$ is even}\\
(\frac \lambda 2)^k 2^d\alpha_k^d  + o(\sigma^k + 2^d \alpha_k^d) &\text{if $k$ is odd}
\end{cases},
\end{equation}
where
\[ \sigma^2 := \frac14 2^d \lambda^2 \left( \alpha_2^d + \left(\frac{p(1-p) \lambda^2 d}{(1+\lambda-p\lambda)^2} - 1\right) \alpha_1^{2d}\right) .\]
Furthermore, for $p=\frac{(1+\lambda)^2}{2\lambda(2+\lambda)} \pm O(\frac1d)$, the left-hand side of~\eqref{eq:moment-asymptotics} is $\Theta(1)$, and for $p \ge \frac{(1+\lambda)^2}{2\lambda(2+\lambda)}$, it is $O((2\alpha_2)^{dk/2} + 2^d \alpha_k^d)$.
\end{thm}

Observe that $\frac{(1+\lambda)^2}{2\lambda(2+\lambda)}$ is always greater than $\frac12$ and it is greater than 1 when $\lambda<\sqrt 2 -1$, so that the assumptions of (say the first part of) \cref{thm:central-moments} can only hold when $p>\frac12$ and $\lambda>\sqrt 2 -1$.
We remark that $\alpha_2=\frac12$ when $p=\frac{(1+\lambda)^2}{2\lambda(2+\lambda)}$, and that $\sigma$ is of constant order when $p=\frac{(1+\lambda)^2}{2\lambda(2+\lambda)} \pm O(\frac1d)$ and is $o(1)$ for larger $p$.

The following is an extension of \Cref{thm:normal} and is a corollary of the previous theorems.
Recall that the $k$-th moment of a standard normal random variable is $(k-1)!!$ for $k$ even and zero for $k$ odd.

\begin{cor}
\label{cor:normal-general}
Suppose that $\lambda \le \lambda_0$, $p=1-o(1)$ and $\frac{(1+\lambda)^2}{2\lambda(2+\lambda)} + \omega(\frac1d) \le p \le 1 - 2^{-d/3 + \omega(\log d)}$.
Then for any fixed $k \ge 1$,
\[ \E\left(\frac{Z-\E Z}{\sqrt{\Var(Z)}}\right)^k \to \E N^k \qquad\text{as }d \to \infty, \]
where $N$ is a standard normal random variable. In particular, the standardization of $Z$ converges in distribution to $N$.
\end{cor}

Before going into the proofs, the reader may find it helpful to recall the computations done in \cref{sec:computational-examples}.
To ease notation throughout the section (recall the error term from \cref{thm:ClusterEx}), we write $a \simeq b$ as shorthand for $a/b=1+O(\exp(-2^d/d^4))$.
We also sometimes write $\cC_{k,m}$ for $\cC_\cD$ where $\cD \in \{\cE,\cO\}^k$ is the vector given by $\cD_i=\cE$ for $1 \le i \le m$ and $\cD_i=\cO$ otherwise (recall that this would be essentially the same for any $\cD$ with $m$ or $k-m$ coordinates equal to $\cE$; see the last remark in \cref{sec:remarks}).

\begin{proof}[Proof of \cref{thm:expectation-general}]
By \cref{prop:relation,thm:ClusterEx},
\begin{equation}\label{eq:formula-for-Z1}
\E Z = Z_1 \simeq 2(1+\lambda)^{2^{d-1}} \exp\left(\sum_{\Gamma \in \cC_{1,0}} \omega(\Gamma)\right) .
\end{equation}
Thus, we are just left with computing the cluster expansion series. For our desired accuracy, we will compute the exact contribution from clusters of size 1 and 2, and only upper bound the contribution from larger clusters. The latter is done by using \cref{lem:cluster-tail}, which gives that
\[ \sum_{\Gamma \in \cC_{1,0} : \|\Gamma\|\ge3} |\omega(\Gamma)| = O(d^4) 2^d \lambda^3 \alpha_1^{3d} .\]
There is only one type of cluster of size 1, that consisting of a single polymer of size 1, and its Ursell function is 1. There are $2^{d-1}$ such clusters and each has weight $\lambda \alpha_1^d$ (recall the computation of scenario~I in \cref{sec:computational-examples}).
Thus,
\[ \sum_{\Gamma \in \cC_{1,0} : \|\Gamma\|=1} \omega(\Gamma) = 2^{d-1} \lambda \alpha_1^d .\]
There are two types of clusters of size 2: those consisting of a single polymer of size 2, whose Ursell function is 1, and those consisting of two polymers of size 1, whose Ursell function is $-\frac12$.
There are $2^{d-2} \binom d2$ clusters of the former type and $2^{d-1} (1+\binom d2)$ of the latter type. The former type clusters have weight $\lambda^2 \alpha_1^{2d-4} \frac{(1+\lambda e^{-2\beta})^2}{(1+\lambda)^2}$ (recall scenario~II in \cref{sec:computational-examples}) and the latter have weight $-\frac12 \lambda^2 \alpha_1^{2d}$.
Thus,
\[ \sum_{\Gamma \in \cC_{1,0} : \|\Gamma\|=2} \omega(\Gamma) = 2^{d-2} \lambda^2 \alpha_1^{2d} \left( \binom d2 \left( \frac{(1+\lambda)^2(1+\lambda e^{-2\beta})^2}{(1+\lambda e^{-\beta})^4} - 1\right) - 1 \right). \]
Putting these together yields the theorem.
\end{proof}

\begin{remark}\label{re:betterAccuracy}
	The proof of \cref{thm:expectation-general} can be modified to obtain better accuracy. Specifically, for any fixed $n \ge 1$, using \cref{lem:cluster-tail} and~\eqref{eq:formula-for-Z1}, we see that
	\begin{equation}
	 \E Z = Z_1 = 2(1+\lambda)^{2^{d-1}} \exp\left(\sum_{\Gamma \in \cC_{1,0}: \|\Gamma\|<n} \omega(\Gamma) + O(d^{2n-2}) 2^d \lambda^n \alpha_1^{nd} \right) .
	\end{equation}
	Thus, by computing the contribution to the cluster expansion from clusters of size less than $n$ (as we have done for clusters of size 1 and 2 in the proof above), one may obtain an explicit formula for $\E Z$. This will show that $\sum_{\Gamma \in \cC_{1,0}: \|\Gamma\|<n} \omega(\Gamma)$ has the form $2^d \sum_{i=1}^{n-1} f_i(d,\lambda,p) \alpha_1^{di}$, where $f_1,\dots,f_{n-1}$ are polynomials in $d$ (with $f_i$ having degree $2i-2$) with coefficients depending on $\lambda$ and $p$. For example, to obtain \cref{thm:expectation-general} we calculated the first two polynomials, showing that $f_1 = \frac\lambda2$ and $f_2 = \lambda^2(a\binom d2 - \frac14)$ with $a=a(p)$ as in \cref{thm:expectation-general}.
\end{remark}

We now move on to compute the higher moment of $Z$.
For this, as well as for the central moments later on, it is useful for us to be able to view $k'$-systems as embedded in the $k$-system when $k'<k$. We make this precise via the notion of the span of a polymer/cluster; recall the definition from~\eqref{eq:span-def}.
Observe that the set of clusters $\Gamma \in \cC_{k,m}$ whose span is contained in a given subset $S \subset [k]$ can be identified with $\cC_{|S|,|S \cap \{1,\dots,m\}|}$. For example, the set of clusters $\Gamma \in \cC_{k,m}$ whose span is a given singleton is identified with $\cC_{1,0}$, which is itself identifiable by even-odd symmetry with $\cC_{1,1}$. This will allow us to easily compare the $k$-th moment of $Z$ with the $k$-th power of its expectation.



\begin{proof}[Proof of \cref{thm:moments-general}]
By \cref{prop:relation,thm:ClusterEx},
\[ \E Z^k = Z_k \simeq (1+\lambda)^{k 2^{d-1}} \sum_{m=0}^k \binom km \exp\left( \sum_{\Gamma \in \cC_{k,m}} \omega(\Gamma) \right) .\]
Since for any $m$ and $i\in[k]$, we have that $\sum_{\Gamma \in \cC_{k,m} : \spn(\Gamma)=\{i\}} \omega(\Gamma) = \sum_{\Gamma \in \cC_{1,0}} \omega(\Gamma)$, and using~\eqref{eq:formula-for-Z1}, we obtain that
\begin{equation}\label{eq:moment-ratio-cluster-expansion}
\frac{\E Z^k}{(\E Z)^k} \simeq 2^{-k} \sum_{m=0}^k \binom km \exp\left( \sum_{\Gamma \in \cC_{k,m},~|\spn(\Gamma)|>1} \omega(\Gamma) \right) .
\end{equation}
Thus, similarly to before, we are left with computing the cluster expansion series to some desired accuracy. We will compute the exact contribution from clusters of size 2, and upper bound the contribution from larger clusters (note that all clusters in the sum have size at least 2 since their span has size at least 2). Indeed, by \cref{lem:cluster-tail-improved},
\[ \sum_{\Gamma \in \cC_{k,m} : \|\Gamma\|\ge3} |\omega(\Gamma)| = O(d^42^d \lambda^3) \cdot \begin{cases}
 \alpha_2^d \alpha_1^d &\text{if }k=2\\
 \alpha_3^d &\text{if }k\ge3
\end{cases} .\]

It remains to do the exact computation regarding clusters $\Gamma \in \cC_{k,m}$ having $\|\Gamma\|=|\spn(\Gamma)|=2$. There are $\binom k2$ ways to choose the span of $\Gamma$. However, there is some lack of symmetry between the choices. There are $m(k-m)$ choices in which the chosen coordinates of $[k]$ are associated to different sides of the hypercube, and there are $\binom m2 + \binom{k-m}2$ choices in which the chosen coordinates are associated to the same side of the hypercube. Thus,
\[ \sum_{\substack{\Gamma \in \cC_{k,m},~\|\Gamma\|=2,\\|\spn(\Gamma)|=2}} \omega(\Gamma) = \left(\tbinom m2 + \tbinom{k-m}2\right) A_{\text{same}} + m(k-m) A_{\text{diff}} ,\]
where
\begin{equation}\label{eq:A-same-diff-def}
A_{\text{same}} := \sum_{\substack{\Gamma \in \cC_{2,0},~\|\Gamma\|=2,\\\spn(\Gamma)=\{1,2\}}} \omega(\Gamma) \qquad\text{and}\qquad A_{\text{diff}} := \sum_{\substack{\Gamma \in \cC_{2,1},~\|\Gamma\|=2,\\\spn(\Gamma)=\{1,2\}}} \omega(\Gamma) .
\end{equation}

Let us first compute $A_{\text{same}}$. Consider a cluster $\Gamma \in \cC_{2,0}$ such that $\|\Gamma\|=2$ and $\spn(\Gamma)=\{1,2\}$. There are two different types of such clusters:
\begin{itemize}
 \item $\Gamma$ consists of a single polymer of size 2 which spans $\{1,2\}$ and whose support is a singleton.
 \item $\Gamma$ consists of two polymers of size 1, one of which spans $\{1\}$ and the other $\{2\}$.
\end{itemize}
There are $2^{d-1}$ clusters of the first type, each having weight $\lambda^2 \alpha_2^d$ (recall the computation of scenario~III in \cref{sec:computational-examples} and that the Ursell function of a vertex is 1).
There are $2 \cdot 2^{d-1}$ clusters of the second type, each having weight $-\frac12 \lambda^2 \alpha_1^{2d}$
(recall that the Ursell function of an edge is $-\frac12$).
Thus,
\begin{equation}\label{eq:A-same-formula}
 A_{\text{same}} = 2^{d-1} \lambda^2 (\alpha_2^d - \alpha_1^{2d}) .
\end{equation}

Let us now compute $A_{\text{diff}}$.
Consider a cluster $\Gamma \in \cC_{2,1}$ such that $\|\Gamma\|=2$ and $\spn(\Gamma)=\{1,2\}$. There are two different types of clusters:
\begin{itemize}
 \item $\Gamma$ consists of a single polymer of size 2 which spans $\{1,2\}$ and whose support has size 2.
 \item $\Gamma$ consists of two polymers of size 1, one of which spans $\{1\}$ and the other $\{2\}$.
\end{itemize}
There are $2^{d-1} d$ clusters of the first type, each having weight $\lambda^2 \alpha_1^{2d-2} \cdot \frac{1+2\lambda e^{-\beta} + \lambda^2 e^{-\beta}}{(1+\lambda)^2}$ (recall the computation of scenario~IV in \cref{sec:computational-examples}). There are $2^d d$ clusters of the second type, each having weight $-\frac12 \lambda^2 \alpha_1^{2d}$.
Thus,
\begin{equation}\label{eq:A-diff-formula}
 A_{\text{diff}} = 2^{d-1} d \lambda^4 \alpha_1^{2d} \cdot \frac{e^{-\beta}(1-e^{-\beta})}{(1+\lambda e^{-\beta})^2} .
\end{equation}
This completes the proof of the theorem.
\end{proof}

We now turn to computing the normalized central moments of $Z$. These are the moments of
\[ X := \frac{Z}{\E Z}-1 .\]
The first step toward proving \cref{thm:central-moments} is to establish an asymptotic formula for the $k$-th moment of $X$ in terms of the cluster expansion of the $k$-system. This can be formulated directly via certain sequences of clusters, but we find it more convenient here to work instead with sequences of sets which indicates the spans of these clusters.

Fix $k \ge 2$.
For $D \subset [k]$, we write $\cC_D$ for $\cC_\cD$ where $\cD \in \{\cE,\cO\}^k$ satisfies $\cD_i=\cE$ for $i \in D$ and $\cD_i=\cO$ for $i \in [k] \setminus D$.
For $S \subset [k]$, define
\[ \omega(S;D) := \sum_{\Gamma \in \cC_D : \spn(\Gamma)=S} \omega(\Gamma) .\]
For $S_1,\dots,S_n \subset [k]$, we also define $\omega(S_1,\dots,S_n;D):=\frac1{n!} \prod_{i=1}^n \omega(S_i;D)$.

\begin{lemma}\label{lem:central-moment-cluster-formula}
For any fixed $k \ge 2$,
\[ \E X^k \simeq 2^{-k} \sum_{D \subset [k]} \sum_{\substack{(S_1,\dots,S_n):\\|S_1|,\dots,|S_n| \ge 2\\S_1\cup\cdots\cup S_n=[k]}} \omega(S_1,\dots,S_n;D) ,\]
where the second sum is absolutely convergent.
\end{lemma}

The lemma roughly says that the $k$-th moment of $X$ is effectively governed by sequences of clusters (of span size at least 2) which together span all of $[k]$, and in this sense do not appear (jointly) in any proper subsystem of the $k$-system.



\begin{proof}
Recalling~\eqref{eq:moment-ratio-cluster-expansion}, we have
\[ \E(1+X)^k \simeq 2^{-k} \sum_{D \subset [k]} \exp\left(\sum_{S \subset [k] : |S|>1} \omega(S;D) \right) .\]
Using the binomial expansion, we get
\[ \E X^k = \sum_{\ell=0}^k \binom k\ell (-1)^{k-\ell} \E(1+X)^\ell \simeq \sum_{T \subset [k]} (-1)^{k-|T|} 2^{-|T|} \sum_{D \subset T} \exp\left(\sum_{S \subset T, |S|>1} \omega(S;D) \right) .\]
Expanding the exponential via its Taylor series, we may write it as a sum of $\frac1{n!}\omega(S_1;D)\cdots\omega(S_n;D)$ over all sequences $(S_1,\dots,S_n)$ with $n \ge 0$ and each $S_i$ a subset of $T$ of size greater than 1. Thus,
\[ \E X^k \simeq \sum_{T \subset [k]} (-1)^{k-|T|} 2^{-|T|} \sum_{D \subset T} \sum_{\substack{S_1,\dots,S_n \subset T:\\|S_1|,\dots,|S_n| \ge 2}} \omega(S_1,\dots,S_n;D) .\]
Since the sums are absolutely convergent, we may now change the order of summation.
Fix a sequence $(S_1,\dots,S_n)$ appearing in the last sum and denote $T_0 := S_1\cup\cdots\cup S_n$.
Observe that $\omega(S_1,\dots,S_n;D)=\omega(S_1,\dots,S_n;D')$ whenever $D$ and $D'$ are such that $D \cap T_0 = D' \cap T_0$. We gather these terms together in the sum and identify them with a canonical representative $\omega(S_1,\dots,S_n;D_0)$ with $D_0 \subset T_0$. Then the effective coefficient of $\omega(S_1,\dots,S_n;D_0)$ is
\[ \sum_{\substack{D \subset T \subset [k]\\T_0 \subset T,~ D_0 \subset D \subset D_0 \cup (T\setminus T_0)}} (-1)^{k-|T|} 2^{-|T|} = \sum_{T_0 \subset T \subset [k]} (-1)^{k-|T|} 2^{-|T_0|} = \begin{cases} 2^{-k} &\text{if }T_0=[k] \\0&\text{otherwise} \end{cases} .\]
Thus, only sequences with $S_1\cup\cdots\cup S_n=[k]$ remain, and their coefficient is $2^{-k}$.
\end{proof}


Before proving \cref{thm:central-moments}, we collect some facts we will require.
\begin{lemma}\label{lem:alpha-bounds}
Suppose that $\lambda = \Theta(1)$ and fix $a \ge 1$.
\begin{enumerate}[(i)]
 \item\label{lem:alpha-bound-a1} If $p \gg \frac{\log d}d$, then $\alpha_1^d d^a \ll 1$.
 \item\label{lem:alpha-bound-a0} If $p=\Omega(1)$, then $\alpha_1^{2d} d(1-p) = O(\alpha_2^d)$.
 \item\label{lem:alpha-bound-f1} If $p \ge \frac{(1+\lambda)^2}{2\lambda(2+\lambda)}+\omega(\frac1d)$, then $(2\alpha_k)^d \ll 1$ for any fixed $k \ge 2$.
 \item\label{lem:alpha-bound-sigma} If $p \ge \frac{(1+\lambda)^2}{2\lambda(2+\lambda)}+\omega(\frac1d)$, then $\sigma \ll 1$.
 \item\label{lem:alpha-bound-a2} If $\frac{\log d}d \ll p \le 1 - (1+\lambda-\Omega(1))^{-d}$, then $(2\alpha_1\alpha_2)^d d^a \ll \sigma^2$.
 \item\label{lem:alpha-bound-p-near-1} If $p \to 1$ and $p \le 1 - d^{\omega(1)} 2^{-d/3}$, then $(2\alpha_k)^d d^a \ll \sigma^k$ for any fixed $k \ge 3$.
 \item\label{lem:alpha-bound-f} If $\Omega(1) \le p \le 1 - d^{\omega(1)} 2^{-d/3}$, then $(2\alpha_m)^{d/m} d^a \ll (2\alpha_k)^{d/k} \vee \sigma$ for any fixed $k>m \ge 3$.
 \item\label{lem:alpha-bound-cl} If $\Omega(1) \le p \le 1 - d^{\omega(1)} 2^{-d/3}$, then $(2\alpha_{m_1}\cdots\alpha_{m_t})^d d^a \ll (2\alpha_k)^d \vee \sigma^k$ for any fixed $k \ge 3$, $t \ge 2$ and $m_1,\dots,m_t \ge 1$ such that $m_1+\cdots+m_t = k$.
\end{enumerate}
\end{lemma}
\begin{proof}
\emph{\eqref{lem:alpha-bound-a1}.}
It suffices to show that $\alpha_1 \le 1 - \omega(\frac{\log d}d)$. Since $\alpha_1 = 1 - \frac{p\lambda}{1+\lambda}$, this is immediate.

\smallskip\noindent
\emph{\eqref{lem:alpha-bound-a0}.} It suffices to show that $1-\frac{\alpha_1^2}{\alpha_2} = \Omega(1-p)$. Plugging in the definitions of $\alpha_1$ and $\alpha_2$, this is easily verified.

\smallskip\noindent
\emph{\eqref{lem:alpha-bound-f1}.} Since $\alpha_k$ is decreasing in $k$, it suffices to show that $(2\alpha_2)^d \ll 1$. Equivalently, $2\alpha_2 \le 1 - \omega(\frac1d)$. Plugging in the definition of $\alpha_2$, we see that this is the same as $2p \ge \frac{(1+\lambda)^2}{\lambda(2+\lambda)}(1+\omega(\frac1d))$.

\smallskip\noindent
\emph{\eqref{lem:alpha-bound-sigma}.}
Using only that $\lambda=\Theta(1)$, one checks that
\begin{equation}\label{eq:sigma-expr}
\sigma^2 = (2\alpha_2)^d \cdot \Omega(1 \wedge p(1-p)d) \qquad\text{and}\qquad \sigma^2 = (2\alpha_2)^d \cdot O\left(1 + \tfrac{\alpha_1^{2d} d(1-p)}{\alpha_2^d}\right).
\end{equation}
The claim now follows from~\eqref{lem:alpha-bound-f1} and~\eqref{lem:alpha-bound-a0}.

\smallskip\noindent
\emph{\eqref{lem:alpha-bound-a2}.}
By~\eqref{eq:sigma-expr}, it suffices to show that $\alpha_1^d d^a \ll 1 \wedge (1-p)d$. For $p$ bounded away from 1, this follows from~\eqref{lem:alpha-bound-a1}. When $p \to 1$, this follows using that $\alpha_1 = \frac{1+o(1)}{1+\lambda}$ and the upper bound on $p$.

\smallskip\noindent
\emph{\eqref{lem:alpha-bound-p-near-1}.}
Denote $\gamma := \frac{(2\alpha_k)^{1/k}}{(2\alpha_2)^{1/2}}$.
By~\eqref{eq:sigma-expr} and noting that $1-p \le 1 \wedge p(1-p)d$, it suffices to show that $\gamma^{2d} d^{2a/k} \ll 1-p$.
Observe that $p \to 1$ implies that $\gamma \to 2^{1/k-1/2}$. For $k>3$, this is strictly less than $2^{-1/6}$ and the claim follows easily using that $1-p=\Omega(2^{-d/3})$. For $k=3$, the claim similarly follows when $1-p \ge (2-\Omega(1))^{-d/3}$. When $1-p$ is smaller than this, we argue that $\gamma \le 2^{-1/6}(1+O(1-p))$, so that $\gamma^{2d} = O(2^{-d/3})$ and the claim follows since $1-p = \omega(d^{2a/k} 2^{-d/3})$ by assumption.
To show that $\gamma \le 2^{-1/6}(1+O(1-p))$, first note that it is equivalent to $\frac{(\alpha_3)^{1/3}}{(\alpha_2)^{1/2}} \le 1+O(1-p)$. Now observe that $(1+\lambda)(\alpha_\ell)^{1/\ell} = 1+ h(\ell) (1-p)(1+o(1))$, where $h(x) := \frac{(1+\lambda)^x-1}{x}$. Since $h(2)$, $h(3)$ and $h(3)-h(2)$ are all $\Theta(1)$ (note that $h$ is strictly increasing and continuous in both $x$ and $\lambda$), we have that $\frac{(\alpha_3)^{1/3}}{(\alpha_2)^{1/2}} = 1 + (h(3)-h(2))(1-p)(1+o(1)) = 1 + \Theta(1-p)$.

\smallskip\noindent
\emph{\eqref{lem:alpha-bound-f}.}
The claim follows from~\eqref{lem:alpha-bound-p-near-1} when $p \to 1$, and so we may assume that $p$ is bounded away from~1.
Let $g \colon (0,1)\times(0,\infty) \to [0,\infty)$ be the function $(p,\lambda) \mapsto \frac{(2\alpha_m)^{1/m}}{(2\alpha_2)^{1/2} \vee (2\alpha_k)^{1/k}}$.
By~\eqref{eq:sigma-expr}, it suffices to show that $g(p,\lambda)^d \ll 1$.
Note that $g$ is continuous and that $g<1$ by \cref{cl:f-inc-dec}. Thus, $g(p,\lambda)$ is bounded away from 1 (since $p$ is bounded away from 0 and 1), and the claim follows.

\smallskip\noindent
\emph{\eqref{lem:alpha-bound-cl}.}
Using \cref{cl:super-multiplicative}, we see that $\alpha_{m_1} \cdots \alpha_{m_t} < \alpha_k$. The claim thus follows from~\eqref{lem:alpha-bound-p-near-1} when $p \to 1$, and so we may assume that $p$ is bounded away from 1. 
Let $g \colon (0,1)\times(0,\infty) \to [0,\infty)$ be the function $(p,\lambda) \mapsto \alpha_{m_1} \cdots \alpha_{m_t} / \alpha_k$. Note that $g$ is continuous and that $g<1$. Thus, $g(p,\lambda)$ is bounded away from 1, and the claim follows.
\end{proof}

We are now ready to prove \cref{thm:central-moments}.
The idea behind the proof is that the main contribution to the second sum  in \cref{lem:central-moment-cluster-formula} is from one of two cases: either there is a single set $S_1$ (which must equal $[k]$), or there are $k/2$ sets $S_1,\dots,S_n$ all of size 2. The clusters contributing non-negligibly in the former case are of size $k$ and in the latter case of size 2. We will first upper bound the absolute contribution from all other cases. We then compute the contribution from these two main cases.

\begin{proof}[Proof of \cref{thm:central-moments}]
Fix $k \ge 3$. We begin by proving the first part of the theorem, and thus suppose that $\frac{(1+\lambda)^2}{2\lambda(2+\lambda)} + \omega(\frac1d) \le p \le 1 - 2^{-d/3 + \omega(\log d)}$.

Our starting point is \cref{lem:central-moment-cluster-formula} which says that
\[ \E X^k \simeq 2^{-k} \sum_{D \subset [k]} \sum_{(S_1,\dots,S_n) \in \cS} \omega(S_1,\dots,S_n;D) ,\]
where $\cS$ is the collection of all sequences $(S_1,\dots,S_n)$ of subsets of $[k]$ of size at least 2 satisfying that $S_1 \cup \cdots \cup S_n = [k]$. Let $\cS^* \subset \cS$ be the collection of those sequences which partition $[k]$ into sets of size 2.
We will first compute the contribution from these sequences, showing that
\begin{equation}\label{eq:partition-sum}
2^{-k} \sum_{(S_1,\dots,S_n) \in \cS^*} \sum_{D \subset [k]} \omega(S_1,\dots,S_n;D) = \begin{cases} (1+o(1)) \cdot \sigma^k (k-1)!! &\text{if $k$ is even}\\0 &\text{if $k$ is odd} \end{cases}.
\end{equation}
We will next compute the contribution from the single sequence $([k])$, showing that
\begin{equation}\label{eq:cm0}
2^{-k} \sum_{D \subset [k]} \omega([k];D) = (\tfrac \lambda2)^k2^d\alpha_k^d + o(\sigma^k \vee 2^d \alpha_k^d) .
\end{equation}
We will then show that the contribution from all other sequences is negligible:
\begin{equation}\label{eq:cm-negligible}
2^{-k} \sum_{D \subset [k]} \sum_{\substack{(S_1,\dots,S_n) \in \cS \setminus \cS^*\\n \ge 2}} \omega(S_1,\dots,S_n;D) = o(\sigma^k \vee 2^d \alpha_k^d) .
\end{equation}
This will yield the first part of the theorem.

%

Let us begin with~\eqref{eq:partition-sum}.
When $k$ is odd, $\cS^*$ is empty, and \eqref{eq:partition-sum} is immediate. Suppose that $k$ is even and set $n := k/2$.
There are $|\cS^*|=n! \cdot (k-1)!!$ many ordered partitions of $[k]$ into sets of size~2.
Since any such partition contributes the same, and since its contribution factorizes over the $n$ pairs, to obtain~\eqref{eq:partition-sum}, it suffices to show that
\[ \sum_{D \subset \{1,2\}} \omega(\{1,2\};D) = (1+o(1)) 4\sigma^2 .\]
Thus, we are left with a computation concerning clusters of the $2$-system. Specifically, the sum of weights of clusters which span $\{1,2\}$. Since clusters of size larger than 2 are negligible in comparison to $\sigma^2$ by \cref{lem:cluster-tail-improved} and \cref{lem:alpha-bounds}\eqref{lem:alpha-bound-a2}, it suffices to show that
\[ \sum_{D \subset \{1,2\}} \sum_{\substack{\Gamma \in \cC_D: \|\Gamma\|=2\\\spn(\Gamma)=\{1,2\}}} \omega(\Gamma) = 4\sigma^2 .\]
In fact, we have already done this computation. Indeed, the left-hand side is $2A_{\text{same}}+2A_{\text{diff}}$, where $A_{\text{same}}$ and $A_{\text{diff}}$ were defined in~\eqref{eq:A-same-diff-def}. These were subsequently computed in~\eqref{eq:A-same-formula} and~\eqref{eq:A-diff-formula}, from which we see that $2A_{\text{same}}+2A_{\text{diff}}=4\sigma^2$. We note for later use that both $A_{\text{same}}$ and $A_{\text{diff}}$ are non-negative, so that $|A_{\text{same}}|,|A_{\text{diff}}| \le 2\sigma^2$.
This establishes~\eqref{eq:partition-sum}.

We now move on to showing~\eqref{eq:cm0}.
Recall that $\omega([k];D) = \sum_{\Gamma \in \cC_D : \spn(\Gamma)=[k]} \omega(\Gamma)$.
The total contribution to $\omega([k];D)$ from clusters of size larger than $k$ is negligible in comparison to $2^d\alpha_k^d$ by \cref{lem:cluster-tail-improved} and \cref{lem:alpha-bounds}\eqref{lem:alpha-bound-a1}. Consider a cluster $\Gamma$ with $\spn(\Gamma)=[k]$ and $\|\Gamma\|=k$.
Suppose first that $D=\emptyset$ or $D=[k]$ (all defects on the same side). Then the support of $\Gamma$ must be a singleton.
Recall from scenario~III in \cref{sec:computational-examples} that the weight of a polymer $\gamma$ with singleton support and $|\spn(\gamma)|=\ell$ is $\lambda^\ell \alpha_\ell^d$.
There are numerous types of clusters with singleton support, according to the number of polymers and their sizes. There are $2^{d-1}$ clusters of weight $\lambda^k\alpha_k^d$ (those which consist of a single polymer), and all other clusters have combined weight $o(\sigma^k \vee 2^d \alpha_k^d)$ by \cref{lem:alpha-bounds}\eqref{lem:alpha-bound-cl}. Thus,
\[ \omega([k];D) = 2^{d-1}\lambda^k\alpha_k^d + o(\sigma^k \vee 2^d \alpha_k^d) .\]
Suppose now that $D \neq \emptyset,[k]$. In this case, the support of $\Gamma$ need not be a singleton, but it is easy to see that there are at most $2^d d^{O(1)}$ ways to choose the support. It follows that there are at most $2^d d^{O(1)}$ ways to choose the cluster $\Gamma$. \cref{lem:weight-of-polymer} implies that any polymer $\gamma$ with $|\spn(\gamma)|=\ell$ has weight at most $\lambda^{\|\gamma\|} (\tilde\alpha_\ell)^{\|N(\gamma)\|}$, where $\tilde\alpha_\ell = (\alpha_\ell)^{1/\ell}$ was defined at the beginning of \cref{sec:convergence}. Since $\|N(\gamma)\| \ge d\ell$, the weight of $\gamma$ is at most $\lambda^{\|\gamma\|} \alpha_\ell^d$. Finally, \cref{lem:alpha-bounds}\eqref{lem:alpha-bound-cl} yields that $\omega([k];D) = o(\sigma^k \vee 2^d \alpha_k^d)$.
This establishes~\eqref{eq:cm0}.

It remains to show~\eqref{eq:cm-negligible}.
To this end, we may fix $D$ and show that
\[ \sum_{\substack{(S_1,\dots,S_n) \in \cS \setminus \cS^*\\n \ge 2}} \omega(S_1,\dots,S_n;D) = o(\sigma^k \vee 2^d\alpha_k^d) .\]
Fix $(S_1,\dots,S_n) \in \cS \setminus \cS^*$ with $n \ge 2$ and denote $\ell_i := |S_i|$.
By \cref{lem:cluster-tail-improved},
\[ |\omega(S_1,\dots,S_n;D)| \le \tfrac{O(1)^n}{n!} 2^{dn} \left(\alpha_{\ell_1}\cdots\alpha_{\ell_n}\right)^d .\]
Define $f(\ell) := (2\alpha_\ell)^{1/\ell}$. Using \cref{cor:alpha-product-bound}, we obtain that
\[ |\omega(S_1,\dots,S_n;D)| \le \tfrac{O(1)^n}{n!} (f(2)\vee f(k))^{(\ell_1+\cdots+\ell_n)d} .\]
This bound will not suffice for us, and we need to tweak it by separating those $\ell_i$ which are 2 or $k$ from the rest. Write $n=t+r+m$, where $t$ and $r$ are the number of $\ell_i$ which are $2$ and $k$, respectively, and let $3 \le \ell'_1,\dots,\ell'_m \le k-1$ be the remaining elements and denote $\ell := \ell'_1+\dots+\ell'_m$. Using that $|\omega(S_i;D)| \le 2\sigma^2$ for those $S_i$ of size 2 (since $\omega(S_i;D)$ equals either $A_{\text{same}}$ or $A_{\text{diff}}$), using \cref{lem:cluster-tail-improved} for larger $S_i$, and applying \cref{cor:alpha-product-bound} to $\ell'_1,\dots,\ell'_m$, we obtain that
\begin{align*}
|\omega(S_1,\dots,S_n;D)| 
 &\le \tfrac{O(1)^n}{n!} \sigma^{2t} (2\alpha_k)^{rd} (2^m \alpha_{\ell'_1}\cdots\alpha_{\ell'_m})^d \\
 &\le \tfrac{O(1)^n}{n!} \sigma^{2t} f(k)^{krd} (f(3)\vee f(k-1))^{\ell d} .
\end{align*}
Since for a given $n$, there are at most $2^{kn}$ sequences $(S_1,\dots,S_n)$, in order to obtain the desired bound, it suffices to show that, uniformly in $(t,r,\ell)$,
\[ \star := \sigma^{2t} f(k)^{krd} (f(3)\vee f(k-1))^{\ell d} = o(\sigma^k \vee f(k)^{kd}) .\]
Since $\star$ only decreases when increasing $t$, $r$ or $\ell$, it suffices to show this for any fixed $(t,r,\ell)$.
By \cref{lem:alpha-bounds}\eqref{lem:alpha-bound-f1}-\eqref{lem:alpha-bound-sigma}, each of $\sigma, f(2)^d, f(3)^d, \dots, f(k)^d$ is $o(1)$, and we assume that $d$ is large enough so that they are all at most 1.
If $r \ge 2$, then $\star \le f(k)^{2kd} \ll f(k)^{kd}$. If $r=1$, then either $t \ge 1$ or $\ell \ge 1$ (since $n \ge 2$), and again we have that $\star \ll f(k)^{kd}$. Now suppose that $r=0$. If $\ell=0$, then $t>k/2$ (since $(S_1,\dots,S_n) \notin \cS^*$) so that $\star \le \sigma^{k+1} \ll \sigma^k$. If $\ell \ge 1$ (note that $2t+\ell \ge k$), then $\star \le (\sigma \vee f(k)^d)^{2t} (f(3) \vee f(k-1))^{\ell d} \ll (\sigma \vee f(k)^d)^{2t+\ell} \le (\sigma \vee f(k)^d)^k$ by \cref{lem:alpha-bounds}\eqref{lem:alpha-bound-f}. This establishes~\eqref{eq:cm-negligible}, and thus completes the proof of the first part of the theorem.


\medskip

To see the claim regarding $p=\frac{(1+\lambda)^2}{2\lambda(2+\lambda)} \pm O(\frac1d)$, we first note that in this regime, we have that $\sigma=\Theta(1)$ and $f(k)^{kd} \ll 1$.
In particular, \cref{lem:alpha-bounds}\eqref{lem:alpha-bound-f1}-\eqref{lem:alpha-bound-sigma} hold for this $p$ if one replaces ``$\ll 1$'' with ``$=O(1)$''. Consequently, the bounds in \eqref{eq:partition-sum} and \eqref{eq:cm0} remain unchanged, while the bound in \eqref{eq:cm-negligible} becomes $O(\sigma^k) \pm O(2^d\alpha_k^d) = O(1)$ for $k$ even (and still 0 for $k$ odd). To see the latter, note that $\star$ is at most $O(f(k)^{kd}) + o(\sigma^k) = o(1)$ when $r+\ell \ge 1$; when $r=\ell=0$ and $t>k/2$, we use that $0\le \omega(S_i;D) \le 2\sigma^2$ to obtain that $\omega(S_1,\dots,S_n;D) = \frac{O(1)^n}{n!}$. Putting these together yields that $\E X^k = \Theta(\sigma^k) \pm O(2^d\alpha_k^d) = \Theta(1)$ for $k$ even. For $k$ odd, this gives that $\E X^k = O(1)$, and to see that it is $\Theta(1)$, it suffices to note that $\omega(S_1,\dots,S_n;D)$ is $\Theta(1)$ for some sets $S_1,\dots,S_n$ of size $2$ and some $D$.
This yields the statement of the theorem in this case.

To see the claim regarding $p \ge \frac{(1+\lambda)^2}{2\lambda(2+\lambda)}$, we note that \cref{lem:alpha-bounds}\eqref{lem:alpha-bound-a2}-\eqref{lem:alpha-bound-f} remain true without the upper bound on $p$ if one replaces $\sigma$ with $f(2)^d=(2\alpha_2)^{d/2}$.
The bounds in \eqref{eq:partition-sum}, \eqref{eq:cm0} and \eqref{eq:cm-negligible} then all become $O((2\alpha_2)^{kd/2} \vee 2^d\alpha_k^d)$.
This yields the statement of the theorem in this case.
\end{proof}

\begin{proof}[Proof of \cref{cor:normal-general}]
\cref{lem:alpha-bounds}\eqref{lem:alpha-bound-p-near-1} implies that $\sigma^k \gg (2\alpha_k)^d$ for any $k \ge 3$. \cref{thm:central-moments} now yields that
\[ \E\left(\frac{Z-\E Z}{\sigma \E Z}\right)^k \to \begin{cases} (k-1)!! &\text{if $k$ is even}\\0 &\text{if $k$ is odd} \end{cases}. \]
Since the right-hand side equals $\E N^k$, and since $\Var(Z) = (1+o(1)) (\sigma \E Z)^2$ by \cref{thm:moments-general} (note that $\sigma=o(1)$ by \cref{lem:alpha-bounds}\eqref{lem:alpha-bound-sigma}), we conclude the first part of the corollary. Since normal random variables are determined by their moments, we also conclude that $(Z-\E Z)/\sqrt{\Var(Z)}$ converges in distribution to $N$.
\end{proof}

Let us now explain how the theorems stated at the beginning of this section imply the theorems stated in the introduction.

\begin{proof}[Proof of \cref{thm:expectation}]
Immediate from \cref{thm:expectation-general} by plugging in $\lambda=1$ and $\alpha_1=1-\frac p2$.
\end{proof}

\begin{proof}[Proof of \cref{thm:normal}]
Immediate from \cref{cor:normal-general} by plugging in $\lambda=1$.
\end{proof}

\begin{proof}[Proof of \cref{thm:variance}]
Immediate from \cref{thm:moments-general} by plugging in $\lambda=1$, $\alpha_1=1-\frac p2$ and $\alpha_2=1-\frac{3p}4$.
\end{proof}

\begin{proof}[Proof of \cref{thm:moments}]
The first part follows from \cref{thm:moments-general} by plugging in $\lambda=1$, $\alpha_1=1-\frac p2$ and $\alpha_2=1-\frac{3p}4$, and using that $d^3 \epsilon_k^d \ll \alpha_1^{2d}$, which is simple to check using only that $p \gg \frac{\log d}{d}$.

For the second part, suppose first that $k \ge 4$ is even and $\frac23 + \omega(\frac1d) \le p \le 1 - \omega(\frac1d)$. Using the upper bound on $p$, it is not hard to check that $\alpha_1^{2d} \ll \alpha_2^d$ and $(1-p)d\alpha_1^{2d} \ll \alpha_2^d$. Thus, $\sigma = \frac12 2^{d-1} \alpha_1^d (1+o(1))$ and the formula stated in the theorem follows from \cref{thm:central-moments}. Suppose now that $k \ge 4$ and $p = \frac23 \pm O(1)$. In this case, $(2-\frac{3p}2)^{kd/2} = \Theta(1)$ and $2^d(1-p+p2^{-k})^d \ll 1$, and the claimed result follows from \cref{thm:central-moments}. The remaining two cases also follow from \cref{thm:central-moments} since $(2-\frac{3p}2)^{kd/2} \le 1$ and $2^d(1-p+p2^{-k})^d \ll 1$ when $k \ge 3$ and $p \ge \frac23$.
\end{proof}

\section{Discussion and open questions}\label{sec:open}

We have introduced several results on the partition function of the hard-core model (and, in particular, on the number of independent sets) in a random subgraph of the hypercube. 
For a wide range of $p$ we have found precise asymptotics for the expected value and higher moments. 
For values of $p$ tending to 1, we further established a normal limiting distribution result, whereas for $p>\frac23$, we have a concentration result which yields estimates on the partition function which hold with high probability.

Our work raises several natural questions. 
Our results can be interpreted as results about the partition function of a family of positive-temperature models on the hypercube (recall \cref{prop:relation}).
Interestingly, our results allow to extract information about the structure of a random configuration chosen from the positive-temperature model on the hypercube. Indeed, we have established a convergent cluster expansion representation (see \cref{thm:ClusterEx}) from which it is rather standard to deduce such structural information. For example (for appropriate parameters of the model), one may deduce from our results that the probability that two given vertices, one even and one odd, belong to such a configuration is at most $C\lambda \left(\frac{1+\lambda e^{-\beta}}{1+\lambda}\right)^d$. For $p>\frac23$ (say, constant), it should also be possible to deduce such probabilistic information about the typical independent sets (i.e., the hard-core model with $\lambda=1$) in the random graph $Q_{d,p}$, using that the relevant quantities are concentrated in this regime (as demonstrated by \cref{cor:fluctuation-bound}), but we have not pursued this here. For other $\lambda$, the value $\frac23$ becomes $\frac{(1+\lambda)^2}{2\lambda(2+\lambda)}$ (see \cref{sec:moments}). It would be interesting to study the structure of a typical configuration in the hardcore model on $Q_{d,p}$ for smaller values of $p$.
For non-random graphs, such results were established on the hypercube~\cite{kahn2001entropy,galvin2011threshold,jenssen2020independent} and also on the closely related $\Z^d$ lattice~\cite{galvin2004phase,peled2014odd} (and for positive temperature in~\cite{peled2020long}).

We have seen that for $p>\frac 23+\omega\left(\frac 1d\right)$, $i(Q_{d,p})$ is concentrated around its mean, in the sense that with high probability  $i(Q_{d,p})=(1+o(1))\E(i(Q_{d,p}))$. On the other hand, 
we have seen that in $p=\frac 23$ it is non-concentrated (around its mean or otherwise). This raises the question of whether or not there is concentration for smaller values of $p$, for example for $p=\frac 12$ (recall~\eqref{eq:Expectation1/2} and~\eqref{eq:Variance1/2}). More generally, it is interesting to determine the typical order of magnitude of $i(Q_{d,p})$, specifically if it is close to its mean either in the sense of $\Theta(\E(i(Q_{d,p})))$ or as in the sense of the right hand side of~\eqref{eq:expectation-simple}.

Another possible approach to studying the hard-core model on $Q_{d,p}$ (different than the approach taken in this paper) is to try to apply the machinery of~\cite{sapozhenko1987onthen,galvin2011threshold,jenssen2020independent} directly to the random graph. As Galvin noted in~\cite{galvin2011threshold}, this machinery relies on only few properties of the hypercube, specifically the fact that it is a regular bipartite graph with certain isoperimetric bounds. However, $Q_{d,p}$ is typically very non-regular, and perhaps more crucially, while certain isoperimetric bounds for $Q_{d,p}$ are known (see~\cite{erde2021expansion} and references therein), these do not seem to be suited for the problem at hand.

We have seen that for $p$ tending to 1 (not too fast), there is a normal limiting behavior for $i(Q_{d,p})$. It is natural to ask what the limiting behavior is for other values of $p$, e.g., for constant $p \in (0,1)$.
\cref{thm:moments} implies that (see the discussion after the theorem) the central moments of $i(Q_{d,p})$ do \emph{not} behave asymptotically like those of a normal random variable. While this does not necessarily preclude the possibility of a normal limit, it might suggest a different limiting distribution, perhaps log-normal.

Some of our results are in the regime where $p$ is at least $\frac{C\log d}{d^{1/3}}$. It is natural to wonder how small $p$ can be for these results to hold. For example, when does the expected number of independent sets behave as in~\eqref{eq:expectation-simple}? It is not hard to see that this fails for $p=o(\frac1d)$ (or even for $p \le \frac cd$ for small $c>0$). Indeed, by considering subsets of one bipartition class of the hypercube and the isolated vertices in the other bipartition class, one sees that for $p=o(\frac 1d)$, with high probability, $i(Q_{d,p})=2^{2^{d-1}(2-o(1))}$, and thus also in expectation. Similarly, for $p=O(\frac1d)$, with high probability, $i(Q_{d,p}) = 2^{2^{d-1}(1+\Omega(1))}$.
This further shows that
\begin{equation}\label{eq:ExpectationHigh}
\E i(Q_{d,p}) = 2^{2^{d-1}(1+o(1))}
\end{equation} 
does not hold for $p=O(\frac1d)$, whereas~\eqref{eq:expectation-simple} implies that it does hold for $p \ge \frac{C\log d}{d^{1/3}}$.
In fact, it is not too hard to show that this weaker form of~\eqref{eq:expectation-simple} holds for $p=\omega(\frac{\log d}d)$. This can be seen by using~\cref{lem:weighted-shearer0} (or alternatively~\cite[Theorem~1.3]{galvin2012bounding}) to obtain that $\E i(Q_{d,p}) \le Z(K_{d,d},1,\beta)^{\frac1d 2^{d-1}}$, and then following the proof of~\cref{lem:Z-Psi-bound-gen} (and noting that $Z(\{\bar0\})=2^d$) to deduce that $Z(K_{d,d},1,\beta) \le 2^{d(1+o(1))}$.
In this context, we mention that $\tfrac 1d$ is the threshold for the appearance of a giant component in $Q_{d,p}$ (see, e.g., \cite{van2017hypercube}), and that there is a positive proportional of vertices of bounded degree when $p=O(\frac1d)$ but not when $p=\omega(\frac1d)$.

Lastly, we note that independent sets can be seen as a special case of graph homomorphisms. In this direction, results about graph homomorphisms were established in~\cite{engbers2012h,jenssen2020homomorphisms,peled2020long} and for the special case of $q$-colorings in~\cite{kahn2020number,peled2018rigidity}.
It is likely that the techniques in this paper, together with those of previous works, can be extended to tackle the problem of counting more general homomorphisms in $Q_{d,p}$. In this context, we note that the relation given in \cref{prop:relation} easily extends to any (weighted) homomorphism model.

\bigskip
\noindent\textbf{Acknowledgments.}
This work was carried out alongside two \emph{independent} and \emph{hypercute} babies, Nur and Barr, to which we are grateful for inspiring us and helping us stay awake at night to think about this problem. This paper is dedicated to them.
	Research of GK was supported by the European Union’s Horizon 2020 research and innovation programme under the Marie Sk\l odowska Curie grant agreement No.\ 101030925.
	Research of YS was supported in part by NSERC of Canada.

\appendix
\section{}
Fix $p \in (0,1)$ and $c>0$.
For $x>0$, define
\[ \alpha_x := 1-p(1-e^{-cx}) \qquad\text{and}\qquad f(x) := (2\alpha_x)^{1/x} .\]

\begin{claim}\label{cl:super-multiplicative}
The function $x \mapsto \log \alpha_x$ is strictly convex. In particular, $\alpha_x \alpha_y > \alpha_{x-t}\alpha_{y+t}$ for $y \ge x \ge t>0$.
\end{claim}
\begin{proof}
Define $h(x) := \alpha_x$ and $g(x) := \log h(x)$ for $x>0$. Let us show that $g''>0$. Indeed, $h>0$ and $g''h^2=h''h-(h')^2 = (1-p)pc^2e^{-cx} \ge 0$.
\end{proof}

\begin{cor}\label{cor:alphaIncreasing}
	The function $x \mapsto \alpha_x^{1/x}$ is increasing.
\end{cor}
\begin{proof}
	Since $g(x)=\log \alpha_x$ is convex and $g(0)=0$, we have that $g'(x)\geq \frac {g(x)}{x}$ and hence $(\frac 1x g(x))' \ge 0$. Thus, $\frac 1x g(x)$ is increasing and, in particular, so is $e^{\frac 1x g(x)}=\alpha_x^{1/x}$.
\end{proof}

\begin{claim}\label{cl:f-inc-dec}
If $p \le \frac12$, then $f$ is strictly decreasing on $(0,\infty)$.
If $p>\frac12$, then there exists $x_* \in (0,\infty)$ such that $f$ is strictly decreasing on $(0,x_*]$ and strictly increasing on $[x_*,\infty)$.
\end{claim}
\begin{proof}
We prove both parts simultaneously, taking $x_*=\infty$ when $p \le \frac12$.
It suffices to show that $f'$ vanishes at a unique point $x_*$ and that it is negative on $(0,x_*)$ and positive on $(x_*,\infty)$. The derivative of $f$ is
\[ f'(x) = f(x) \left(- \frac{\log(2\alpha_x)}{x^2} + \frac{\alpha'_x}{x\alpha_x}\right) = \frac{f(x)g(x)}{x^2 \alpha_x} ,\]
where $g(x) := x\alpha'_x - \alpha_x \log(2\alpha_x)$.
Since $f(x)$ and $\alpha_x$ are positive, it suffices to show that $g$ vanishes at a unique point $x_*$ and that it is negative on $(0,x_*)$ and positive on $(x_*,\infty)$. In fact, $g(0)=-\log2<0$ and $g(\infty)=-(1-p)\log(2-2p)$ is positive exactly when $p>\frac12$, and we claim that $g$ is strictly increasing on $(0,\infty)$. Indeed, its derivative is
\[ g'(x) = \alpha'_x + x\alpha''_x - \alpha'_x\log(2\alpha_x) - \alpha'_x = -\alpha'_x (cx + \log(2\alpha_x)) .\]
Since $\alpha'_x$ is negative, it suffices to show that $cx + \log(2\alpha_x)$ is positive for all $x>0$, or equivalently, that $\alpha_x>\frac12e^{-cx}$. This is straightforward to verify.
\end{proof}

\begin{cor}\label{cor:alpha-product-bound}
For any $n \ge 1$ and $x' \ge x_1,\dots,x_n \ge x>0$, we have
\[ \prod_{i=1}^n (2\alpha_{x_i}) \le \max\{ f(x),f(x') \}^{x_1+\cdots+x_n} .\]
\end{cor}
\begin{proof}
This is equivalent to
$\prod_{i=1}^n f(x_i)^{x_i/(x_1+\cdots+x_n)} \le \max\{ f(x),f(x') \}$.
The left-hand side is at most $\max\{f(x_1),\dots,f(x_n)\}$, which is at most the right-hand side by \cref{cl:f-inc-dec}.
\end{proof}

\bibliographystyle{amsplain}
\bibliography{hardcore}

\providecommand{\bysame}{\leavevmode\hbox to3em{\hrulefill}\thinspace}
\providecommand{\MR}{\relax\ifhmode\unskip\space\fi MR }
\providecommand{\MRhref}[2]{%
  \href{http://www.ams.org/mathscinet-getitem?mr=#1}{#2}
}
\providecommand{\href}[2]{#2}
\begin{thebibliography}{10}

\bibitem{ajtai1982largest}
Mikl{\'o}s Ajtai, J{\'a}nos Koml{\'o}s, and Endre Szemer{\'e}di, \emph{Largest
  random component of a k-cube}, Combinatorica \textbf{2} (1982), no.~1, 1--7.

\bibitem{bollobas1983evolution}
B{\'e}la Bollob{\'a}s, \emph{The evolution of the cube}, North-Holland
  Mathematics Studies, vol.~75, Elsevier, 1983, pp.~91--97.

\bibitem{bollobas1990complete}
\bysame, \emph{Complete matchings in random subgraphs of the cube}, Random
  Structures \& Algorithms \textbf{1} (1990), no.~1, 95--104.

\bibitem{bollobas1994diameter}
B{\'e}la Bollob{\'a}s, Yoshiharu Kohayakawa, and T~{\L}uczak, \emph{On the
  diameter and radius of randon subgraphs of the cube}, Random Structures \&
  Algorithms \textbf{5} (1994), no.~5, 627--648.

\bibitem{bollobas1992evolution}
B{\'e}la Bollob{\'a}s, Yoshiharu Kohayakawa, and Tomasz {\L}uczak, \emph{The
  evolution of random subgraphs of the cube}, Random Structures \& Algorithms
  \textbf{3} (1992), no.~1, 55--90.

\bibitem{borgs2006random}
Christian Borgs, Jennifer~T Chayes, Remco Van~der Hofstad, Gordon Slade, and
  Joel Spencer, \emph{Random subgraphs of finite graphs: Iii. the phase
  transition for the n-cube}, Combinatorica \textbf{26} (2006), no.~4,
  395--410.

\bibitem{brydges1984short}
David~C Brydges, \emph{A short course on cluster expansions}, Les Houches
  (1984), no.~PART I.

\bibitem{burtin1977probability}
Ju~D Burtin, \emph{The probability of connectedness of a random subgraph of an
  $n$-dimensional cube}, Problemy Peredachi Informatsii \textbf{13} (1977),
  90--95.

\bibitem{condon2021hamiltonicity}
Padraig Condon, Alberto~Espuny D{\'\i}az, Ant{\'o}nio Girao, Daniela K{\"u}hn,
  and Deryk Osthus, \emph{Hamiltonicity of random subgraphs of the hypercube},
  Proceedings of the 2021 ACM-SIAM Symposium on Discrete Algorithms (SODA),
  SIAM, 2021, pp.~889--898.

\bibitem{engbers2012h}
John Engbers and David Galvin, \emph{H-coloring tori}, Journal of Combinatorial
  Theory, Series B \textbf{102} (2012), no.~5, 1110--1133.

\bibitem{erde2021expansion}
Joshua Erde, Mihyun Kang, and Michael Krivelevich, \emph{Expansion in
  supercritical random subgraphs of the hypercube and its consequences}, arXiv
  preprint arXiv:2111.06752 (2021).

\bibitem{erdos1979evolution}
Paul Erd{\"o}s and Joel Spencer, \emph{Evolution of the $n$-cube}, Computers \&
  Mathematics with Applications \textbf{5} (1979), no.~1, 33--39.

\bibitem{galvin2011threshold}
David Galvin, \emph{A threshold phenomenon for random independent sets in the
  discrete hypercube}, Combinatorics, Probability and Computing \textbf{20}
  (2011), 27--51.

\bibitem{galvin2012bounding}
\bysame, \emph{Bounding the partition function of spin-systems}, arXiv preprint
  arXiv:1206.3200 (2012).

\bibitem{galvin2019independent}
\bysame, \emph{Independent sets in the discrete hypercube}, arXiv preprint
  arXiv:1901.01991 (2019).

\bibitem{galvin2004phase}
David Galvin and Jeff Kahn, \emph{On phase transition in the hard-core model on
  {${\bf Z}^d$}}, Combinatorics, Probability and Computing \textbf{13} (2004),
  no.~2, 137--164.

\bibitem{galvin2004weighted}
David Galvin and Prasad Tetali, \emph{On weighted graph homomorphisms}, DIMACS
  Series in Discrete Mathematics and Theoretical Computer Science \textbf{63}
  (2004), 97--104.

\bibitem{hulshof2020slightly}
Tim Hulshof and Asaf Nachmias, \emph{Slightly subcritical hypercube
  percolation}, Random Structures \& Algorithms \textbf{56} (2020), no.~2,
  557--593.

\bibitem{janson2018critical}
Svante Janson and Lutz Warnke, \emph{On the critical probability in
  percolation}, Electronic Journal of Probability \textbf{23} (2018), 1--25.

\bibitem{jenssen2020homomorphisms}
Matthew Jenssen and Peter Keevash, \emph{Homomorphisms from the torus}, arXiv
  preprint arXiv:2009.08315 (2020).

\bibitem{jenssen2020independent}
Matthew Jenssen and Will Perkins, \emph{Independent sets in the hypercube
  revisited}, Journal of the London Mathematical Society \textbf{102} (2020),
  no.~2, 645--669.

\bibitem{kahn2001entropy}
Jeff Kahn, \emph{An entropy approach to the hard-core model on bipartite
  graphs}, Combinatorics, Probability and Computing \textbf{10} (2001), no.~3,
  219--237.

\bibitem{kahn2020number}
Jeff Kahn and Jinyoung Park, \emph{The number of 4-colorings of the {H}amming
  cube}, Israel Journal of Mathematics \textbf{236} (2020), no.~2, 629--649.

\bibitem{Korshunov1983Th}
Aleksej~D. Korshunov and Alexander~A. Sapozhenko, \emph{The number of binary
  codes with distance 2}, Problemy Kibernet (Russian) \textbf{40} (1983),
  no.~1, 111--130.

\bibitem{kostochka1993radius}
Alexandr~V. Kostochka, Alexander~A. Sapozhenko, and K~Weber, \emph{Radius and
  diameter of random subgraphs of the hypercube}, Random Structures \&
  Algorithms \textbf{4} (1993), no.~2, 215--229.

\bibitem{kotecky1986cluster}
Roman Koteck{\`y} and David Preiss, \emph{Cluster expansion for abstract
  polymer models}, Communications in Mathematical Physics \textbf{103} (1986),
  no.~3, 491--498.

\bibitem{kulich2012diameter}
Tom{\'a}{\v{s}} Kulich, \emph{The diameter of a random subgraph of the
  hypercube}, Random Structures \& Algorithms \textbf{41} (2012), no.~2,
  282--291.

\bibitem{mcdiarmid2021component}
Colin McDiarmid, Alex Scott, and Paul Withers, \emph{The component structure of
  dense random subgraphs of the hypercube}, Random Structures \& Algorithms
  (2021).

\bibitem{peled2014odd}
Ron Peled and Wojciech Samotij, \emph{Odd cutsets and the hard-core model on
  {${\bf Z}^d$}}, Annales de l'IHP Probabilit{\'e}s et statistiques, vol.~50,
  2014, pp.~975--998.

\bibitem{peled2018rigidity}
Ron Peled and Yinon Spinka, \emph{Rigidity of proper colorings of {${\bf
  Z}^d$}}, arXiv preprint arXiv:1808.03597 (2018).

\bibitem{peled2020long}
\bysame, \emph{Long-range order in discrete spin systems}, arXiv preprint
  arXiv:2010.03177 (2020).

\bibitem{samotij2015counting}
Wojciech Samotij, \emph{Counting independent sets in graphs}, European Journal
  of Combinatorics \textbf{48} (2015), 5--18.

\bibitem{sapozhenko1987onthen}
Alexander.~A. Sapozhenko, \emph{On the number of connected subsets with given
  cardinality of the boundary in bipartite graphs}, Metody Diskretnogo Analiza
  (Russian) \textbf{45} (1987).

\bibitem{scott2005repulsive}
Alexander~D Scott and Alan~D Sokal, \emph{The repulsive lattice gas, the
  independent-set polynomial, and the {L}ov{\'a}sz local lemma}, Journal of
  Statistical Physics \textbf{118} (2005), no.~5, 1151--1261.

\bibitem{soshnikov2003largest}
Alexander Soshnikov and Benny Sudakov, \emph{On the largest eigenvalue of a
  random subgraph of the hypercube}, Communications in mathematical physics
  \textbf{239} (2003), no.~1, 53--63.

\bibitem{van2017hypercube}
Remco van~der Hofstad and Asaf Nachmias, \emph{Hypercube percolation}, Journal
  of the European Mathematical Society \textbf{19} (2017), no.~3, 725--814.

\end{thebibliography}

\end{document}